\documentclass[
    11pt,
    fleqn,
    final
    ]{article}

\usepackage[utf8]{inputenc}
\usepackage[T1]{fontenc}
\usepackage{etoolbox}
\usepackage{amsmath}
\usepackage{MnSymbol}
\usepackage{amsfonts}
\usepackage{mathtools}
\usepackage{xpatch}
\usepackage{amsthm}
\usepackage{ifdraft}
\usepackage[only, 
            llparenthesis, rrparenthesis,
            llbracket, rrbracket
            ]{stmaryrd}
\usepackage{enumitem}
\usepackage{mathrsfs}
\usepackage{hyperref}
\usepackage{xpatch}
\usepackage[atend]{bookmark}

\usepackage[numeric, msc-links, nobysame]{amsrefs}
\xpatchcmd{\bibsection}{*}{}{}{}

\usepackage{calc}
\usepackage{array}
\usepackage[raggedright]{titlesec}
\usepackage{titletoc}
\usepackage{etoolbox}
\usepackage{tikz}
\usepackage[skins, breakable, hooks]{tcolorbox}
\usetikzlibrary{calc}
\usetikzlibrary{cd}
 
\usepackage{geometry}

\makeatletter
\AtBeginDocument{%
  \let\[\@undefined
  \DeclareRobustCommand{\[}{\begin{equation}}%
  \let\]\@undefined
  \DeclareRobustCommand{\]}{\end{equation}}%
}
\makeatother
\mathtoolsset{showonlyrefs,showmanualtags}

\hypersetup{
  colorlinks, 
  linkcolor=-red!75!green!50, 
  citecolor=-red!75!green!50,
  urlcolor=green!30!black
}

\usepackage{showlabels}

\newtheorem{Theorem}{Theorem}
\newtheorem{PropositionIntro}[Theorem]{Proposition}

\newtheorem{Proposition}{Proposition}
\newtheorem{Lemma}[Proposition]{Lemma}
\newtheorem{Corollary}[Proposition]{Corollary}

\theoremstyle{definition}
\newtheorem{Example}[Proposition]{Example}

\AtBeginEnvironment{Example}{\pushQED{\qed}}
\AtEndEnvironment{Example}{\popQED}

\tcbset{common/.style={
  breakable,
  enhanced jigsaw,
  before skip=10pt,
  colback=#1,
  grow to right by=0.25cm,
  grow to left by=0.25cm,
  if odd page={right=0.25cm,left=0.25cm}{right=0.25cm,left=0.25cm},
  toggle enlargement,
  boxsep=0pt,
  toprule=0.07mm,
  bottomrule=0.07mm,
  leftrule=0mm,
  rightrule=0mm,
  arc=0mm,
  }
}
\newlength{\normalparindent}
\AtBeginDocument{\setlength{\normalparindent}{\parindent}}
\tcbset{before upper={\setlength{\parindent}{\normalparindent}}} 

\colorlet{thm-color}{blue!5}
\colorlet{proof-color}{gray!8}
\colorlet{question-color}{magenta!4}
\definecolor{example-color}{HTML}{d0efb1}

\foreach \t in {Proposition, PropositionIntro, Theorem, Lemma, Corollary}
  {
    \expandafter\tcolorboxenvironment\expandafter{\t}{
        common={thm-color},
    }
  }
\tcolorboxenvironment{proof}{common={proof-color}}
\tcolorboxenvironment{Example}{common={example-color!25}}

\newlist{thmlist}{enumerate}{1}
\setlist[thmlist]{
        nolistsep,
        topsep=1ex,
        itemsep=0.5ex,
        ref={\mdseries\theProposition\textup{(\emph{\roman*})}},
        label={\mdseries\textup{(\emph{\roman*})}},
        listparindent=\parindent,
        }
\newcommand\thmitem[1]{\textup{(\emph{\romannumeral #1})}}

\newlist{tfae}{enumerate}{1}
\setlist[tfae]{
        nolistsep,
        ref={\mdseries\textup{(\emph{\alph*})}},
        label={\mdseries\textup{(\emph{\alph*})}},
        }
\makeatletter
\newcommand\tfaeitem[1]{{\textup{(\emph{\@alph #1})}}}

\makeatother

\setlist[itemize]{itemsep=0.5ex, topsep=0.5ex, listparindent=1em, parsep=0pt}
\setlist[itemize, 1]{label=\textbullet}
\setlist[itemize, 2]{label=$\blacktriangleright$}
\setlist[itemize, 3]{label=--}
\setlist[enumerate]{itemsep=0.5ex, topsep=0.5ex, listparindent=1em, parsep=0pt}

\titlecontents{section}
    [0pc] 
    {\footnotesize\filright} 
    {\hspace{-1.25pc}\makebox[1.25pc]{\thecontentslabel.\hfil}} 
    {\hspace*{0pc}} 
    {\titlerule*[0.25pc]{.}\contentspage} 
\titlecontents{subsection}
    [1.25pc] 
    {\footnotesize} 
    {\hspace*{-1.25pc}\makebox[1.25pc]{\thecontentslabel.\hfil}\space} 
    {\hspace*{2pc}} 
    {\titlerule*[0.25pc]{.}\contentspage} 
\titlecontents{subsubsection}
    [1.5pc+2.75pc] 
    {\footnotesize} 
    {\hspace*{-2.5pc}\makebox[2.5pc]{\thecontentslabel.\hfil}} 
    {\hspace*{2pc}} 
    {\titlerule*[0.25pc]{.}\contentspage} 
\makeatletter
\renewcommand\tableofcontents{%
    \begingroup
    \centering
    {\large\bfseries\contentsname}\par
    \vspace*{0.5pc}%
    \begin{minipage}{.8\textwidth}
    \@starttoc{toc}%
    \end{minipage}\par
    \endgroup
    \vspace*{2pc}
    }
\makeatother

\BookmarkAtEnd{%
 \bookmarksetup{startatroot}%
 \bookmark[named=LastPage, level=0]{Last page}%
}
\AtBeginDocument{%
  \bookmark[level=1, named=FirstPage]{Title}
}

\newcommand\claim[2][.8]{%
  \begin{minipage}[c]{#1\displaywidth}%
  \itshape
  #2
  \end{minipage}%
}

\DeclareSymbolFont{AMSb}{U}{msb}{m}{n}
\DeclareMathSymbol{\Bbbk}{\mathord}{AMSb}{"7C}

\makeatletter
\DeclareRobustCommand{\mod}[1]{\allowbreak\if@display\mkern10mu
  \else\mkern12mu\fi{\operator@font mod}\,\,#1}
\makeatother

\DeclareMathOperator{\Hom}      {Hom}

\DeclareMathOperator{\GL}       {GL}
\DeclareMathOperator{\Aut}      {Aut}
\DeclareMathOperator{\Inn}      {Inn}
\DeclareMathOperator{\Out}      {Out}
\DeclareMathOperator{\Pic}      {Pic}
\DeclareMathOperator{\Ext}      {Ext}
\DeclareMathOperator{\Tor}      {Tor}
\DeclareMathOperator{\Der}      {Der}
\DeclareMathOperator{\InnDer}   {InnDer}
\DeclareMathOperator{\OutDer}   {OutDer}
\DeclareMathOperator{\tr}       {tr}
\DeclareMathOperator{\trd}      {trd}
\DeclareMathOperator{\jac}      {jac}
\let\div\relax
\DeclareMathOperator{\div}      {div}
\DeclareMathOperator{\M}        {M}

\DeclareMathOperator{\res}      {res}
\DeclareMathOperator{\Alt}      {Alt}
\DeclareMathOperator{\CE}       {CE}
\DeclareMathOperator{\rad}      {rad}

\DeclareMathOperator{\ad}       {ad}
\DeclareMathOperator\Z          {Z}
\let\H\relax
\DeclareMathOperator\H          {H}
\DeclareMathOperator\HH         {HH}

\newcommand\inter[1]{\llbracket#1\rrbracket}

\DeclarePairedDelimiter\abs{\lvert}{\rvert}
\DeclarePairedDelimiter\gen{\langle}{\rangle}
\DeclarePairedDelimiter\ggen{\llangle}{\rrangle}
\DeclarePairedDelimiter\parens{(}{)}

\let\epsilon\varepsilon

\newcommand\id{\mathord{\mathsf{id}}}

\newcommand\D[1]{\operatorname{D}\parens{#1}}

\newcommand\NN{\mathbb{N}}

\newcommand\ZZ{\mathbb{Z}}
\newcommand\CC{\mathbb{C}}
\newcommand\FF{\mathbb{F}}

\newcommand\kk{\Bbbk}

\newcommand\dR{\mathsf{dR}}
\newcommand\CB{\mathsf{B}}
\newcommand\B{\mathcal{B}}

\newcommand\F{\mathscr{F}}
\newcommand\T{\mathsf{T}} 

\newcommand\place{\mathord-}

\renewcommand\d{\mathrm{d}}

\newcommand\op{\mathsf{op}}
\newcommand\odd{\mathsf{od}}
\newcommand\even{\mathsf{ev}}
\newcommand\JAC{\mathfrak{Jac}}
\newcommand\DIV{\mathfrak{Div}}
\newcommand\m{\mathfrak{m}}
\newcommand\Ring{\mathsf{Ring}}

\newcommand\crs{\mathbin{\#}} 
\newcommand\fe{\mathrel/} 

\colorlet{newterm-color}{blue!50!black}
\newcommand\newterm[2][]{\textbf{\itshape\color{newterm-color}#2\/}}

\title{The action of the Nakayama automorphism\\of a Frobenius algebra on
Hochschild cohomology}
\author{Mariano Suárez-Álvarez\thanks{Universidad de Buenos Aires,
IMAS/CONICET and Guangdong Technion Institute of Technology.\\
E-mail: \texttt{mariano@dm.uba.ar}}}
\date{December 2024\ifoptionfinal{}{; compiled \today}}

\begin{document}

\maketitle

\tableofcontents

\section{Introduction}
\label{sect:intro}

Ferdinand Georg Frobenius presented in his paper~\cite{Frobenius} of~1903,
among many other things, a characterization of those finite-dimensional
algebras over a field whose left and right regular representations are
equivalent. We call them \newterm{Frobenius algebras} --- thankfully, the
term \emph{Frobeniusean} that Tadasi Nakayama favored did not catch on.
Their study was soon after taken on by Richard Brauer and Cecil Nesbitt
in~\cite{BN} and by Nakayama in~\citelist{\cite{Nakayama:1}
\cite{Nakayama:2}}, who were the first to realize that the understanding of
such algebras is a key step in the general program of understanding
non-semisimple algebras and their representations, and after them by many,
many others, and to this day. The continued interest in them resides in
that, on one hand, many algebras of central importance in representation theory are
Frobenius algebras --- semisimple algebras, group algebras and their
blocks, the Iwahori--Hecke algebras of finite Coxeter groups,
finite-dimensional Hopf algebras, cohomology algebras of compact
oriented manifolds, Yoneda algebras of Calabi--Yau algebras, many diagram
algebras,~\&c~--- and, on the other, in that they show up in useful ways in
very different contexts, such as low-dimensional topology, topological
quantum field theory, combinatorics, number theory, algebraic geometry, and
quantum information theory.

One of the modern characterizations of these algebras, the one that we will
use here as a definition, and that is one of the cornerstones of the
theory developed by Brauer, Nesbitt and Nakayama, is the following: a
finite-dimensional algebra~$A$ over a field~$\kk$ is a Frobenius algebra
exactly when there is a non-degenerate bilinear form
$\gen{\place,\place}:A\times A\to\kk$ on~$A$ that is \emph{associative}, in
the sense that for all choices of~$x$,~$y$ and~$z$ in~$A$ we have
  \[
  \gen{xy,z} = \gen{x,yz}.
  \]
When that is the case we will fix one such bilinear form --- of
which there are in general many --- and view it as an extra structure
on~$A$. 

If we can choose the bilinear form so that it is symmetric we say that the
algebra itself is \newterm{symmetric}. This is usually not the case,
though. A momentous observation that one has to make at this point is that,
despite that, there is always a unique algebra automorphism $\sigma:A\to A$
for which we have
  \[
  \gen{x,y} = \gen{y,\sigma(x)}
  \]
for all~$x$ and all~$y$ in~$A$. This is the \newterm{Nakayama automorphism}
of~$A$ with respect to that bilinear form. Of course, if the form is
symmetric, then we have that $\sigma=\id_A$, the identity map of~$A$.
It is important that --- since $A$ admits, in general, many non-degenerate
and associative bilinear forms --- the Nakayama automorphism~$\sigma$ is not
determined by~$A$ alone, but that it turns out that the class of~$\sigma$ in the
outer automorphism group~$\Out(A)$ of~$A$ is.

\bigskip

A Frobenius algebra~$A$ is thus canonically endowed with a distinguished
symmetry. Now, it is a truth universally acknowledged that symmetries are
important, and canonical ones even more so, so it is not surprising that
the Nakayama automorphism plays a significant and special role in almost
every aspect of the study of a Frobenius algebra. The main result of this
paper is the following manifestation of this:

\begin{Theorem}
The Nakayama automorphism $\sigma:A\to A$ of a Frobenius algebra~$A$ acts
trivially on the Hochschild cohomology~$\HH^*(A)$ of~$A$.
\end{Theorem}

In general, the automorphism group of an algebra~$A$ acts in a canonical
way on the Hochschild cohomology of~$A$, and this action is usually highly
non-trivial. The claim of the theorem is that when the algebra is Frobenius
the Nakayama automorphism acts on~$\HH^*(A)$ as the identity. We can see
this as saying that taking the Hochschild cohomology automagically takes into
account the symmetry implied by the Nakayama automorphism, so that this
automorphism does not do anything on the result. It should be noted that
the corresponding statement about Hochschild \emph{homology} is not true,
so there is certainly something non-obvious happening here.

The $0$th Hochschild cohomology space~$\HH^0(A)$ of an algebra~$A$ can be
identified canonically with the center~$\Z(A)$ of~$A$, and the canonical
action of an automorphism of~$A$ on~$\HH^0(A)$ corresponds under this
identification to the obvious action of the automorphism on~$\Z(A)$. The
theorem therefore says, in particular, that the Nakayama automorphism of a
Frobenius algebra fixes the center of the algebra pointwise. This statement
was proved by Masaru Osima in his paper~\cite{Osima:1} of~1951, and the
theorem above is a «higher» version this.

\bigskip

Our main theorem says that something is trivial, and usually such a result
is the end of a discussion, for we cannot do much with trivial things. If
we recall how we construct the Hochschild cohomology of an algebra, though,
then we can view the claim of the theorem as an \emph{existence} result. Indeed,
the theorem says that every class~$\alpha$ in~$\HH^*(A)$ coincides with the
class $\sigma(\alpha)$ that we obtain from it by applying the Nakayama
automorphism, and if we have a concrete realization for~$\alpha$ --- as the
cohomology class of a cocycle in some cochain complex, for example, ---
then we can generally construct a realization of the same nature
for~$\sigma(\alpha)$, and the fact that the two classes are equal means
that something exists that attests that equality in terms of those
realizations --- in that example, a
cochain whose coboundary is the difference of the cocycles
representing~$\alpha$ and~$\sigma(\alpha)$. After proving our main result,
we spend the second half of the paper putting this idea in practice, and
show how to use it to obtain invariants for Frobenius algebras. Let us
finish this introduction explaining what we find.

\bigskip

Let us fix a Frobenius algebra~$A$ with bilinear
form~$\gen{\place,\place}:A\times A\to\kk$ and Nakayama
automorphism~$\sigma:A\to A$. In Subsection~\ref{subsect:jac} we prove the
following:

\begin{PropositionIntro}
If $u:A\to A$ is an endomorphism of~$A$,
then there is exactly one element~$\jac_\sigma(u)$ in~$A$, which we call
the \newterm{Nakayama Jacobian} of~$u$, such that
  \[
  \gen{u(a),u(b)} = \gen{\jac_\sigma(u)\cdot a,b}
        \qquad\text{for all $a$,~$b\in A$.}
  \]
\end{PropositionIntro}

\noindent This provides us with an invariant attached to each endomorphism
of~$A$. One motivation for naming it as we have is that this makes the
following version of the \emph{Jacobian Conjecture} of Eduard Ott-Heinrich
Keller \cite{Keller:cremona} true.

\begin{PropositionIntro}\label{prop:jacobian-conjecture}
An endomorphism $u:A\to A$ of~$A$ is an automorphism of $A$ if and only if
its Nakayama Jacobian $\jac_\sigma(u)$ is a unit in~$A$.
\end{PropositionIntro}

The easy proof of this is based on the following observation:
  \[
  \claim[.95]{If $u$ and~$v$ are two automorphisms of~$A$, then
  $\jac_\sigma(u\circ v) = \jac_\sigma(v)\cdot v^{-1}(\jac_\sigma(u))$.
  }
  \]
We can view the identity in this claim as a slightly twisted variant
of the chain rule of calculus for Jacobians, which is nice, and also as an
equally slightly twisted cocycle condition. Untwisting it, we can construct
an invariant of cohomological nature, now attached to the algebra~$A$
itself:

\begin{PropositionIntro}
The function
  \[
  \jac^\sim_\sigma : u\in\Aut(A) \mapsto \jac_\sigma(u^{-1})\in A^\times
  \]
is a $1$-cocycle on the group~$\Aut(A)$ with values in the possibly
non-abelian group of units~$A^\times$ whose class~$\JAC(A)$
in~$\H^1(\Aut(A),A^\times)$ depends only on~$A$ and not on the choice of
the non-degenerate and associative bilinear form~$\gen{\place,\place}$ used
to compute~it.
\end{PropositionIntro}

We do not know what exactly this cohomology class means --- the fact that
the non-abelian cohomology $\H^1(\Aut(A),A^\times)$ is somewhat hard to
compute in any non-trivial example does not help here. It is
not difficult to establish the following, though:

\begin{PropositionIntro}
The algebra~$A$ is symmetric if and only if the restriction of the
class~$\JAC(A)$ from~$\H^1(\Aut(A),A^\times)$ to~$\H^1(\Inn(A),A^\times)$ is
trivial, and when that is the case that class is the image under a certain
canonical `inflation' map
  \[
  \inf : \H^1(\Out(A),A^\times\cap\Z(A)) \to \H^1(\Aut(A),A^\times)
  \]
of exactly one element~$\underline\JAC(A)$ of~$\H^1(\Out(A),A^\times\cap\Z(A))$.
\end{PropositionIntro}

\noindent This tells us that the restriction of the class~$\JAC(A)$ to the
inner automorphism group $\Inn(A)$ of~$A$ is precisely the obstruction to the
symmetry of~$A$. One should note that neither of the groups~$\Aut(A)$
nor~$\Out(A)$ is even a Morita invariant of~$A$. In contrast to this, the connected
component~$\Out_0(A)$ of the identity element in the algebraic group
$\Out(A)$ is Morita invariant by a result of Richard Brauer --- this is R.
David Pollack's \cite{Pollack}*{Theorem 2.1} --- and even derived invariant
according to a theorem of Birge Huisgen-Zimmermann and Manuel Saorín
\cite{HZS} and of Raphaël Rouquier \cite{Rouquier}, so restricting the
class~$\underline\JAC(A)$ to~$\Out_0(A)$ may provide something more
directly interpretable in terms of representation theory. Similarly, it is
natural to ask if that class~$\underline\JAC(A)$ can be extended all the
way to the Picard group~$\Pic(A)$.

In any case, all this provides us with invariants of automorphisms of
Frobenius algebras and of those algebras themselves. We spend some time --- 
Subsubsections~\ref{subsubsect:jac:grass} to~\ref{subsubsect:jac:separable}
--- presenting some examples that show, on one hand, that these invariants
are not trivial and, on the other, that they are interesting, known and
useful is certain situations.

\begin{itemize}[
        leftmargin=1pc
        ]

\item\textbf{\itshape Grassmann algebras.}
Let $V$ be a finite-dimensional vector space, let $\B=(x_1,\dots,x_n)$ be
an ordered basis for~$V$, and let~$A\coloneqq\bigwedge^* V$ be the exterior
algebra on~$V$. It is easy to see that~$A$ is a Frobenius algebra, and even
a symmetric one when the dimension~$n$ of~$V$ is odd. It is also naturally
a $\ZZ$-graded algebra, and we may thus consider the group~$\Aut_\odd(A)$
of automorphisms $u:A\to A$ of~$A$ that map~$V$ into the subspace~$A^\odd$
spanned by the homogeneous elements of odd degree. In~\cite{Bavula},
Vladimir Bavula defines for each automorphism $u:A\to A$ belonging
to~$\Aut_\odd(A)$ a Jacobian,
  \[
  \jac(u) \coloneqq 
      \det
      \left(
      \frac{\partial u(x_i)}{\partial x_j}
      \right) \in A.
  \]
Here $\frac{\partial}{\partial x_1}$,~\dots,~$\frac{\partial}{\partial
x_n}:A\to A$ are certain left \emph{skew} derivations of~$A$ associated
to the ordered basis~$\B$, and the determinant makes sense because the
entries of the matrix whose determinant we are computing belong to the
center of~$A$ --- this is precisely why Bavula does this only for
automorphism in~$\Aut_\odd(A)$, in fact. We prove in
Subsubsection~\ref{subsubsect:jac:grass} that Bavula's Jacobians
essentially coincide with ours:

\begin{PropositionIntro}
For every automorphism~$u:A\to A$ of the exterior algebra~$A$ that belongs
to~$\Aut_\odd(A)$ we have that $\jac_\sigma(u^{-1})=\jac(u)^{-1}$.
\end{PropositionIntro}

In particular, this tells us that we can use our definition
of~$\jac_\sigma$ to extend Bavula's $\jac$ to the whole automorphism group
of~$A$ in what is probably the best possible way. Moreover, this
identification of what the Nakayama Jacobian is in this case allows us to
easily verify that the class~$\JAC(A)$ in~$\H^1(\Aut(A),A^\times)$ is not
trivial.

\item\textbf{\itshape Trivial extensions.}
Let $B$ be a finite-dimensional algebra, and let $\D B\coloneqq\Hom(B,\kk)$
be the dual space of~$V$ viewed as a $B$-bimodule in the usual way. We can then
consider the so-called trivial extension algebra $A\coloneqq B\oplus\D B$,
which is a symmetric algebra. It is easy to compute the Nakayama Jacobian
of any automorphism of~$A$, and with that information we can show that the
classes~$\JAC(A)$ in~$\H^1(\Aut(A),A^\times)$ and~$\underline\JAC(A)$
in~$\H^1(\Out(A),A^\times\cap\Z(A))$ are most often non-trivial --- for
example, if they are trivial the ground field necessarily has exactly two
elements. We do this in Subsubsection~\ref{subsubsect:jac:trivial}.

In~\cite{Bavula} Bavula spends some time describing what are the possible
values for the Jacobian of an automorphism of an exterior algebra. Here we
do something similar for automorphisms of trivial extensions, and the end result
is interesting:

\begin{PropositionIntro}
An element of~$A$ is the Jacobian of an automorphism of~$A$ if and only
if it is of the form $t+\tau$ with $t$ a central unit of~$B$ and
$\tau:B\to\kk$ an element of~$\D B$ in the image of the transposed map
  \[
  \CB^\T : \HH_1(B)^* \to \HH_0(B)^*
  \]
of the Connes boundary $\CB:\HH_0(B)\to\HH_1(B)$.
\end{PropositionIntro}

\noindent Here we are identifying elements of the dual space~$\HH_0(B)^*$
of the $0$th Hochschild homology space of~$B$ with linear maps $B\to\kk$
that vanish on commutators. This result can be stated in terms of the
non-commutative de Rham cohomology of~$A$ --- as defined, for example, by
Max Karoubi in~\cite{Karoubi} --- and is a version for finite-dimensional
algebras of the theorem of Jürgen Moser~\cite{Moser} that says that two
volume forms on a compact orientable manifold are related by a
self-diffeomorphism exactly when they have the same
integral.

One of the points of considering trivial extensions is that they allow us
to view any algebra~$B$ as a subalgebra of a symmetric algebra, and we can
use this to construct invariants for arbitrary algebras from invariants
for Frobenius algebras. In Subsubsection~\ref{subsubsect:jac:trivial} we
also show how to do this and obtain, for example, for \emph{any}
finite-dimensional algebra~$B$ a class in
  \[
  \H^1(\HH_1(B)^*,B^\times\cap\Z(B)),
  \]
the first group cohomology group of the group $\HH_1(B)^*$ with values in the
group of central units~$B^\times\cap\Z(B)$ of~$B$ --- the group~$\HH_1(B)^*$ that
are are talking about here is the additive group of the dual space of the
first Hochschild cohomology of~$B$. This is not the first place in which
one goes looking for invariants, surely\dots

\item\textbf{\itshape Quantum complete intersections of dimension~$4$.}
Since it was used by Ragnar-Olaf Buchweitz, Edward Green, Dag Madsen, and
\O yvind Solberg in~\cite{BGMS} to answer negatively a famous question
about Hochschild cohomology asked by Dieter Happel  in~\cite{Happel}, the
so-called quantum complete intersection algebra of dimension~$4$ has been
used to illustrate many different phenomena, and we also use it here for
that. We fix a scalar~$q$ in~$\kk\setminus\{0,\pm1\}$ and consider the
algebra~$A$ freely generated by two letters~$x$ and~$y$ subject to the relations
  \[
  x^2 = y^2 = 0, \qquad yx = qxy.
  \]
This is a non-symmetric Frobenius algebra. In
Subsubsection~\ref{subsubsect:jac:quantum} we compute explicitly the
Nakayama Jacobian of each automorphism of this algebra and show that the
cohomology class~$\JAC(A)$ is non-trivial in~$\H^1(\Aut(A),A^\times)$. In
particular, this example is nice in that shows that the Jacobian of an automorphism need
not be central.

\item\textbf{\itshape Separable algebras and group algebras.} 
A strongly separable algebra in the sense of Teruo
Kanzaki~\cite{Kanzaki:strongly} is a Frobenius algebra with respect to the
usual trace form --- that is, its Killing form ---  and we show in
Subsubsection~\ref{subsubsect:jac:separable} that any automorphism of such
an algebra has trivial Nakayama Jacobian with respect to that bilinear form.
Similarly, any separable algebra is a Frobenius algebra with respect now to
its \emph{reduced} trace form, and we show there that, again, any
automorphism of such an algebra has trivial Nakayama Jacobian with respect
to that form. This shows that all automorphism of semisimple and involutory
Hopf algebras have trivial Jacobian with respect to their usual Frobenius
structure given by an integral --- in particular, this includes all
semisimple group algebras.

In contrast with that observation, we show in the same subsubsection that
already for the smallest non-semisimple group algebra, that of the cyclic
group of prime order over a field of that characteristic, the group of
Jacobians of automorphisms is quite non-trivial.

\end{itemize}
We finish our study of Nakayama Jacobians by showing in
Subsubsection~\ref{subsubsect:jac:cross} how they can be used to describe
the Nakayama automorphisms of crossed products of Frobenius algebras by
finite groups and, more generally, of generalized cross products, and in
Subsubsections~\ref{subsubsect:jac:reductions} we make several observations
useful in computing those Jacobians in practice.

\bigskip

All this comes from looking, essentially, at the consequences of the
triviality of the action of the Nakayama automorphism on~$\HH^0(A)$ and
on~$\Out(A)$. In Subsection~\ref{subsect:div} we do the same but this time for
the action on~$\HH^1(A)$. We find the following existence result:

\begin{PropositionIntro}
Let $\delta:A\to A$ be a derivation.  There is a unique
element~$\div_\sigma(\delta)$, which we call the \newterm{Nakayama
divergence} of~$\delta$, such that 
for all $a$ and~$b$ in~$A$ we have 
  \[
  \gen{\delta(a),b} + \gen{a,\delta(b)} = \gen{a,b\cdot\div_\sigma(\delta)}.
  \]
If the bilinear form on~$A$ is symmetric, then that element is central
in~$A$.
\end{PropositionIntro}

The two basic properties of this divergence operator are as follows.

\begin{PropositionIntro}
Let $\delta$,~$\eta:A\to A$ be two derivations of~$A$.
\begin{thmlist}

\item If $z$ is a central element in~$A$, then 
  \(
  \div_\sigma(z\delta) = z\div_\sigma(\delta) - \delta(z)
  \).

\item We have that
  \(
  \div_\sigma([\delta,\eta])
        = \delta(\div_\sigma(\eta)) - \eta(\div_\sigma(\delta))
                + [\div_\sigma(\delta),\div_\sigma(\eta)]
  \).

\end{thmlist}
\end{PropositionIntro}

The first of these claims states that the map $\div_\sigma:\Der(A)\to A$ is a
differential operator of order~$1$, and the second one tells us that --- when
$2\neq0$ in~$\kk$ --- that map is a solution of the Maurer--Cartan
equation in the differential graded Lie algebra $\CE(\Der(A),A)$ that we
get from the Chevalley--Eilenberg complex that computes the Lie algebra
cohomology of the Lie algebra of derivations~$\Der(A)$ with values in the
tautological module~$A$ when we endow it with the Schouten--Nijenhuis
bracket that arises from the associative commutator of~$A$. In particular,
the map~$\div_\sigma$ is, in general, not a cocycle in that complex, but we
do have the following:

\begin{PropositionIntro}
If the bilinear form on~$A$ is symmetric, then the map
$\div_\sigma:\Der(A)\to\Z(A)$ is a $1$-cocycle in the
complex~$\CE(\Der(A),\Z(A))$ that vanishes on the ideal of inner derivations~$\InnDer(A)$, 
and it determines a cohomology class~$\DIV(A)$ in
  \[
  \H^1(\HH^1(A),\Z(A)),
  \]
the Lie algebra cohomology of the Lie algebra~$\HH^1(A)$ of exterior
derivations of~$A$ with values in the center~$\Z(A)$, computed with respect
to the action given by the Gerstenhaber bracket. This class depends only
on~$A$ and not on the choice of the symmetric non-degenerate and
associative bilinear form used to construct it.
\end{PropositionIntro}

After doing this, we go through the same list of examples that we
considered when we were studying Jacobians and look now at the divergences of their
derivations. In Subsubsection~\ref{subsubsect:div:grass} we compute the
divergence of the derivations of an exterior algebra, establishing the pleasing
fact that they are given by a formula similar to that of the divergence of
a vector field in calculus, and using this show that the class~$\DIV(A)$ is
a non-trivial element of~$\H^1(\HH^1(A),\Z(A))$. In
Subsubsection~\ref{subsubsect:div:trivial} we go back to trivial
extensions: using what is essentially the description of the Lie algebra of
outer derivations of such an algebra given in the work~\cite{CMRS} of
Claude Cibils, Eduardo Marcos, María Julia Redondo and Andrea Solotar, we
are able to compute the divergence of all derivations and to show that,
once again, the class~$\DIV(A)$ is not trivial. We finish this exploration
of low-hanging examples with the calculation of the divergences of
derivations of the four-dimension quantum complete intersection in
Subsubsection~\ref{subsubsect:div:quantum}. We could make about divergences
observations entirely similar to those we made in
Subsubsection~\ref{subsubsect:jac:reductions} about Jacobians, but it seems
unnecessary as there is not much to change. On the other hand, in parallel
to what we did in Subsubsection~\ref{subsubsect:jac:cross} with crossed
products with groups there is a similar theory about cross products of
Frobenius algebras with restricted Lie algebras that act on them by
derivations: while this is interesting, it is also long and technical,
so we will deal with it in future work.

We finish the paper, at long last, with an example of a result that shows that
our Nakayama Jacobians and divergences share many non-trivial properties
with the classical operators of calculus. We prove in
Subsection~\ref{subsect:div:liouville} an analogue
of the classical formula of Liouville and Ostrogradski that relates the
Jacobian of the flow generated by a vector field with the divergence of the
latter, of which the following is the most significant part:

\begin{PropositionIntro}
Let us suppose that the ground field~$\kk$ has characteristic zero,
let $A$ be a Frobenius algebra with respect to a bilinear
form~$\gen{\place,\place}$, and let $\sigma:A\to A$ be the corresponding
Nakayama automorphism. If $\delta:A\to A$ is a nilpotent derivation, then
  \[ 
  \frac{\d}{\d t}\biggm\mid_{t=0}\jac_\sigma(\exp(t\delta))
        = \sigma^{-1}(\div_\sigma(\delta)).
  \]
\end{PropositionIntro}

For each $t$ in~$\kk$ here $\exp(t\delta)$ denotes the exponential of the
derivation~$t\delta$: this makes sense because $\delta$ is nilpotent and
$\kk$ has characteristic zero, and turns out to be given by a polynomial
in~$t$. As a consequence of that, $\jac_\sigma(\exp(t\delta))$ is an
$A$-valued polynomial function of~$t$ and, in particular, we can compute
its derivative at~$t=0$. This proposition can be extended to the situation
of an arbitrary derivation~$\delta$, provided the field~$\kk$ has
characteristic zero and is complete with respect to some absolute value.

\bigskip

Let us finish the introduction of this paper by mentioning one thing that we have \emph{not}
done in it. We looked at consequences of the triviality of the
action of the Nakayama automorphism on~$\HH^0(A)$ and on~$\HH^1(A)$, and,
of course, the next step is to consider what happens when we look
at~$\HH^2(A)$. From this invariants for deformations of the algebra~$A$
arise, and doing this was the actual motivation for this whole paper. We
expect to write about this elsewhere.

On the other hand, we have to remark that the theory developed in this
paper is but one of a pair of non-identical twins: we have in preparation a
second paper that deals with twisted Calabi-Yau algebras and their modular
automorphisms --- which are often also called Nakayama automorphisms --- in
which we establish results very similar to those obtained here for them.
Dealing with Calabi--Yau algebras is much more complicated than with
Frobenius algebras, as doing anything with them requires generous doses of
homological algebra, but the results that one obtains, or hopes to obtain,
are somewhat deeper. For example, when we go from Frobenius algebras to
Calabi-Yau algebras what we stated as
Proposition~\ref{prop:jacobian-conjecture} becomes a statement that
includes as a very special case the actual \emph{Jacobian Conjecture} of
Keller, and thus goes from very easy to something we cannot prove.

\bigskip

\noindent\textbf{Notations.} We will have a ground field~$\kk$ fixed
throughout the paper, and when we talk about vector spaces, linear maps,
algebras, tensor products, and so on, without any further qualification, we
have the corresponding objects defined over that field in mind. Similarly,
we write~$\otimes$ and~$\Hom$, without any extra embellishments, to denote
those functors taken over~$\kk$. We write $\D V$ for the dual space of a
vector space~$V$, $f^\T:\D W\to\D V$ for the transpose of a linear map
$f:V\to W$, and $f^{-\T}$ for its inverse, in case it exists.

Let $\Lambda$ be an algebra. If $x$ is an element of~$\Lambda$, then we
write $\ad(x):a\in\Lambda\mapsto xa-ax\in\Lambda$ for the inner derivation
of~$\Lambda$ corresponding to~$x$ and, if $x$ is invertible,
$\iota_x:a\in\Lambda\mapsto xax^{-1}\in\Lambda$ for the inner automorphism
corresponding to~$x$. An automorphism of~$\Lambda$ is \emph{central} if it
fixes the center~$\Z(\Lambda)$ of~$\Lambda$ pointwise, and a
$\Lambda$-bimodule~$M$ is \emph{central} if $zm=mz$ whenever
$z\in\Z(\Lambda)$ and~$m\in M$. If $M$ is a left $\Lambda$-module and
$\alpha:\Lambda\to \Lambda$ is an endomorphism of~$\Lambda$, then we write
${}_\alpha M$ for the left $\Lambda$-module which coincides with~$M$ as a
vector space and on which the action of~$\Lambda$ is such that $a\cdot
m=\alpha(a)m$ for all $a\in \Lambda$ and all~$m\in M$ --- we say ${}_\alpha
M$ is obtained from~$M$ by \emph{twisting}. We use a similar notation for
twists of right modules and of bimodules. Finally, we write
$C_\Lambda^*(M,N)$ for the standard unnormalized Hochschild cochain complex
corresponding to two left $\Lambda$-modules $M$ and~$M$, whose cohomology
is canonically isomorphic to~$\Ext_\Lambda^*(M,N)$, and
$C_{\Lambda^\op}^*(M,N)$ and $C_{\Lambda^e}^*(M,N)$ for the corresponding
constructions for \emph{right} $\Lambda$-modules $M$ and~$N$ and for
$\Lambda$-\emph{bi}modules~$M$ and~$N$, respectively. In particular,
$C_{\Lambda^e}^*(\Lambda,M)$, when $M$ is a $\Lambda$-bimodule, denotes the
standard Hochschild complex whose cohomology is the Hochschild cohomology
of~$\Lambda$ with values in~$M$.

For each pair of real numbers~$s$ and~$t$ we put
$\inter{s,t}\coloneqq\{n\in\ZZ:s\leq n\leq t\}$ and
$\inter{t}\coloneqq\inter{1,t}$.

\section{Frobenius algebras and their Hochschild cohomology}

We will present in this section the proof of the main theorem of this
paper. We use the first subsection to set up the basic notation that we
will use throughout the text, and the second one to describe the action of
Nakayama automorphism on the center and on the Lie algebra of derivations
of a Frobenius algebra --- we take the time do this separately because it
can be done without any cohomological complications, and this has a certain
charm, and allows us to ignore boundary cases in the proof of the general
case. In Subsection~\ref{subsect:action} we recall the way automorphisms
act on Hochschild cohomology, in Subsection~\ref{subsect:hh2} we prove the
triviality of the action of the Nakayama automorphism on the second
Hochschild cohomological space of our algebra so as to illustrate the
argument that we will use to take on the general result, and in
Subsection~\ref{subsect:hhx} we complete the proof of the theorem.
Following that we describe what can be said about the action of the
Nakayama automorphism on Hochschild homology.

\subsection{Frobenius algebras}%
\label{subsect:frobenius}

Let $A$ be a \newterm{Frobenius algebra}, that is, a finite-dimensional
algebra endowed with a fixed choice of a non-degenerate bilinear form
$\gen{\place,\place}:A\times A\to\kk$ that is \newterm{associative}, so
that
  \[
  \gen{ab,c} = \gen{a,bc}
  \]
for all choices of~$a$,~$b$, and~$c$ in~$A$. There is then a unique linear
map $\sigma:A\to A$ such that
  \[ \label{eq:sigma}
  \gen{a,b} = \gen{b,\sigma(a)}
  \]
for all~$a$ and~$b$ in~$A$, and it is an algebra automorphism: it is the
\newterm{Nakayama automorphism} of~$A$ with respect to the bilinear
form~$\gen{\place,\place}$. Similarly, there is a unique linear map
$\beta:A\to\D A$ that for all $a$ and~$b$ in~$A$ has
  \[
  \beta(a)(b) = \gen{a,b},
  \]
it is bijective, and for all $a$ and~$b$ in~$A$ has
  \[
  \beta(ab) = \sigma(a)\beta(b),
  \qquad
  \beta(ab) = \beta(a)b.
  \]
These last two identities tell us that the map $\beta:A\to {}_\sigma\D A$
is a map of $A$-bimodules, and as a consequence of the first of them we
also have that
  \[
  \beta^{-1}(ab) = \sigma^{-1}(a)\beta^{-1}(b)
  \]
for all $a$ and~$b$ in~$A$. In most of what follows we will use these
notations.

\bigskip

The bilinear form~$\gen{\place,\place}$ is not uniquely determined by the
algebra~$A$, except in trivial situations. Let us recall how this is
reflected in the rest of the structure.

\begin{Lemma}\label{lemma:change}
Let $\gen{\place,\place}$,~$\gen{\place,\place}':A\times A\to\kk$ be two
non-degenerate and associative bilinear forms on~$A$, and let
$\sigma$,~$\sigma':A\to A$ be corresponding Nakayama automorphisms.
\begin{thmlist}

\item There is a unit~$t$ in~$A$ such that 
  \(
  \gen{a,b}' = \gen{a,bt}
  \)
for all $a$ and~$b$ in~$A$, and we have that
  \(
  \sigma'(a) = t\sigma(a)t^{-1}
  \)
for all~$a\in A$. 

\item A unit $s$ in~$A$ is such that $\gen{a,b}'=\gen{a,bs}$ for all $a$
and~$b$ in~$A$ if and only if~$st^{-1}$ belongs to the center of~$A$.

\end{thmlist}
\end{Lemma}

\begin{proof}
The two bilinear forms determine two isomorphisms of $A$-bimodules
$\beta:A\to{}_\sigma\D A$ and~$\beta':A\to{}_{\sigma'}\D A$, and the
composition $\beta^{-1}\circ\beta':A\to A$ is an isomorphism of right
$A$-modules. As the map~$\sigma$ is an automorphism of~$A$, there is then a
unique unit~$t$ in~$A$ such that $\beta^{-1}(\beta'(a))=\sigma(t)a$ for
all~$a\in A$. This equality implies that for all $a$ and all~$b$ in~$A$ we
have
  \[
  \gen{a,b}' 
        = \beta'(a)(b) 
        = \beta(\sigma(t)a)(b)
        = \gen{\sigma(t)a,b}
        = \gen{a,bt}.
  \]
Similarly, for each $a$ in~$A$ we have that
  \[
  \sigma(t)a
        = \beta^{-1}(\beta'(a))
        = \beta^{-1}(\sigma'(a)\beta'(1))
        = \sigma^{-1}(\sigma'(a))\beta^{-1}(\beta'(1))
        = \sigma^{-1}(\sigma'(a))\sigma(t),
  \]
so that $t\sigma(a)t^{-1}=\sigma'(a)$. This proves part~\thmitem{1} of the
lemma. On the other hand, if $s$ is another unit of~$A$ such that
$\gen{a,b}'=\gen{a,bs}$ for all~$a$ and~$b$ in~$A$, then as before we also
have that $\sigma'(a)=s\sigma(a)s^{-1}$ for all~$a\in A$, so that
  \[
  \sigma(a) 
        = s^{-1}\sigma'(a)s 
        = s^{-1}t\sigma(a)t^{-1}s
  \]
and, since $\sigma$ is of course surjective, we see that $t^{-1}s$ is
central in~$A$, as part~\thmitem{2} claims.
\end{proof}

It follows from part~\thmitem{1} of this lemma that
$\sigma'=\iota_t\circ\sigma$, with $\iota_t:a\in A\mapsto tat^{-1}\in A$
the inner automorphism of~$A$ determined by the unit~$t$, so that the
classes of $\sigma'$ and~$\sigma$ in the outer automorphism group~$\Out(A)$
coincide. The class of~$\sigma$ there therefore depends only on the
algebra~$A$. In this way we obtain an invariant of the algebra~$A$, an
element of~$\Out(A)$: it is the first and simplest of a series of
invariants of the same nature, as we will see later.

\subsection{The action on central elements and derivations}%
\label{subsect:hh01}

The following observation was made by Masaru Osima in~\cite{Osima:1}:

\begin{Proposition}\label{prop:sigma:central}
The action of the Nakayama automorphism of~$A$ on the
center~$\Z(A)$ of~$A$ is trivial.
\end{Proposition}

In~\cite{Osima:1} Osima proves this using the method of parastrophic
matrices. The following proof is simply a rephrasing of his in more modern
language.

\begin{proof}
If $z$ is a central element in~$A$, then for all $a$ in~$A$ we have that
  \[
  \beta(a)(z)
         = (\beta(a)z)(1)
         = \beta(az)(1)
         = \beta(za)(1)
         = (\sigma(z)\beta(a))(1) 
         = \beta(a)(\sigma(z))
  \]
and, since the function~$\beta$ is surjective, this implies that
$z=\sigma(z)$.
\end{proof}

Let $\alpha:A\to A$ be an automorphism of~$A$. If $\delta:A\to A$ is a
derivation of~$A$, then the map
  \[
  \delta^\alpha\coloneqq\alpha\circ\delta\circ\alpha^{-1}:A\to A
  \]
is also a derivation of~$A$. The function 
  \[
  \alpha^\sharp : \delta\in\Der(A) \mapsto \delta^\alpha \in \Der(A)
  \]
that we obtain in this way is an automorphism of the Lie algebra~$\Der(A)$
that preserves the ideal~$\InnDer(A)$ of inner derivations and that
therefore induces an automorphism
  \[
  \alpha^\sharp : \HH^1(A) \to \HH^1(A)
  \]
of the Lie algebra
  \[
  \HH^1(A) \coloneqq \Der(A) / \InnDer(A)
  \]
of outer derivations of~$A$ or, equivalently, on the first Hochschild
cohomology of~$A$. In this way we obtain a function
$\alpha\in\Aut(A)\mapsto\Aut(\HH^1(A))$ that is a morphism of groups and,
therefore, an action of~$\Aut(A)$ on the Lie algebra~$\HH^1(A)$.

In general, that action is non-trivial, and that is why the following
result --- which we can view as an «infinitesimal» version of Osima's
result above --- is remarkable.

\begin{Proposition}
The action of~$\sigma$ on~$\HH^1(A)$ is trivial.
\end{Proposition}

\begin{proof}
Let $\delta:A\to A$ be a derivation of~$A$, and let $\delta^\T:\D A\to\D A$
be the transposed map, so that 
  \(
  \delta^\T(\lambda)(a) = \lambda(\delta(a))
  \)
for all~$\lambda\in\D A$ and all $a\in A$. A little calculation shows that
  \[
  \delta^\T(a\lambda) = a\delta^\T(\lambda) - \delta(a)\lambda,
  \qquad
  \delta^\T(\lambda a) = \delta^\T(\lambda)a - \lambda\delta(a)
  \]
for all $\lambda\in\D A$ and all~$a\in A$. In the language
of~\cite{MSA:Bit}, these identities tell us that the map~$\delta^\T$ is a
left and right $(-\delta)$-operator on~$\D{A}$.

We have the bijective linear map~$\beta:A\to\D A$ that we defined above,
and we can therefore consider the map
  \[
  \phi \coloneqq \beta^{-1}\circ\delta^\T\circ\beta + \delta : A\to A.
  \]
If $a$ and~$b$ are elements of~$A$, then
  \begin{align}
  \phi(a\cdot b) 
       &= \beta^{-1}(\delta^\T(\beta(a\cdot b)))
          + \delta(a\cdot b) \\
       &= \beta^{-1}(\delta^\T(\beta(a)\cdot b))
          + a\cdot\delta(b) + \delta(a)\cdot b \\
       &= \beta^{-1}(\delta^\T(\beta(a))\cdot b - \beta(a)\cdot\delta(b))
          + a\cdot\delta(b) + \delta(a)\cdot b \\
       &= \beta^{-1}(\delta^\T(\beta(a)))\cdot b - a\cdot\delta(b)
          + a\cdot\delta(b) + \delta(a)\cdot b \\
       &= \phi(a)\cdot b.
  \end{align}
This tells us that the map~$\phi$ is right $A$-linear, and thus that the
element 
  \[
  x\coloneqq\phi(1)=\beta^{-1}(\delta^\T(\beta(1)))
  \]
of~$A$ is such that $\phi(a)=xa$ for all~$a\in A$. In particular, for every
$a$ in~$A$ we have that
  \begin{align}
  x\cdot a 
       &= \phi(a\cdot 1) \\
       &= \beta^{-1}(\delta^\T(\beta(a\cdot 1)))
          + \delta(a) \\
       &= \beta^{-1}(\delta^\T(\sigma(a)\cdot \beta(1)))
          + \delta(a) \\
       &= \beta^{-1}(\sigma(a)\cdot \delta^\T(\beta(1))
                     - \delta(\sigma(a))\cdot \beta(1))
          + \delta(a) \\
       &= a\cdot \beta^{-1}(\delta^\T(\beta(1)))
          - \sigma^{-1}(\delta(\sigma(a)))
          + \delta(a) \\
       &= a\cdot x
          - \delta^{\sigma^{-1}}(a)+\delta(a),
  \end{align} 
and this means that the difference~$\delta-\delta^{\sigma^{-1}}$ is equal
to~$\ad(x)$, an inner derivation, so that the classes of~$\delta$
and~$\delta^{\sigma^{-1}}$ in~$\HH^1(A)$ are equal. The claim of the
proposition follows immediately from this.
\end{proof}

\subsection{The action on Hochschild cohomology}%
\label{subsect:action}

An automorphism of an algebra~$A$ restricts to the center~$\Z(A)$ and, as
we saw in the previous section, acts in a certain natural way on the Lie
algebra~$\Der(A)$ of derivations of~$A$ and its quotient~$\HH^1(A)$. In
fact, automorphisms of~$A$ act in a canonical way on the whole Hochschild
cohomology~$\HH^*(A)$ of~$A$. Our objective here is to recall the
construction of this action, for we will need them.

\bigskip

Let $\Lambda$ be any algebra, and let $\alpha:\Lambda\to \Lambda$ be an
automorphism of~$\Lambda$. If $M$ is a left $\Lambda$-module, we write
${}_\alpha M$ for the left $\Lambda$-module that coincides with~$M$ as a
vector space and on which $\Lambda$ acts so that $a\cdot m=\alpha(a)m$ for
all $a\in \Lambda$ and all $m\in M$ --- we say that ${}_\alpha M$ is
obtained from~$M$ by \newterm{twisting} it by~$\alpha$. Of course, there is
a similar construction on right modules and on bimodules, and in this way
we obtain exact endofunctors of the various categories of (bi)modules that
act as the identity on morphisms.

Let now $\epsilon:P\to \Lambda$ be a projective resolution of~$\Lambda$ as
a $\Lambda$-bimodule. The map $\epsilon:{}_{\alpha}
P_{\alpha}\to{}_{\alpha} \Lambda_{\alpha}$ is then a projective resolution
of~${}_{\alpha} \Lambda_{\alpha}$ as an $\Lambda$-bimodule, and the map
$\alpha^{-1}:{}_{\alpha} \Lambda_\alpha \to \Lambda$ is an isomorphism of
$\Lambda$-bimodules, so there is a morphism
$(\alpha^{-1})_P:{}_{\alpha}P_{\alpha}\to P$ of complexes
$\Lambda$-bimodules that makes the square
  \[
  \begin{tikzcd}
  {}_{\alpha} P_{\alpha} \arrow[r, "\epsilon"] \arrow[d, swap, "(\alpha^{-1})_P"]
    & {}_{\alpha} \Lambda_{\alpha} \arrow[d, "\alpha^{-1}"]
    \\
  P \arrow[r, "\epsilon"] 
    & \Lambda 
  \end{tikzcd}
  \]
commute. As the map $\alpha:\Lambda\to{}_\alpha \Lambda_\alpha$ is also a
map of $\Lambda$-bimodules, we can consider the composition
  \[
  \alpha^\sharp_P:
  \begin{tikzcd}
  \Hom_{\Lambda^e}(P,\Lambda) 
        \arrow[r, "\alpha_*\circ(\alpha^{-1})_P^*"] 
  &[2.5em] \Hom_{\Lambda^e}({}_{\alpha} P_{\alpha},{}_\alpha \Lambda_\alpha) 
        \arrow[r, equal]
  &[-1em] \Hom_{\Lambda^e}(P,\Lambda) 
  \end{tikzcd}
  \]
It is easy to check that the homotopy type of~$\alpha^\sharp_P$ does not
depend on the choice of the morphism of complexes~$\alpha_P$. Furthermore,
if $\eta:Q\to \Lambda$ is another projective resolution of~$\Lambda$ as an
$\Lambda$-bimodule,
$\alpha^\sharp_Q:\Hom_{\Lambda^e}(Q,\Lambda)\to\Hom_{\Lambda^e}(Q,\Lambda)$
is the map that we obtain in this way from a choice of a lifting
$(\alpha^{-1})_Q: {}_{\alpha}Q_{\alpha}\to Q$ as above, and $\phi:Q\to P$
is any morphism of complexes such that $\epsilon\circ\phi=\eta$, then the
square 
  \[
  \begin{tikzcd}
  \Hom_{\Lambda^e}(P,\Lambda) \arrow[r, "\alpha^\sharp_P"]
        \arrow[d, swap, "\phi^*"]
    & \Hom_{\Lambda^e}(P,\Lambda) 
        \arrow[d, "\phi^*"]
    \\
  \Hom_{\Lambda^e}(Q,\Lambda) \arrow[r, "\alpha^\sharp_Q"]
    & \Hom_{\Lambda^e}(Q,\Lambda) 
  \end{tikzcd}
  \]
commutes up to homotopy. All this implies, as usual, that the
map~$\alpha^\sharp_P$ induces on cohomology a map
  \[
  \alpha^\sharp:
        \Ext_{\Lambda^e}^*(\Lambda,\Lambda) 
        \to 
        \Ext_{\Lambda^e}^*(\Lambda,\Lambda)
  \]
that is in fact completely determined by~$\alpha$, and in this way we
obtain a morphism of groups
  \[
  \alpha\in\Aut(\Lambda) \mapsto \alpha^\sharp \in \Aut(\HH^*(\Lambda))
  \]
giving a canonical action of~$\Aut(\Lambda)$ on~$\HH^*(\Lambda)$ as a
graded vector space. In fact, this action respects the Gerstenhaber
structure on~$\HH^*(\Lambda)$: it is easy to check this if we compute the
maps~$\alpha^\sharp$ using the standard bimodule resolution of~$\Lambda$.
When we do so, we moreover see that if $p$ is a non-negative integer and
$c:\Lambda^{\otimes p}\to \Lambda$ is a Hochschild $p$-cocycle representing
a class~$c\in\HH^p(\Lambda)$, then the class~$\alpha^\sharp(c)$ is
represented by the unique cocycle $c^\alpha:\Lambda^{\otimes p}\to \Lambda$
such that
  \[
  c^\alpha(a_1\otimes\cdots\otimes a_p)
        = \alpha(c(\alpha^{-1}(a_1)\otimes\cdots\otimes\alpha^{-1}(a_p)))
  \]
for all choices of~$a_1$,~\dots,~$a_p$ in~$\Lambda$. This tells us, in
particular, that the canonical action of~$\Aut(A)$
on~$\HH^0(\Lambda)=\Z(\Lambda)$ and on~$\HH^1(\Lambda)$ that we are
describing here is precisely the one that we used in the previous section,
which is reassuring.

\bigskip

The action of the subgroup~$\Inn(A)$ of inner automorphisms of~$A$
on~$\HH^*(A)$ is trivial, so what we in fact have is an action of the outer
automorphism group of~$A$ on that cohomology. In particular, when $A$ is a
Frobenius algebra we have a well-determined automorphism of~$\HH^*(A)$
induced by any of its Nakayama automorphisms, as all of them have the same
outer class.

\subsection{The action on~\texorpdfstring{$\HH^2(A)$}{HH\^{}2(A)}}%
\label{subsect:hh2}

Let us go back to the situation of Subsection~\ref{subsect:frobenius}. We
let $f:A\otimes A\to A$ be a Hochschild $2$-cocycle on~$A$ and consider the
map
  \[
  f^\T : \D A\otimes A\to\D A
  \]
such that
  \[
  f^\T(\lambda\otimes a)(b) = \lambda(f(a\otimes b))
  \]
whenever $\lambda\in\D A$ and $a$,~$b\in A$. This map $f^\T$ is
a~$1$-cochain in the standard complex 
  \[
  C_{A^\op}^*(\D A,\D A) \coloneqq \Hom_{A^\op}(\D A\otimes A^{\otimes *},\D A)
  \]
that computes~$\Ext_{A^\op}^*(\D A,\D A)$, and we can describe its
coboundary.

\begin{Lemma}
We have that $df^\T = \id_{\D A}\smile f$.
\end{Lemma}

Here $d$ is the differential of $C_{A^\op}^*(\D A,\D A)$ and $\smile$ is
the classical cup product
  \[
  C_{A^\op}^*(\D A,\D A) \otimes \Hom_{A^e}(A^{\otimes *+1},A)
        \to C_{A^\op}^*(\D A,\D A)
  \]
that computes the right action
  \[
  \Ext_{A^\op}^*(\D A,\D A)\otimes\HH^*(A)
        \to \Ext_{A^\op}^*(\D A,\D A)
  \]
of the Hochschild cohomology~$\HH^*(A)$ on~$\Ext_{A^\op}^*(\D A,\D A)$.

\begin{proof}
Let $a$ and~$b$ be elements of~$A$ and let $\lambda$ be an element of~$\D
A$. Since $f$ is a Hochschild $2$-cocycle, for all $c\in A$ we have that
  \[
  0 = af(b\otimes c) - f(ab\otimes c) + f(a\otimes bc) - f(a\otimes b)c,
  \]
so that
  \begin{align}
  0 
    &= \lambda(af(b\otimes c)) - \lambda(f(ab\otimes c)) 
       + \lambda(f(a\otimes bc)) - \lambda(f(a\otimes b)c) \\
    &= (\lambda a)(f(b\otimes c)) - \lambda(f(ab\otimes c)) 
       + \lambda(f(a\otimes bc)) - (\lambda f(a\otimes b))(c) \\
    &= f^\T(\lambda a\otimes b)(c) - f^\T(\lambda\otimes ab)(c)
       + (f^\T(\lambda\otimes a)b)(c) - (\lambda f(a\otimes b))(c),
  \end{align}
and this means us that
  \begin{align}
  0 &= f^\T(\lambda a\otimes b) - f^\T(\lambda\otimes ab)
       + f^\T(\lambda\otimes a)b - \lambda f(a\otimes b) \\
    &= df^\T(\lambda\otimes a\otimes b) 
        - (\id_{\D A}\smile f)(\lambda\otimes a\otimes b),
  \end{align}
which is precisely what the lemma claims.
\end{proof}

The map $\beta:A\to\D A$ is an isomorphism of left $A$-modules, so we can
consider the composition
  \[
  \beta^{-1}\circ f^\T\circ\beta\otimes\id_A:
  \underline A\otimes A\to \underline A.
  \]
This is a $1$-cochain in the complex $C_{A^{\op}}^*(A,A)$ that
computes~$\Ext_{A^\op}^*(A,A)$ --- we have underlined the appearances
of~$A$ that play the role of the right $A$-modules that are arguments of
the $\Ext$. It is an immediate consequence of the lemma that we have just
proved and of the naturality of the differentials that
  \begin{align}
  d(\beta^{-1}\circ f^\T\circ\beta\otimes\id_A)
        &= \beta^{-1}\circ df^\T\circ\beta\otimes\id_A\otimes\id_A \\
        &= \beta^{-1}
                \circ (\id_{\D A}\smile f)
                \circ\beta\otimes\id_{A}\otimes\id_{A} \\
        &= \id_{A}\smile f.
  \end{align}
On the other hand, if we view the Hochschild $2$-cocycle~$f:A\otimes A\to
A$, which is an element of~$C_{A^e}^2(A,A)$, as an element
of~$C_{A^\op}^1(A,A)$, then we have the relation
  \[
  df = \id_A\smile f,
  \]
precisely because it is a cocycle, and this implies, of course, that the
difference
  \[
  f - \beta^{-1}\circ f^\T\circ\beta\otimes\id_{A} : 
        \underline A\otimes A
        \to \underline A
  \]
is a $1$-cocycle. As $\Ext_{A^\op}^1(A,A)$ vanishes, this $1$-cocycle is a
coboundary, and this means that there is an element~$g:A\to A$
of~$C_{A^\op}^0(A,A)$ such that
  \[
  f - \beta^{-1}\circ f^\T\circ\beta\otimes\id_{A} 
        = dg.
  \]
In other words, the map~$g$ is such that
  \[
  f(a\otimes b) - \beta^{-1}(f^\T(\beta(a)\otimes b)) 
        = g(ab) - g(a)b
  \]
for all~$a$ and~$b$ in~$A$ or, equivalently, since the map~$\beta$ is
bijective and right $A$-linear, that
  \[
  \beta(f(a\otimes b)) - f^\T(\beta(a)\otimes b) 
        = \beta(g(ab)) - \beta(g(a))b.
  \]
This is an equality of two elements of~$\D A$, and it means that for all
$c\in A$ we have that
  \[
  \beta(f(a\otimes b))(c) - f^\T(\beta(a)\otimes b)(c) 
        = \beta(g(ab))(c) - (\beta(g(a))b)(c),
  \]
so that, in fact,
  \[ \label{eq:twist}
  \gen{f(a\otimes b),c} 
        = \gen{a,f(b\otimes c)} + \gen{g(ab),c} - \gen{g(a),bc}.
  \]
This is the key observation in this proof.

We consider now the map
  \[
  f^\sigma\coloneqq\sigma\circ f\circ\sigma^{-1}\otimes\sigma^{-1}:
        A\otimes A
        \to A,
  \]
which is an element of~$C_{A^e}^2(A,A)$, and carry out --- using repeatedly
the identities in~\eqref{eq:sigma} and in~\eqref{eq:twist} --- the
following fun calculation:
  \begin{align}
  \gen{c,f^\sigma(a\otimes b)} 
    &= \gen{c,\sigma(f(\sigma^{-1}(a)\otimes\sigma^{-1}(b)))} \\
    &= \gen{f(\sigma^{-1}(a)\otimes\sigma^{-1}(b)),c} \\
    &= \begin{lgathered}[t]
       \gen{\sigma^{-1}(a),f(\sigma^{-1}(b)\otimes c)} \\
       \qquad
       {}+ \gen{g(\sigma^{-1}(a)\sigma^{-1}(b)),c}
       - \gen{g(\sigma^{-1}(a)),\sigma^{-1}(b)c}
       \end{lgathered}
       \\
    &= \begin{lgathered}[t]
       \gen{f(\sigma^{-1}(b)\otimes c),a} \\
       \qquad
       {}+ \gen{g(\sigma^{-1}(a)\sigma^{-1}(b)),c}
       - \gen{g(\sigma^{-1}(a)),\sigma^{-1}(b)c}
       \end{lgathered}
       \\
    &= \begin{lgathered}[t]
       \gen{\sigma^{-1}(b),f(c\otimes a)} \\
       \qquad {}+ \gen{g(\sigma^{-1}(a)\sigma^{-1}(b)),c}
       - \gen{g(\sigma^{-1}(a)),\sigma^{-1}(b)c} \\
       \qquad {}+\gen{g(\sigma^{-1}(b)c),a}
       - \gen{g(\sigma^{-1}(b)), ca}
       \end{lgathered}
       \\
    &= \begin{lgathered}[t]
       \gen{f(c\otimes a),b} \\
       \qquad {}+ \gen{g(\sigma^{-1}(a)\sigma^{-1}(b)),c}
       - \gen{g(\sigma^{-1}(a)),\sigma^{-1}(b)c} \\
       \qquad {}+\gen{g(\sigma^{-1}(b)c),a}
       - \gen{g(\sigma^{-1}(b)), ca}
       \end{lgathered}
       \\
    &= \begin{lgathered}[t]
       \gen{c,f(a\otimes b)} \\
       \qquad {}+ \gen{g(\sigma^{-1}(a)\sigma^{-1}(b)),c}
       - \gen{g(\sigma^{-1}(a)),\sigma^{-1}(b)c} \\
       \qquad {}+\gen{g(\sigma^{-1}(b)c),a}
       - \gen{g(\sigma^{-1}(b)), ca} \\
       \qquad {}+\gen{g(ca),b}
       - \gen{g(c),ab}. \label{eq:w0}
       \end{lgathered}
  \end{align}
If we put
  \[
  g^\sigma \coloneqq \sigma\circ g\circ \sigma^{-1} : 
        A
        \to A,
  \]
viewed as an element of~$C_{A^e}^1(A,A)$, then what we have found is that
  \begin{align}
  \gen{c,dg^\sigma(a\otimes b)}
        &= \gen{c,ag^\sigma(b)} 
           - \gen{a,g^\sigma(ab)} 
           + \gen{c,g^\sigma(a)b} \\
        &= \gen{g(\sigma^{-1}(b)), ca}
           - \gen{g(\sigma^{-1}(a)\sigma^{-1}(b)),c}
           + \gen{g(\sigma^{-1}(a)),\sigma^{-1}(b)c}.
           \label{eq:w1}
  \end{align}
On the other hand, since the bilinear form~$\gen{\place,\place}:A\times
A\to\kk$ is non-degenerate, there is a unique linear map $g^*:A\to A$ such
that
  \[
  \gen{g(x),y} 
        = \gen{x,g^*(y)}
  \]
for all $x$ and~$y$ in~$A$, and if we also view it as an element
of~$C_{A^e}^1(A,A)$, then
  \begin{align}
  \gen{c,dg^*(a\otimes b)}
        &= \gen{c,ag^*(b)} 
           - \gen{a,g^*(ab)} 
           + \gen{c,g^*(a)b} \\
        &= \gen{g(ca),b}
           - \gen{g(c),ab}
           + \gen{g(\sigma^{-1}(b)c),a}.
           \label{eq:w2}
  \end{align}
Using this equality and that of~\eqref{eq:w1}, and going back
to~\eqref{eq:w0}, we see that for all choices of~$a$,~$b$ and~$c$ in~$A$ we
have that
  \[
  \gen{c,f^\sigma(a\otimes b)} 
        = \gen{c,f(a\otimes b)+dg^\sigma(a\otimes b)+dg^*(a\otimes b)}.
  \]
The non-degeneracy of the bilinear form then allows us to conclude,
finally, that
  \[
  f^\sigma(a\otimes b)
        = f(a\otimes b)+dg^\sigma(a\otimes b)+dg^*(a\otimes b)
  \]
for all~$a$ and~$b$ in~$A$, so that
  \[
  f^\sigma - f 
        = dg^\sigma + dg^*.
  \]
As the class of the $2$-cocycle~$f^\sigma$ is none other than the image of
that of~$f$ by the map~$\sigma^\sharp$ that gives the canonical action
of~$\sigma$ on~$\HH^2(A)$, all this has proved the following result:

\begin{Proposition}
The action of~$\sigma$ on~$\HH^2(A)$ is trivial. \qed
\end{Proposition}

At this point we have established that the Nakayama automorphism of~$A$
acts trivially on~$\HH^0(A)$, on~$\HH^1(A)$ and on~$\HH^2(A)$, and it is
natural, of course, to conjecture that the Nakayama automorphism acts
trivially on the whole Hochschild cohomology~$\HH^*(A)$ of~$A$. We will
prove this in the next section.

\subsection{The action on~\texorpdfstring{$\HH^*(A)$}{Hochschild cohomology}}%
\label{subsect:hhx}

Let $p$ be an integer such that $p\geq2$, and let $f : A^{\otimes p} \to A$
be a Hochschild $p$-cocycle on~$A$. There is a unique linear map
  \[
  f^\T : \D A\otimes A^{\otimes(p-1)}\to\D A
  \]
that for all choices of~$\lambda$ in~$\D A$ and
of~$a_1$,~\dots,~$a_{p-1}$,~$b$ in~$A$ has
  \[
  f^\T(\lambda\otimes a_1\otimes\cdots\otimes a_{p-1})(b)
        = \lambda(f(a_1\otimes\cdots\otimes a_{p-1}\otimes b)).
  \]
It is a $(p-1)$-cochain in the complex~$C_{A^\op}^*(\D A,\D A)$ and we can
describe its coboundary there:

\begin{Lemma}
In $C_{A^\op}^*(\D A,\D A)$ we have that $df^\T=(-1)^p\cdot\id_{\D A}\smile f$. \qed
\end{Lemma}

The proof of this is an uneventful direct calculation that we omit. What
interests us is that, since $\beta:A\to \D A$ is an isomorphism of right
$A$-modules, the claim of the lemma implies at once that
  \[
  d\bigl(\beta^{-1}\circ f^\T\circ \beta\otimes\id_A^{\otimes(p-1)}\bigr)
        = (-1)^p\cdot\id_A\smile f.
  \]
The differential here is now that of the complex~$C_{A^\op}^*(A,A)$.

We can view the Hochschild $p$-cocycle~$f$ on~$A$ as a $(p-1)$-cochain in
the complex $C_{A^\op}^*(A,A)$, and the fact that it is a cocycle
in~$C_{A^e}^*(A,A)$ is easily seen to imply that in~$C_{A^\op}^*(A,A)$ we
have
  \[
  df = \id_A\smile f.
  \]
Putting all this together, we see that the difference
  \[
  f - (-1)^p\cdot\beta^{-1}\circ f^\T\circ \beta\otimes\id_A^{\otimes(p-1)}:
        A\otimes A^{\otimes(p-1)} 
        \to A
  \]
is a $(p-1)$-cocycle in~$C_{A^\op}^*(A,A)$. The cohomology of this complex
is~$\Ext_{A^\op}^*(A,A)$ and the integer~$p-1$ is positive, so that
$(p-1)$-cocycle is actually a coboundary: there exists a linear map $g :
A^{\otimes(p-1)} \to A$ such that
  \[ \label{eq:fdg}
  f - (-1)^p\cdot\beta^{-1}\circ f^\T\circ \beta\otimes\id_A^{\otimes(p-1)}
        = dg
  \]
in $C_{A^\op}^*(A,A)$. This means that for all choices
of~$a_1$,~\dots,~$a_p$ in~$A$ we have
  \[
  f(a_1\otimes\cdots\otimes a_p)
        - (-1)^p\cdot\beta^{-1}(f^\T(\beta(a_1)\otimes a_2\otimes\cdots\otimes a_p)) 
        = dg(a_1\otimes\cdots\otimes a_p)
  \]
and, since the map~$\beta$ is right $A$-linear, that also
  \[
  \beta(f(a_1\otimes\cdots\otimes a_p))
        - (-1)^p\cdot f^\T(\beta(a_1)\otimes a_2\otimes\cdots\otimes a_p) 
        = d(\beta\circ g)(a_1\otimes\cdots\otimes a_p).
  \]
This is an equality in~$\D A$: we can evaluate both of its sides on an
element~$a_{p+1}$ from~$A$ and see that
  \begin{multline} \label{eq:swap}
  \gen{f(a_1\otimes\cdots\otimes a_p),a_{p+1}}
        - (-1)^p\cdot\gen{f(a_2\otimes a_3\otimes\cdots\otimes a_{p+1}),\sigma(a_1)} \\
  \qquad
    = d(\beta\circ g)(a_1\otimes\cdots\otimes a_p)(a_{p+1}). 
  \end{multline}
This allows us to «rotate» the arguments of~$f$, and rotate them we will.

Let us now consider the linear map
  \[
  f^\sigma \coloneqq \sigma\circ f\circ (\sigma^{-1})^{\otimes p}:
  A^{\otimes p} \to A,
  \]
which, as we observed in Subsection~\ref{subsect:action}, represents the
image of the class of~$f$ by~$\sigma$, and let $a_0$,~\dots,~$a_p$ be
elements of~$A$. For the sake of brevity let us write commas instead oft
the operator~$\otimes$, and $\bar a$ instead of~$\sigma^{-1}(a)$. Moreover,
when $i$ and $j$ are elements of~$\inter{0,p}$ we will write $a_{i;j}$ for
the sequence $a_i,a_{i+1},\dots,a_j$ if $i\leq j$, and for the empty
sequence if not. Using~\eqref{eq:swap} we see that
  \begin{align}
  \MoveEqLeft[1]
  \gen{a_0,f^\sigma(a_{1;p}) - f(a_{1;p})} 
     = \gen{f(\bar a_{1;p}),a_0} - \gen{f(a_{1;p}),\sigma(a_0)} 
  \\
  &= \begin{multlined}[t][.8\displaywidth]
     \sum_{i=1}^{p}
        (-1)^{(i-1)p} \cdot
        \smash{\Bigl(}
        \gen{f(\bar a_{i;p},a_{0;i-2}), a_{i-1}} 
        - (-1)^p
        \gen{f(\bar a_{i+1;p},a_{0;i-1}), a_{i}} 
        \smash{\Bigr)}
     \\
      {}+ (-1)^{p^2}\cdot
        \Bigl(
        \gen{f(a_{0;p-1}), a_{p}}
        - (-1)^p\cdot\gen{f(a_{1;p}),\sigma(a_0)}
        \Bigr)
    \end{multlined}
    \\
  &= \sum_{i=1}^{p+1}
        (-1)^{(i-1)p} \cdot
        d(\beta\circ g)(\bar a_{i;p}, a_{0;i-2})(a_{i-1}).
        \label{eq:tangle}
  \end{align}
We have to untangle this last sum.

\bigskip

Since the bilinear form~$\gen{\place,\place}$ on~$A$ is non-degenerate,
there are linear maps 
  \[
  g^*_1,\dots,g^*_{p-1}:A^{\otimes(p-1)}\to A
  \]
such that
  \[
  \gen{g(b_{1;p-1}),b_p}
        = \gen{b_q,g^*_q(b_{q+1;p},\sigma(b_{1,{q-1}}))}
  \]
for all choices of $b_1$,~\dots,~$b_p$ in~$A$ and of~$q$ in~$\inter{p-1}$.
There are, on the other hand, linear maps
  \[
  \partial_0$,~\dots,~$\partial_p:\Hom(A^{\otimes(p-1)},A)\to\Hom(A^{\otimes p},A)
  \]
such that    
  \[
  \partial_k h(b_{1;p})
        = \begin{cases*}
          b_1h(b_{2;p}) & if $k=0$; \\
          h(b_{1;k-1},b_kb_{k+1},b_{k+2;p}) & if $1\leq k<p$; \\
          h(b_{1;p-1})b_p & if $k=p$
          \end{cases*}
  \]
for all $j\in\inter{0,p}$, all $h\in\Hom(A^{\otimes(p-1)},A)$ and all
$b_1$,~\dots,~$b_p\in A$. We care about these, of course, because the
coboundary of a $(p-1)$-cochain $h:A^{\otimes(p-1)}\to A$ in the Hochschild
complex $C_{A^e}^*(A,A)$ is
  \[
  dh = \sum_{j=0}^p(-1)^j\cdot\partial_jh.
  \]

The first term in the sum~\eqref{eq:tangle} is
  \begin{align} 
  \MoveEqLeft[1.5]
  d(\beta\circ g)(\bar a_{1;p})(a_{0}) \\
    &= \sum_{j=1}^{p-1}
        (-1)^{j-1}\cdot
        \gen{g(\bar a_{1;j-1},\bar a_j\bar a_{j+1},\bar a_{j+2;p}),a_0}
       + (-1)^p\cdot\gen{g(\bar a_{1;p-1})\bar a_p,a_0} 
       \\
    &= \sum_{j=1}^{p-1}
        (-1)^{j-1}\cdot
        \gen{a_0,g^\sigma(a_{1;j-1},a_ja_{j+1},a_{j+2;p})}
       + (-1)^p\cdot\gen{a_0,g^\sigma(a_{1;p-1})a_p} 
       \notag
       \\
    &= \sum_{j=1}^{p}
        (-1)^{j-1}\cdot
        \gen{a_0,\partial_jg^\sigma(a_{1;p})}.
        \label{eq:term:1}
  \end{align}
On the other hand, if $i$ is an element of~$\inter{2,p}$, then
  \begin{align}
  \MoveEqLeft[1.5]
  (-1)^{(i-1)p} \cdot d(\beta\circ g)(\bar a_{i;p}, a_{0;i-2})(a_{i-1}) \\
    &= \begin{lgathered}[t]
       \sum_{j=i}^{p-1}
        (-1)^{(i-1)p+(j-i)} \cdot
        \gen{g(\bar a_{i,j-1},\bar a_{j}\bar a_{j+1},\bar a_{j+2,p},a_{0;i-2}), a_{i-1}} 
        \\
       \qquad{}+ (-1)^{(i-1)p+(p-i)} \cdot
        \gen{g(\bar a_{i,p-1},\bar a_{p}a_{0},a_{1;i-2}), a_{i-1}} 
        \\
       \qquad{}+ (-1)^{(i-1)p+(p+1-i)} \cdot
        \gen{g(\bar a_{i,p},a_{0}a_{1},a_{2;i-2}), a_{i-1}} 
        \\
       \qquad{}+\sum_{j=1}^{i-3}
        (-1)^{(i-1)p+(p+1+j-i)} \cdot
        \gen{g(\bar a_{i,p},a_{0;j-1},a_ja_{j+1},a_{j+2,i-2}), a_{i-1}} 
        \\
       \qquad{}+ (-1)^{(i-1)p+(p+1+i-2-i)}
        \gen{g(\bar a_{i,p},a_{0;i-3})a_{i-2}, a_{i-1}} 
       \end{lgathered}
       \\
    &= \begin{lgathered}[t]
       \sum_{j=i}^{p-1}
        (-1)^{(i-1)p+(j-i)} \cdot
        \gen{a_0,g^*_{p-i+1}(a_{1;j-2},a_ja_{j+1},a_{j+2;p})} 
        \\
       \qquad{}+ (-1)^{(i-1)p+(p-i)} \cdot
        \gen{a_{0}, g^*_{p-i+1}(a_{1;p-1})a_p} 
        \\
       \qquad{}+ (-1)^{(i-1)p+(p+1-i)} \cdot
        \gen{a_{0}, a_{1}g^*_{p-i+2}(a_{2;p})}  
        \\
       \qquad{}+\sum_{j=1}^{i-3}
        (-1)^{(i-1)p+(p+1+j-i)} \cdot
        \gen{a_{0},g^*_{p-i+2}(a_{1;j-1},a_ja_{j+1},a_{j+2;p})} 
        \\
       \qquad{}+ (-1)^{(i-1)p+(p+1+i-2-i)}
        \gen{a_0, g^*_{p-i+2}(a_{1;i-3},a_{i-2}a_{i-1},a_{i;p})} 
       \end{lgathered}
       \\
    &= \begin{lgathered}[t]
       \sum_{j=i}^{p}
        (-1)^{(i-1)p+(j-i)} \cdot
        \gen{a_0,\partial_jg^*_{p-i+1}(a_{1;p})} 
        \\
       \qquad{}+\sum_{j=0}^{i-2}
        (-1)^{(i-1)p+(p+1+j-i)} \cdot
        \gen{a_{0},\partial_jg^*_{p-i+2}(a_{1;p})}.
        \label{eq:term:i}
       \end{lgathered}
  \end{align}
Finally, the last term in the sum~\eqref{eq:tangle} is
  \begin{align}
  \MoveEqLeft[1.5]
  (-1)^{p^2} \cdot d(\beta\circ g)(a_{0;p-1})(a_{p}) \\
    &= 
       (-1)^p 
       \sum_{j=0}^{p-2} 
       (-1)^j\cdot\gen{g(a_{0;j-1},a_ja_{j+1},a_{j+2;p-1}), a_p} 
       {}-\gen{g(a_{0;p-2})a_{p-1}, a_p} 
       \\
    &= \begin{lgathered}[t]
       (-1)^{p-1} 
       \sum_{j=0}^{p-2} 
       (-1)^{j-1}\cdot\gen{a_0,g^*_1(a_{1;j-1},a_ja_{j+1},a_{j+2;p})}
       -\gen{a_{0}, g^*_1(a_{1;p-2},a_{p-1}a_p)} 
       \end{lgathered}
       \\
    &= \begin{lgathered}[t]
       (-1)^{p-1} 
       \sum_{j=0}^{p-2} 
       (-1)^{j-1}\cdot\gen{a_0,\partial_jg^*_1(a_{1;p})}
       -\gen{a_{0}, \partial_{p-1}g^*_1(a_{1;p})}.
       \end{lgathered}
        \label{eq:term:p}
  \end{align}
It follows from~\eqref{eq:term:1},~\eqref{eq:term:i} and this that
  \begin{align}
  \MoveEqLeft[1]
  \sum_{i=1}^{p+1}
        (-1)^{(i-1)p} \cdot
        d(\beta\circ g)(\bar a_{i;p}, a_{0;i-2})(a_{i-1}) \\
    &= \begin{lgathered}[t]
       \sum_{j=1}^{p}
        (-1)^{j-1}\cdot
        \gen{a_0,\partial_jg^\sigma(a_{1;p})}
       +
       \sum_{i=2}^{p}
       \sum_{j=i}^{p}
        (-1)^{(i-1)p+(j-i)} \cdot
        \gen{a_0,\partial_jg^*_{p-i+1}(a_{1;p})} 
        \\
       \qquad{}+
       \sum_{i=2}^{p}
        \sum_{j=0}^{i-2}
        (-1)^{(i-1)p+(p+1+j-i)} \cdot
        \gen{a_{0},\partial_jg^*_{p-i+2}(a_{1;p})}
        \\
       \qquad{}+(-1)^{p-1} 
       \sum_{j=0}^{p-2} 
       (-1)^{j-1}\cdot\gen{a_0,\partial_jg^*_1(a_{1;p})}
       -\gen{a_{0}, \partial_{p-1}g^*_1(a_{1;p})}.
       \end{lgathered}
  \end{align}
Changing the index~$i$ of summation in the double sums, we see that this
last expression has the same value as
  \begin{gather}
       \sum_{j=1}^{p}
        (-1)^{j-1}\cdot
        \gen{a_0,\partial_jg^\sigma(a_{1;p})}
       +  
       \sum_{i=1}^{p-1}
        (-1)^{i(p-1)} 
       \sum_{j=p-i+1}^{p}
        (-1)^{j-1} \cdot
        \gen{a_0,\partial_jg^*_{i}(a_{1;p})} 
        \\
       \qquad{}+ 
       \sum_{i=2}^{p}
        (-1)^{i(p-1)} \
        \sum_{j=0}^{p-i}
        (-1)^{j-1} \cdot
        \gen{a_{0},\partial_jg^*_{i}(a_{1;p})}.
        \\
       \qquad{}+(-1)^{p-1} 
       \sum_{j=0}^{p-2} 
       (-1)^{j-1}\cdot\gen{a_0,\partial_jg^*_1(a_{1;p})}
       -\gen{a_{0}, \partial_{p-1}g^*_1(a_{1;p})},
  \end{gather}
and a little regrouping of the terms in these sums shows that what we have
is equal to
  \begin{gather}
  \sum_{j=1}^{p}
    (-1)^{j-1}\cdot
    \gen{a_0,\partial_jg^\sigma(a_{1;p})}
  - \gen{a_{0},\partial_0g^*_{p}(a_{1;p})} 
  \\
  \qquad{}+
  \sum_{i=1}^{p-1}
    (-1)^{i(p-1)} 
    \sum_{j=0}^{p}
    (-1)^{j-1} \cdot
    \gen{a_{0},\partial_jg^*_{i}(a_{1;p})}.
  \end{gather}
Finally, since
$\gen{a_{0},\partial_0g^*_{p}(a_{1;p})}=\gen{a_0,\partial_0g^\sigma(a_{1;p})}$,
we can rewrite this in the form
  \[
  \sum_{j=0}^{p}
    (-1)^{j-1}\cdot
    \gen{a_0,\partial_jg^\sigma(a_{1;p})}
  + \sum_{i=1}^{p-1}
    (-1)^{i(p-1)} 
    \sum_{j=0}^{p}
    (-1)^{j-1} \cdot
    \gen{a_{0},\partial_jg^*_{i}(a_{1;p})}
  \]
or, more simply,
  \[
  \gen*{
      a_0,
      dg^\sigma(a_{1;p})
      + \sum_{i=1}^{p-1}
        (-1)^{i(p-1)} 
        dg^*_i(a_{1;p})
     }.
  \]
According to the equality~\eqref{eq:tangle}, this coincides with
$\gen{a_0,f^\sigma(a_{1;p}) - f(a_{1;p})}$. As this is the case for all
choices of~$a_0$,~\dots,~$a_p$ in~$A$, we can therefore conclude that
  \[
  f^\sigma - f
        = dg^\sigma
          + \sum_{i=1}^{p-1}
            (-1)^{i(p-1)} 
            dg^*_i
  \]
and, in particular, that the $p$-cocycles~$f^\sigma$ and~$f$ are
cohomologous in the complex $C_{A^e}^*(A,A)$. In view of the observations
made in Subsection~\ref{subsect:action}, this proves the following result.

\begin{Proposition}\label{prop:main}
The action of~$\sigma$ on~$\HH^*(A)$ is trivial. \qed
\end{Proposition}

This is the main result of this paper.

\bigskip

We would like to emphasize the fact that this result is effective in the
following sense. Since the algebra~$A$ is unitary, there is a very explicit
(and simple!) contracting homotopy for the complex~$C_{A^{\op}}^*(A,A)$ and
it is therefore possible to find, starting from a Hochschild
$p$-cocycle~$f$ in~$C_{A^e}^*(A,A)$, the $(p-1)$-cochain~$g$ that makes the
equality~\eqref{eq:fdg} true in a very concrete way and from that the
Hochschild $(p-1)$-cochain that certifies that $f^\sigma$ and~$f$ are
cohomologous. Indeed, this is precisely what we did in
Subsection~\ref{subsect:hh01} when dealing with~$\HH^0(A)$ and~$\HH^1(A)$.

\subsection{The action on homology}%
\label{subsect:homology}

In the next section we will present some applications of our main result.
Before doing that we want to spend a little time considering the
homological side of things.

\bigskip

Let $\Lambda$ be an algebra, let $\alpha:\Lambda\to\Lambda$ be an
automorphism of~$\Lambda$, and let $\epsilon:P\to\Lambda$ be a projective
resolution of~$\Lambda$ as a $\Lambda$-bimodule. As in
Subsection~\ref{subsect:action} the map $\epsilon:{}_\alpha P_\alpha\to
{}_\alpha \Lambda_\alpha$ is then a projective resolution of the twisted
bimodule~${}_\alpha \Lambda_\alpha$ over~$\Lambda^e$, the map
$\alpha:\Lambda\to{}_\alpha \Lambda_\alpha$ is $\Lambda^e$-linear, and
there is a morphism of complexes of $\Lambda$-bimodules $\alpha_P:P\to
{}_\alpha P_\alpha$ such that the square
  \[
  \begin{tikzcd}
  P \arrow[r, "\epsilon"] \arrow[d, swap, "\alpha_P"]
    & \Lambda \arrow[d, "\alpha"]
    \\
  {}_\alpha P_\alpha \arrow[r, "\epsilon"]
    & {}_\alpha \Lambda_\alpha
  \end{tikzcd}
  \]
commutes. Let now $M$ be a $\Lambda$-bimodule and let
$\alpha_M:M\to{}_\alpha M_\alpha$ be an arbitrary morphism of
$\Lambda$-bimodules. The homotopy class of the composition\footnote{It is
in order to have the identification indicated by the equal sign in this
diagram that we need $\alpha$ to be an automorphism and not only an
endomorphism of~$\Lambda$.}
  \[
  \begin{tikzcd}
  M\otimes_{\Lambda^e}P \arrow[r, "\alpha_M\otimes\alpha_P"]
    &[2em] ({}_\alpha M_\alpha) \otimes_{\Lambda^e} ({}_\alpha P_\alpha) \arrow[r, equal]
    & M\otimes_{\Lambda^e}P
  \end{tikzcd}
  \]
depends only on~$\alpha$ and on~$\alpha_M$, and not on the specific choice
of the morphism~$\alpha_P$, and the map
  \[
  \H_*(\Lambda,M) = \Tor^{\Lambda^e}_*(M,\Lambda) = \H_*(M\otimes_{\Lambda^e}P)
  \to \H_*(M\otimes_{\Lambda^e}P) = \Tor^{\Lambda^e}_*(M,\Lambda) = \H_*(\Lambda,M)
  \]
that it induces on homology depends only on~$\alpha$ and~$\alpha_M$ and not
on on the choice of the resolution~$P\to \Lambda$, up to canonical
isomorphisms. We therefore have a well-determined map
  \[
  (\alpha,\alpha_M)_\sharp : \H_*(\Lambda,M) \to \H_*(\Lambda,M).
  \]
In particular, if $\beta:\Lambda\to\Lambda$ is another automorphism
of~$\Lambda$ that commutes with~$\alpha$, then we can consider the twisted
$\Lambda$-bimodule $\Lambda_\beta$ and the function
$\alpha_\beta\coloneqq\alpha:\Lambda\to \Lambda_\beta$, which is a morphism
of $\Lambda$-bimodules, and thus obtain a map
  \[
  (\alpha,\alpha_\beta)_\sharp:\H_*(\Lambda,\Lambda_\beta)\to\H_*(\Lambda,\Lambda_\beta).
  \]
We will write $\alpha_\sharp$ to denote it, omitting any reference
to~$\beta$. In this way we obtain a function
  \[
  \alpha\in\Aut(\Lambda)_\beta 
    \mapsto 
    \alpha_\sharp \in \Aut(\H_*(\Lambda,\Lambda_\beta))
  \]
from the centralizer~$\Aut(\Lambda)_\beta$ of~$\beta$ in~$\Aut(\Lambda)$
into the group of automorphisms of the graded vector
space~$\H_*(\Lambda,\Lambda_\beta)$, which is in fact a morphism of groups.
This is what we call the \newterm{canonical action} of~$\Aut(\Lambda)$
on~$\H_*(\Lambda,\Lambda_\beta)$. 

In particular, as $\id_\Lambda$ is of course central in~$\Aut(\Lambda)$, in
this way we obtain an action of the whole automorphism
group~$\Aut(\Lambda)$ on the Hochschild
homology~$\HH_*(\Lambda)\coloneqq\H_*(\Lambda,\Lambda)$ of~$\Lambda$. This
action and the one that we already have on~$\HH^*(\Lambda)$ respect the
whole Tamarkin--Tsygan calculus of the algebra --- the \emph{cup} and
\emph{cap} products, the Gerstenhaber bracket, and the Connes differential,
and it is given, when we compute homology in terms of the standard
Hochschild complex, by the obvious formulas.

\bigskip

Let us go back to the situation in Subsection~\ref{subsect:frobenius}. This
construction provides us with a map~$\sigma_\sharp:\HH_*(A)\to\HH_*(A)$. In
contrast with what we found in the previous section, this map is usually
not trivial.

\begin{Example}\label{ex:four}
Let $q$ be an element of~$\kk$ different from~$0$, let us consider the
algebra $A\coloneqq\kk\gen{X,Y}/(X^2,YX-qXY,Y^2)$, and let us write~$x$
and~$y$ for the classes of~$X$ and~$Y$ in~$A$. It is easy to check that $A$
is freely spanned by the four monomials~$1$,~$x$,~$y$ and~$xy$, so that
there is a linear map $\lambda:A\to\kk$ such that
$\lambda(1)=\lambda(x)=\lambda(y)=0$ and $\lambda(xy)=1$. The bilinear form
  \[
  \gen{\place,\place}:(a,b)\in A\times A\mapsto\lambda(ab)\in\kk
  \]
is obviously associative and a little calculation show that it is
non-degenerate, so we have a Frobenius algebra. The Nakayama automorphism
$\sigma:A\to A$ corresponding to this bilinear form has~$\sigma(x)=q^{-1}x$
and~$\sigma(y)=qy$.

The subspace~$[A,A]$ is spanned by~$xy$, the $0$th Hochschild homology
space~$\HH_0(A)$ is canonically isomorphic to the quotient~$A/[A,A]$, and
the map~$\sigma_\sharp:\HH_0(A)\to\HH_0(A)$ is, up to that isomorphism, the
obvious map~$A/[A,A]\to A/[A,A]$ induced by~$\sigma$. The
map~$\sigma_\sharp$ is thus clearly not trivial.
\end{Example}

What we do have is that $\sigma$ acts trivially on appropriately twisted
homology:

\begin{Proposition}\label{prop:twisted}
The action of~$\sigma$ on the $\sigma$-twisted Hochschild
homology~$\H_*(A,A_\sigma)$ is trivial.
\end{Proposition}

\begin{proof}
Let $P$ be a projective resolution of~$A$ as an $A$-bimodule. There are
obvious identifications
  \[
  {}_\sigma \D A
        = {}_\sigma\Bigl(\Hom_\kk(A,\kk)\Bigr)
        = \Hom_\kk(A_\sigma,\kk),
  \]
and the composition
  \begin{equation*}
  \begin{tikzcd}[column sep=0.6em]
  \Hom_{A^e}(P,A) \arrow[r, "\beta_*"]
    &[0.7em] \Hom_{A^e}(P,{}_\sigma\D A) \arrow[r, equal]
    & \Hom_{A^e}(P,\Hom_\kk(A_\sigma,\kk)) \arrow[r]
    & \Hom_\kk(A_\sigma\otimes_{A^e}P,\kk)
  \end{tikzcd}
  \end{equation*}
of the morphism induced by the isomorphism of $A$-bimodules~$\beta:A\to
{}_\sigma\D A$ and the adjunction isomorphism induces an isomorphism of
graded vector spaces
  \[ \label{eq:partial}
  \HH^*(A) \to \Hom_\kk(\HH_*(A,A_\sigma),\kk).
  \]
This map commutes with the maps induced on its domain and codomain
by~$\sigma$ --- the easiest way to check this is to take for~$P$ the
standard resolution, on which everything can be made very explicit --- and
therefore the action of~$\sigma$ on~$\H_*(A,A_\sigma)$ is trivial, because
it is trivial on~$\HH^*(A)$.
\end{proof}

This «fixes» our Example~\ref{ex:four}:

\begin{Example}
In the situation of Example~\ref{ex:four}, the subspace~$[A_\sigma,A]$ is
spanned by the two elements~$(q^{-1}-1)x$ and~$(q-1)y$. As the
space~$\H_0(A,A_\sigma)$ is canonically isomorphic to the
quotient~$A_\sigma/[A_\sigma,A]$ and the action of~$\sigma$
on~$\H_0(A,A_\sigma)$ corresponds to the obvious one
on~$A_\sigma/[A_\sigma,A]$, it follows at once from this that $\sigma$ acts
trivially on~$\H_0(A,A_\sigma)$, as it should, when $q\neq1$, for that
vector space is then spanned by the classes of~$1$ and~$xy$, which are
fixed by~$\sigma$. If instead $q=1$ then the automorphism~$\sigma$ is just
the identity of~$A$ and, of course, it acts trivially on~$\H_0(A,A_\sigma)$
also in this case.
\end{Example}

\section{Invariants for Frobenius algebras}%
\label{sect:applications}

In this section we will show how to use our main result to construct
invariants for Frobenius algebras. 

\subsection{Jacobians}
\label{subsect:jac}

As before, let $A$ be a finite-dimensional algebra, let
$\gen{\place,\place}:A\times A\to\kk$ be a non-degenerate and associative
bilinear form on~$A$, let $\sigma:A\to A$ be the corresponding Nakayama
automorphism, and let $\beta:A\to{}_\sigma\D A$ be the associated
isomorphism of $A$-bimodules. When $x$ and~$y$ are elements of a group, we
write $[x,y]$ for their commutator~$xyx^{-1}y^{-1}$.

\bigskip

One of our starting points is the following lemma, which states an
observation essentially made by Osima in~\cite{Osima:1}. Together with this
result, Lemmas~\ref{lemma:det:conj} and~\ref{lemma:jac:cocycle}, and
Corollary~\ref{coro:det:commute} that we present below also appear in that
work. We have included proofs for them because Osima's proofs are entirely
based on matrix calculations and thus somewhat opaque --- it is not
entirely clear that our proofs are that much of an improvement in that
respect, though.

\begin{Lemma}\label{lemma:det}
Let $u:A\to A$ be an endomorphism of the algebra~$A$.
\begin{thmlist}

\item If $a$ is an element of~$A$, then there is one and only one
element~$\hat u(a)$ in~$A$ such that $\gen{u(a),u(b)} = \gen{\hat u(a),b}$
for all~$b\in B$.

\item The function $\hat u:A\to{} A$ that we obtain in this way is a
morphism of right $A$-modules, there exists there a unique element
$\jac_\sigma(u)$ in~$A$ such that 
  \(
  \hat u(a)=\jac_\sigma(u)\cdot a
  \)
for all $a\in A$, and we therefore have that
  \[
  \gen{u(a),u(b)} = \gen{\jac_\sigma(u)\cdot a,b} \label{eq:dets:0}
  \]
for all $a$ and~$b$ in~$A$.

\end{thmlist}
\end{Lemma}

We call the element~$\jac_\sigma(u)$ the \newterm{Nakayama Jacobian} of the
endomorphism~$u$ with respect to the Nakayama automorphism~$\sigma$. 

\begin{proof}
The first part of the lemma is an immediate consequence of the
non-degeneracy of the bilinear form~$\gen{\place,\place}$, and the
linearity of the function~$\hat u$ that we obtain is a consequence of the
uniqueness stated there. If $a$ and~$b$ are two elements of~$A$, then for
all $c\in A$ we have that
  \begin{align}
  \gen{\hat u(ab),c}
       &= \gen{u(ab),u(c)}
        = \gen{u(a)u(b),u(c)}
        = \gen{u(a),u(b)u(c)}  
        = \gen{u(a),u(bc)}  \\
       &= \gen{\hat u(a),bc}
        = \gen{\hat u(a)b,c},
  \end{align}
and this implies that $\hat u(ab)=\hat u(a)b$. The function~$\hat u: A\to
A$ is therefore right $A$-linear, and this implies that the element
$\jac_\sigma(u)\coloneqq\hat u(1)$ of~$A$ is the unique one that has the
property that
  \[
  \gen{\jac_\sigma(u)\cdot a,b}
        = \gen{\hat u(a),b}
        = \gen{u(a),u(b)}
  \]
for all~$a$ and~$b$ in~$A$. This shows that all the claims in
part~\thmitem{2} are true.
\end{proof}

The following lemma describes the key property of the Nakayama Jacobian.

\begin{Lemma}\label{lemma:jac:cocycle}
If $u$,~$v:A\to A$ are any two \emph{automorphisms} of~$A$, then
  \[
  \jac_\sigma(u\circ v) = \jac_\sigma(v)\cdot v^{-1}(\jac_\sigma(u)),
  \]
\end{Lemma}

We view this equality as a form of the chain rule of calculus. The fact
that $v$ appears inverted in it is a consequence of our choice of
variances, which is, in turn, motivated by our desire to be able to define
Jacobians for endomorphisms and not just for automorphisms.

\begin{proof}
For any two elements $a$ and~$b$ of~$A$ we have that
  \begin{align}
  \gen{\jac_\sigma(uv)\cdot a,b}
       &= \gen{u(v(a)), u(v((b))}
        = \gen{\jac_\sigma(u)\cdot v(a), v(b)} \\
       &= \gen{v(v^{-1}(\jac_\sigma(v))\cdot a), v(b)} 
        = \gen{\jac_\sigma(v)\cdot v^{-1}(\jac_\sigma(u))\cdot a, b},
  \end{align}
and this implies that $\jac_\sigma(uv) = \jac_\sigma(v)\cdot
v^{-1}(\jac_\sigma(u))$, as the lemma claims.
\end{proof}

A consequence of this is that the following analogue of the Jacobian
Conjecture is true:

\begin{Lemma}\label{lemma:det:JC}
An endomorphism $u:A\to A$ of~$A$ is an automorphism if and only if its
Jacobian~$\jac_\sigma(u)$ is a unit of~$A$, and when that is the case we
have that $\jac_\sigma(u^{-1}) = u(\jac_\sigma(u)^{-1})$.
\end{Lemma}

\begin{proof}
Let $u:A\to A$ be an endomorphism of~$A$. If $u$ is an automorphism, then
it follows from Lemma~\ref{lemma:jac:cocycle} that
  \[
  1 = \jac_\sigma(\id_A)
    = \jac_\sigma(u\circ u^{-1})
    = \jac_\sigma(u)\cdot u^{-1}(\jac_\sigma(u^{-1})),
  \]
and thus that $\jac_\sigma(u)$ is a unit of~$A$ and that
$\jac_\sigma(u^{-1}) = u(\jac_\sigma(u)^{-1})$. This proves that the
condition given by the lemma is necessary and its final claim.

To prove that the condition is also sufficient, let us suppose that the
endomorphism~$u$ is such that $\jac_\sigma(u)$ is a unit. If $a$ is an
element of~$A$ such that $u(a)=0$, then for all $b\in B$ we have that
  \[
  \gen{\jac_\sigma(u)\cdot a,b} = \gen{u(a),u(b)} = 0,
  \]
so that $\jac_\sigma(u)\cdot a=0$ and, in view of the hypothesis, that in
fact $a=0$. This tells us that the map~$u$ is injective and, since $A$ is
finite-dimensional, an automorphism.
\end{proof}

In what follows we will be exclusively concerned with automorphisms of~$A$.
First, we see that when $u$ is an automorphism the map~$\hat u$ of
Lemma~\ref{lemma:det} is left $A$-linear if we twist the left module
structure of its codomain appropriately.

\begin{Lemma}\label{lemma:det:left}
Let $u:A\to A$ be an automorphism of~$A$.
The function 
  \(
  \hat u:A\to{}_{[\sigma^{-1},u^{-1}]} A
  \)
of Lemma~\ref{lemma:det} is an isomorphism of $A$-bimodules, and for all
$a$ in~$A$ we have that
  \[
  \jac_\sigma(u)\cdot u^{-1}(\sigma^{-1}(a)) 
        = \sigma^{-1}(u^{-1}(a))\cdot \jac_\sigma(u) \label{eq:detcom}
  \]
\end{Lemma}

\begin{proof}
We already know that $\hat u$ is right $A$-linear and bijective, and it is
right $A$-linear because for all~$a$,~$b$ and~$c$ in~$A$ we have that
  \begin{align}
  \gen{\hat u(ab),c}
       &= \gen{u(ab),u(c)}
        = \gen{u(a)\cdot u(b),u(c)}
        = \gen{u(b),u(c)\cdot \sigma(u(a))} \notag \\
       &= \gen{u(b),u(c\cdot u^{-1}(\sigma(u(a))))} 
        = \gen{\hat u(b),c\cdot u^{-1}(\sigma(u(a)))} \\
       &= \gen{\sigma^{-1}(u^{-1}(\sigma(u(a))))\cdot \hat u(b),c}
        = \gen{[\sigma^{-1},u^{-1}](a)\cdot \hat u(b),c}.
  \end{align}
It follows from this that 
  \[ \label{eq:detcom:2}
  \jac_\sigma(u)\cdot a
        = \hat u(1)\cdot a
        = \hat u(a)
        = [\sigma^{-1},u^{-1}](a)\cdot \hat u(1)
        = [\sigma^{-1},u^{-1}](a)\cdot \jac_\sigma(u)
  \]
for all~$a\in A$, and this equality implies the one in the lemma.
\end{proof}

An immediate consequence of this lemma is the following observation.

\begin{Corollary}\label{coro:det:deriv}
If $u:A\to A$ is an automorphism of~$A$, then 
  \(
  [\sigma^{-1},u^{-1}](\jac_\sigma(u)) = \jac_\sigma(u)^{-1}
  \).
\end{Corollary}

This tells us that the Jacobian of an automorphism~$A$ is in the same orbit
under the action of derived subgroup of~$\Aut(A)$ as its inverse. This is a
non-trivial condition that limits the possible values that Jacobian can
take.

\begin{proof}
This follows from the identity~\eqref{eq:detcom} of
Lemma~\ref{lemma:det:left} or, equivalently, from the
identity~\eqref{eq:detcom:2} in the proof of that lemma if we take there
$a=\jac_\sigma(u)^{-1}$.
\end{proof}

We can evaluate automorphisms conjugate to~$\sigma$ using Jacobians in the
following way:

\begin{Lemma}\label{lemma:det:conj}
Let $u:A\to A$ be an automorphism of~$A$. For all $a\in A$ we have that
  \[
  (u\circ\sigma\circ u^{-1})(a) 
        = \sigma(\jac_\sigma(u)\cdot a\cdot \jac_\sigma(u)^{-1})
  \]
\end{Lemma}

The proof that we can provide of this is remarkably unenlightening.

\begin{proof}
The identity~\eqref{eq:dets:0} implies immediately that for all $a\in A$ we
have
  \[ \label{eq:t1:x}
  (u^\T\circ\beta\circ u)(a) = \beta(\jac_\sigma(u)\cdot a),
  \]
and therefore that also
  \[ \label{eq:t2:x}
  (u^ {-1}\circ\beta^{-1}\circ u^{-\T})(\lambda) 
        = \jac_\sigma(u)^{-1}\cdot\beta^{-1}(\lambda)
  \]
for all $\lambda\in \D A$. On the other hand, if~$a$ and~$\lambda$ are
elements of~$A$ and of~$\D A$, respectively, then for all $b\in B$ we have
that
  \begin{align}
  u^\T(\sigma(u(a))\cdot u^{-\T}(\lambda))(b)
        &= (\sigma(u(a))\cdot u^{-\T}(\lambda))(u(b)) \\
        &= u^{-\T}(\lambda)(u(b)\cdot \sigma(u(a))) \\
        &= \lambda(b\cdot u^ {-1}(\sigma(u(a)))) \\
        &= (u^{-1}(\sigma(u(a)))\cdot \lambda)(b), 
  \end{align}
and this means that
  \[ \label{eq:t3:x}
  u^{-1}(\sigma(u(a)))\cdot \lambda = u^\T(\sigma(u(a))\cdot u^{-\T}(\lambda)).
  \]
Finally, if $a$ is an element of~$A$, then for all $\lambda\in\D A$ we have
that
  \begin{alignat*}{3}
  \MoveEqLeft[1]
  u^{-1}(\sigma(u(a)))\cdot\lambda \\
        &= u^\T(\sigma(u(a))\cdot u^{-\T}(\lambda)) 
        &\qquad& \text{because of~\eqref{eq:t3:x}} \\
        &= u^\T(\beta(u(a)\cdot\beta^{-1}(u^{-\T}(\lambda)))) 
        && \text{because $\beta:A\to{}_\sigma\D A$ is left $A$-linear} \\
        &= u^\T(\beta(u(a\cdot u^{-1}(\beta^{-1}(u^{-\T}(\lambda)))))) \\
        &= \mathrlap{(u^ \T\circ\beta\circ u)
                     (a\cdot (u^{-1}\circ\beta^{-1}\circ u^{-\T})(\lambda))} \\
        &= \beta(\jac_\sigma(u)\cdot a\cdot (u\circ\beta^{-1}\circ u^\T)(\lambda)))) 
        && \text{because of~\eqref{eq:t1:x}} \\
        &= \beta(\jac_\sigma(u)\cdot a\cdot \jac_\sigma(u)^{-1}\cdot \beta^{-1}(\lambda)) 
        && \text{because of~\eqref{eq:t2:x}} \\
        &= \sigma(\jac_\sigma(u)\cdot a\cdot \jac_\sigma(u)^{-1})\cdot\lambda
        && \text{because $\beta:A\to{}_\sigma\D A$ is left $A$-linear,} 
  \end{alignat*}
and this implies at once that
  \(
  (u^{-1}\circ\sigma\circ u)(a) = \sigma(\jac_\sigma(u)\cdot a\cdot \jac_\sigma(u)^{-1})
  \).
\end{proof}

Lemma~\ref{lemma:det:conj} tells us, in a very explicit way, that every
automorphism conjugate to~$\sigma$ differs from~$\sigma$ only by an inner
automorphism, and we can rephrase this as follows.

\begin{Corollary}
The class of the Nakayama automorphism is central in the outer automorphism
group~$\Out(A)$. \qed
\end{Corollary}

We should keep in mind that according to Lemma~\ref{lemma:change} the class
of the Nakayama automorphism in~$\Out(A)$ is canonically determined by~$A$,
and this is what makes the statement in the corollary make sense. In the
same spirit, we have the following connection of our Jacobians with
commutativity relations in the group~$\Aut(A)$.

\begin{Corollary}\mbox{}\label{coro:det:commute}
\begin{thmlist}

\item An automorphism~$u:A\to A$ of~$A$ commutes with~$\sigma$ if and only
if its Nakayama Jacobian~$\jac_\sigma(u)$ is central in~$A$.

\item Two automorphisms $u$,~$v:A\to A$ are such that $u\sigma
u^{-1}=v\sigma v^{-1}$ if and only if the element
$\jac_\sigma(u)\cdot\jac_\sigma(v)^{-1}$ is central in~$A$.

\end{thmlist}
\end{Corollary}

\begin{proof}
Let $u:A\to A$ be an automorphism of~$A$. According to the
identity~\eqref{eq:detcom} of Lemma~\ref{lemma:det}, for all $a\in A$ we
have 
  \(
  \jac_\sigma(u)\cdot a = [\sigma^{-1},u^{-1}](a)\cdot \jac_\sigma(u)
  \).
It is clear then that if $\sigma$ and~$u$ commute the
element~$\jac_\sigma(u)$ is central. Conversely, if that element is central
we have that 
  \(
  a\cdot\jac_\sigma(u) = [\sigma^{-1},u^{-1}](a)\cdot \jac_\sigma(u)
  \)
for all~$a\in A$ and, since $\jac_\sigma(u)$ is a unit, that the
commutator~$[\sigma^{-1},u^{-1}]$ is the identity and thus that $\sigma$
and~$u$ commute. This proves the first part of the corollary. The second
one follows immediately from Lemma~\ref{lemma:det:conj}.
\end{proof}

An immediate consequence of the first part of this corollary is the
following:

\begin{Corollary}
If $\sigma=\id_A$, then the Nakayama Jacobian of every automorphism of~$A$
is central. \qed
\end{Corollary}

We will show below --- in Lemma~\ref{lemma:jac:quantum}, for example ---
that the Jacobians of an automorphism of a non-symmetric Frobenius algebra
may well be non-central.

\bigskip

Using the chain rule from Lemma~\ref{lemma:jac:cocycle} we can make a few
computations:

\begin{Lemma}\label{lemma:det:sigma}\mbox{}
\begin{thmlist}

\item For all $n\in\ZZ$ we have 
  \(
  \jac_\sigma(\sigma^n)=1
  \).

\item For all automorphisms~$u:A\to A$ we have that
  \[
  \jac_\sigma(\sigma \circ u) = \jac_\sigma(u),
  \qquad
  \jac_\sigma(u\circ \sigma) = \sigma^{-1}(\jac_\sigma(u)).
  \]

\end{thmlist}
\end{Lemma}

\begin{proof}
We have that
  \[
  \gen{a,b} 
        = \gen{b,\sigma(a)}
        = \gen{\sigma(a),\sigma(b)}
        = \gen{\jac_\sigma(\sigma)\cdot a,b}
  \]
for all $a$ and~$b$ in~$A$, and thus that $\jac_\sigma(\sigma)=1$. Using
this, Lemma~\ref{lemma:jac:cocycle}, and an obvious induction we can easily
see now that $\jac_\sigma(\sigma^n)=1$ for all non-negative integers~$n$.
Finally, using the last claim in Lemma~\ref{lemma:det:JC} we see that that
equality in fact holds for all integers~$n$. This proves~\thmitem{1}, and
\thmitem{2} follows immediately from~\thmitem{1} and
Lemma~\ref{lemma:jac:cocycle}.
\end{proof}

Apart from the powers of the Nakayama automorphism, the only elements
of~$\Aut(A)$ whose Nakayama Jacobian is easy to compute are the inner ones.

\begin{Lemma}\label{lemma:jac:inner}
Let $s$ be a unit in~$A$. The Nakayama Jacobian of the inner automorphism
$\iota_s:a\in A\mapsto sas^{-1}\in A$ is
  \[
  \jac_\sigma(\iota_s) = \sigma^{-1}(s^{-1})\cdot s.
  \]
\end{Lemma}

\begin{proof}
If $a$ and~$b$ are elements of~$B$, then
  \[
  \gen{\iota_s(a),\iota_2(b)}
        = \gen{sas^{-1},sas^{-1}}
        = \gen{\sigma^{-1}(s^{-1})s\cdot a,b},
  \]
and the claim of the lemma follows from this.
\end{proof}

Our Frobenius algebra~$A$ usually has many non-degenerate and associative
bilinear forms, each of which determines a Nakayama automorphism: the
following lemma describes the effect of this indeterminacy on
Nakayama Jacobians of automorphisms of~$A$.

\begin{Lemma}\label{lemma:det:coboundary}
Let $\gen{\place,\place}$,~$\gen{\place,\place}':A\times A\to\kk$ be two
non-degenerate and associative bilinear forms on~$A$, let
$\sigma$,~$\sigma':A\to A$ be corresponding Nakayama automorphisms, let $t$
be a unit of~$A$ such that $\gen{a,b}'=\gen{a,bt}$ for all~$a$ and all~$b$
in~$A$, as in Lemma~\ref{lemma:change}, and let us consider the unit
$\xi\coloneqq \sigma'^{-1}(t)$ of~$A$. For all automorphisms $u:A\to A$
of~$A$ we have that
  \[
  \jac_{\sigma'}(u) = \xi^{-1}\cdot\jac_\sigma(u)\cdot u^ {-1}(\xi).
  \]
\end{Lemma}

\begin{proof}
Let $u:A\to A$ be an automorphism of~$A$. For all $a$ and~$b$ in~$A$ we
have that
  \begin{align}
  \gen{\jac_{\sigma'}(u)\cdot a,b}'
       &= \gen{u(a),u(b)}'
        = \gen{u(a),u(b)\cdot t}
        = \gen{u(a),u(b\cdot u^{-1}(t))} \\
       &= \gen{\jac_{\sigma}(u)\cdot a,b\cdot u^{-1}(t)} 
        = \gen{\jac_{\sigma}(u)\cdot a,b\cdot u(t)\cdot t^{-1}}' \\
       &= \gen{\sigma'^{-1}(u^{-1}(t)\cdot t^{-1})\cdot \jac_{\sigma}(u)\cdot a,b}' ,
  \end{align}
and this implies that $\jac_{\sigma'}(u) = \sigma'^{-1}(u^{-1}(t)\cdot
t^{-1})\cdot \jac_{\sigma}(u)$. As a consequence of this and of the
identity~\eqref{eq:detcom} in Lemma~\ref{lemma:det} we see that
  \begin{align}
  \xi^{-1}\cdot\jac_\sigma(u)
        &= \sigma'^{-1}(t^{-1}) \cdot \jac_\sigma(u)
         = \sigma'^{-1}(u^{-1}(t^{-1}))\cdot\jac_{\sigma'}(u)
        \\
        &= \jac_{\sigma'}(u)\cdot u^{-1}(\sigma'^{-1}(t^{-1})))
         = \jac_{\sigma'}(u)\cdot u^{-1}(\xi^{-1}),
  \end{align}
and the equality in the lemma follows from this.
\end{proof}

We can present these results in cohomological terms. We will do this in
terms of non-abelian group cohomology. If we let $\Aut(A)$ act in the
obvious way on the group of units~$A^\times$, then we can compute the first
cohomology set $\H^1(\Aut(A),A^\times)$. Moreover, from the exact sequence
of groups
  \[
  \begin{tikzcd}
  1 \arrow[r]
    & \Inn(A) \arrow[r, hook]
    & \Aut(A) \arrow[r]
    & \Out(A) \arrow[r]
    & 1
  \end{tikzcd}
  \]
we obtain an exact «inflation-restriction» sequence of sets
  \[
  \begin{tikzcd}[column sep=2em]
  0 \arrow[r]
    & H^1(\Out(A),A^\times\cap\Z(A)) \arrow[r, "\inf"]
    & H^1(\Aut(A),A^\times) \arrow[r, "\res"]
    & H^1(\Inn(A),A^\times) 
  \end{tikzcd}
  \]
since the fixed subgroup $(A^\times)^{\Inn(A)}$ is
simply~$A^\times\cap\Z(A)$. Good references for what this is are
Jean-Pierre Serre's~\cite{Serre}*{Chapitre VII, Annexe} and James
Milne's~\cite{Milne}*{Chapter 27}.

\begin{Proposition}\mbox{}\label{prop:class:jac}
\begin{thmlist}

\item The function
  \[
  \jac^\sim_\sigma : u\in\Aut(A) \mapsto \jac_\sigma(u^{-1})\in A^\times
  \]
is a $1$-cocycle on the group~$\Aut(A)$ with values in the possibly
non-abelian group of units~$A^\times$ whose class~$\JAC(A)$
in~$\H^1(\Aut(A),A^\times)$ depends only on~$A$ and not on the choice of
the non-degenerate and associative bilinear form~$\gen{\place,\place}$ used
to compute~it.

\item The restriction of the class~$\JAC(A)$ to the subgroup~$\Inn(A)$ of
inner automorphisms of~$A$ is the trivial element
of~$\H^1(\Inn(A),A^\times)$ if and only if the algebra~$A$ is symmetric,
and when that is the case there is exactly one class~$\underline\JAC(A)$
in~$\H^1(\Out(A),A^\times\cap\Z(A))$ whose inflation to~$\Aut(A)$
is~$\JAC(A)$.

\end{thmlist}
\end{Proposition}

\begin{proof}
That the map~$\jac^\sim_\sigma$ is a $1$-cocycle is a direct consequence of
Lemma~\ref{lemma:jac:cocycle}. On the other hand, if $\gen{\place,\place}'$
is another non-degenerate and associative bilinear form on~$A$,
$\sigma':A\to A$ is the corresponding Nakayama automorphism, $t$ is the
unit of~$A$ such that $\gen{a,b}'=\gen{a,bt}$ for all~$a$ and~$b$ in~$A$,
and we put $\xi\coloneqq \sigma'^{-1}(t)$, then
Lemma~\ref{lemma:det:coboundary} implies immediately that for all
$u\in\Aut(A)$ we have $\jac^\sim_{\sigma'}(u) = \xi^{-1}\cdot
\jac^\sim_\sigma(u)\cdot u(\xi)$, and this means that the $1$-cocycles
$\jac^\sim_\sigma$,~$\jac^\sim_{\sigma'}:\Aut(A)\to A^\times$ are
cohomologous. Part~\thmitem{1} of the proposition follows from this.

The restriction of the cocycle~$\jac^\sim_\sigma$ to~$\Inn(A)$ is,
according to Lemma~\ref{lemma:jac:inner}, the unique function $f:\Inn(A)\to
A^\times$ such that $f(\iota_s)=\sigma(s)\cdot s^{-1}$ for all $s\in
A^\times$, and it is easy to see that there exists a unit~$\xi$ in~$A$ such
that $f(u)=\xi^{-1}\cdot u(\xi)$ for all~$u\in\Inn(A)$ if and only if the
Nakayama automorphism~$\sigma$ is inner, which happens exactly when the
algebra~$A$ is symmetric. Moreover, when that is the case the
class~$\JAC(A)$ in~$\H^1(\Aut(A),A^\times)$ is the inflation of exactly one
element of~$\H^1(\Out(A),A^\times\cap\Z(A))$ because of the exactness of
the inflation-restriction sequence. 
\end{proof}

This proposition provides us, for each Frobenius algebra~$A$, with a
canonical class~$\JAC(A)$ in~$\H^1(\Aut(A),A^\times)$, with the class
$\res\JAC(A)$ in~$\H^1(\Inn(A),A^\times)$, which is the sole obstruction
for the algebra to be symmetric, and, when the algebra is symmetric
algebra~$A$, with a lifting~$\underline\JAC(A)$ of that class
to~$\H^1(\Out(A),A^\times\cap\Z(A))$. It would be extremely interesting to
know what these invariants mean. 

We do not know, though. What we are able to do is to show that those
classes are not always trivial, so that there really is something to be
explained, and that in special examples they give something that is
actually interesting. We will present examples in the following pages with
those objectives.

\subsubsection{An example: the Grassmann algebra}
\label{subsubsect:jac:grass}

We suppose here that the characteristic of our ground field~$\kk$ is
not~$2$. We let $V$ be a finite-dimensional vector space, write~$n$ for its
the dimension, and choose an arbitrary ordered basis $\B=(x_1,\dots,x_n)$
for~$V$. We want to consider the exterior algebra $A\coloneqq\Lambda(V)$
on~$V$, and endow it with its usual $\ZZ$-grading, so
that~$A^i=\Lambda^i(V)$ for all~$i\in\ZZ$.

There is a unique linear form $\int:A\to\kk$ that vanishes on~$A^i$ for all
$i\in\inter{0,n-1}$ and has 
  \[
  \tint(x_1\wedge\cdots\wedge x_n)=1,
  \]
and from it we can construct a bilinear form $\gen{\place,\place}$ on~$A$
putting, for all $a$ and~$b$ in~$A$,
  \[
  \gen{a,b} \coloneqq \tint(a\wedge b).
  \]
This bilinear form is non-degenerate and associative, so $A$ is a Frobenius
algebra, and an easy calculation shows that the corresponding Nakayama
automorphism $\sigma:A\to A$ is the unique automorphism of~$A$ such that
$\sigma(x)=(-1)^{n-1}x$ for all~$x\in V$. In particular, the algebra~$A$ is
symmetric exactly when the integer~$n$ is odd, for $\sigma$ is not inner
if~$n$ es even, as one can easily check.

\bigskip

There is a unique `parity' automorphism $\mathsf{s}:A\to A$ such that
$\mathsf{s}(x)=-x$ for all~$x\in V$, it is diagonalizable, and we write
$A^\odd\coloneqq\{a\in A:\mathsf{s}(a)=-a\}$ and $A^\even\coloneqq\{a\in
A:\mathsf{s}(a)=a\}$ for its two eigenspaces. The subspace~$A^\even$ is a
subalgebra of~$A$ contained in its center. We can consider in~$\Aut(A)$ the
subgroup
  \[
  \Aut_\odd(A) \coloneqq \{u\in\Aut(A) : u(V)\subseteq A^\odd\}.
  \]
On the other hand, we say that a linear map~$\partial:A\to A$ is a
\newterm{left skew derivation}\label{def:lsder} if 
  \[
  \partial(a\cdot b)=\partial(a)\cdot b+(-1)^ia\cdot \partial(b)
  \]
whenever $i\in\inter{0,n}$, $a\in A^i$ and~$b\in A$, and it is easy to
check that there are well-determined left skew derivations
$\frac{\partial}{\partial x_1}$,~\dots,~$\frac{\partial}{\partial x_n}:A\to
A$ such that
  \[
  \frac{\partial}{\partial x_i}(x_j) = \delta_{i,j}
  \]
for all choices of~$i$ and~$j$ in~$\inter{n}$. If $u:A\to A$ is an element
of~$\Aut_\odd(A)$, then following Vladimir Bavula in~\cite{Bavula} we
define the \newterm{Jacobian} of~$u$ to be
  \[
  \jac(u) \coloneqq
      \det
      \left(
      \frac{\partial u(x_i)}{\partial x_j}
      \right) \in A^\even.
  \]
It is important to notice that this makes sense because the entries of the
matrix whose determinant we are taking are all in~$A^\even$ and therefore
commute. 

Bavula notes in~\cite{Bavula} that for all choices of $u$ and~$v$
in~$\Aut_\odd(A)$ we have that
  \[ \label{eq:jac:ch}
  \jac(u\circ v) = \jac(u)\cdot u(\jac(v)).
  \]
As $A^\even$ is contained in the center of~$A$, this implies that the
Jacobian of any element of~$\Aut_\odd(A)$ is an element of~$A^\times\cap
\Z(A)$, and that the map $\jac:\Aut_\odd(A)\to A^\times\cap\Z(A)$ that we
obtain in this way is a $1$-cocycle. We want to prove that this Jacobian
almost coincides with the restriction of the Nakayama Jacobian of~$A$
to~$\Aut_\odd(A)$.

\begin{Proposition}\label{prop:jacjac}
For all $u\in\Aut_\odd(A)$ we have $\jac^\sim_\sigma(u)=\jac(u)^{-1}$.
\end{Proposition}

This result justifies our decision to call the map~$\jac_\sigma$ the
Nakayama \emph{Jacobian}. One should recall that, in general,
$\jac^\sim_\sigma(u^{-1})$ and $\jac^\sim_\sigma(u)^{-1}$ are not equal.

\bigskip

Before proving this we need to recall some facts about the structure of the
group~$\Aut(A)$ from Bavula's paper~\cite{Bavula}.
\begin{itemize}

\item Let $P$ be the set of all $3$-element subsets of~$\inter{n}$, and
when $\alpha$ is an element of~$P$ and~$\alpha_1$,~$\alpha_2$
and~$\alpha_k$ are the three elements of~$\alpha$ listed in increasing
order let us put $x^\alpha\coloneqq x_{\alpha_1}\wedge x_{\alpha_2}\wedge
x_{\alpha_k}$. For each choice of~$i$ in~$\inter{n}$, of~$\lambda$
in~$\kk$, and of~$\alpha$ in~$P$, there is a unique automorphism
$\gamma_{i,\lambda,\alpha}:A\to A$ such that
  \[
  \gamma_{i,\lambda,\alpha}(x_j) 
        = \begin{cases*}
          x_i + \lambda x^\alpha & if $j=i$; \\
          x_j   & if $j\neq i$.
          \end{cases*}
  \]
According to \cite{Bavula}*{Theorem 3.11(1)} the set
  \(
  \{\gamma_{i,\lambda,\alpha}:i\in\inter{n},\lambda\in\kk,\alpha\in P\}
  \)
generates in~$\Aut(A)$ the subgroup
  \[
  \Gamma \coloneqq 
        \{ u\in\Aut(A): 
           \text{$u(x)\equiv x\mod A^\odd\cap\m^3$ for all $x\in V$}
         \}.
  \]

\item For each linear map~$f:V\to V$ in~$\GL(V)$ there is a unique
automorphism $\phi_f:A\to A$ such that $\phi_f(v)=f(v)$ for all~$x\in V$,
it is an element of the subgroup~$\Aut_\ZZ(A)$ of~$\Aut(A)$ of those
automorphisms that preserve the $\ZZ$-grading, and, in fact, the function
  \[ \label{eq:gliso}
  \phi:f\in\GL(V)\mapsto\phi_f\in\Aut_\ZZ(A)
  \]
that we obtain in this way is an isomorphism of groups. 

\item The subgroup~$\Aut_\ZZ(A)$ intersects~$\Gamma$ trivially and
normalizes it, and combining \cite{Bavula}*{Lemmas~2.9,~2.15(2),
and~2.16(1)} we can see that
  \[
  \Aut_\odd(A) = \Gamma\rtimes\Aut_\ZZ(A).
  \]

\end{itemize}

\begin{proof}[Proof of Proposition~\ref{prop:jacjac}] 
The function
  \(
  u\in\Aut_\odd(A) \mapsto \jac(u)^{-1} \in A
  \)
is a $1$-cocycle on the group~$\Aut_\odd(A)$: this follows from the «chain
rule»~\eqref{eq:jac:ch} that~$\jac$ satisfies and the fact that it takes
values in the center of~$A$. Since $\jac_\sigma^\sim:\Aut_\odd(A)\to A$ is
also a $1$-cocycle, to prove that $\jac^\sim_\sigma(u)=\jac(u)^{-1}$ for
all $u\in\Aut_\odd(A)$ it is enough then to check that that equality is
true when $u=\gamma_{i,\lambda,\alpha}$ for some choice of~$i\in\inter{n}$,
$\lambda\in\kk$ and~$\alpha\in P$, and when $u=\phi_f$ for some choice
of~$f\in\GL(V)$. This follows immediately from the fact that these elements
generate~$\Aut_\odd(A)$, as we observed above.

\medskip

Let $i$,~$\lambda$ and~$\alpha$ be elements of~$\inter{n}$,~$\kk$ and the
set~$P$, respectively, and let $\alpha_1$,~$\alpha_2$ and~$\alpha_3$ be the
three elements of~$\alpha$ listed in increasing order. We claim that
  \[ \label{eq:jac:gamma}
  \jac^\sim_\sigma(\gamma_{i,\lambda,\alpha})
    = \begin{cases*}
      1 & if $i\notin\alpha$; \\
      1 - \lambda\cdot x_{\alpha_2}\wedge x_{\alpha_3} & if $i = \alpha_1$; \\
      1 + \lambda\cdot x_{\alpha_1}\wedge x_{\alpha_3} & if $i = \alpha_2$; \\
      1 - \lambda\cdot x_{\alpha_1}\wedge x_{\alpha_3} & if $i = \alpha_3$.
      \end{cases*}
  \]
Let us write $d$ for the element of~$A$ appearing in right hand side of
this equality, which is an element of~$A^0+A^2$, and, for simplicity, let
us write~$\gamma$ instead of~$\gamma_{i,\lambda,\alpha}$ and remark that
for all $i\in\inter{0,n}$ we have that $\gamma^{-1}(A^i)\subseteq
A^i+A^{i+2}$. To prove our claim it is enough that we choose~$s$
in~$\inter{0,n}$ and~$r_1$,~\dots,~$r_s$ in~$\inter{n}$ such that
$r_1<\cdots<r_s$, put $a\coloneqq x_{r_1}\wedge\cdots\wedge x_{r_s}$, and
show that
  \[ \label{eq:ggg}
  \gen{1,\gamma^{-1}(a)} = \gen{1, d\wedge a}.
  \]
If $s\notin\{n-2,n\}$, then $\gamma^{-1}(a)$ and~$d\wedge a$ belong
to~$A^s+A^{s+2}$, which intersects trivially~$A^n$, and both sides of the
equality~\eqref{eq:ggg} are zero. If, on the other hand, $s=n$, then
$\gamma^{-1}(a)=a$ and $d\wedge a=a$, and that equality is just as clear.
Let us consider the case in which $s=n-2$.
\begin{itemize}

\item We suppose first that $i\notin\{r_1,\dots,r_{n-2}\}$, so that
$\gamma^{-1}(a)=a$: to prove~\eqref{eq:ggg} in this situation it is then
enough then we show that $d\wedge a=a$. This is obvious if $i\notin\alpha$,
for then $d=1$. Let us then suppose that $i\in\alpha$ and let $\beta_1$
and~$\beta_2$ be the elements of~$\alpha\setminus\{i\}$ indexed so that
$\beta_1<\beta_2$: we then have $d=1\pm\lambda\cdot x_{\beta_1}\wedge
x_{\beta_2}$ for some choice of the sign. The set
$\{\beta_1,\beta_2\}\cap\{r_1,\dots,r_{n-2}\}$ is not empty, for the sets
$\{\beta_1,\beta_2\}$ and~$\{r_1,\dots,r_{n-2}\}$ have cardinals~$2$
and~$n-2$, respectively, and their union has cardinal at most~$n-1$, since
it does not contain~$i$. It follows from this that
  \[
  d\wedge a = (1\pm\lambda\cdot x_{\beta_1}\wedge x_{\beta_2})
              \wedge(x_{r_1}\wedge\cdots\wedge x_{r_{n-2}})
            = x_{r_1}\wedge\cdots\wedge x_{r_{n-2}} = a,
  \]
as we want.

\item Let us suppose now that there is a~$j$ in~$\inter{n-2}$ such that
$r_j=i$, so that
  \[ \label{eq:seq:0}
  \gamma^{-1}(a) 
        = x_{r_1}\wedge \cdots\wedge x_{r_{j-1}}
          \wedge
          (x_i-\lambda\cdot x_{\alpha_1}\wedge x_{\alpha_2}\wedge x_{\alpha_3})
          \wedge
          x_{r_{j+1}}\wedge\cdots\wedge r_{n-2}.
  \]
There are two possibilities now.
\begin{itemize}

\item It may be the case that there are repeated elements in the sequence
  \[ \label{eq:seq}
  (r_1,\dots,r_{j-1},\alpha_1,\alpha_2,\alpha_3,r_{j+1},\dots,r_{n-2}).
  \]
It follows then from~\eqref{eq:seq:0} that $\gamma^{-1}(a)=a\in A^{n-2}$
and thus that $\gen{1,\gamma^{-1}(a)}=0$. If~$i\notin\alpha$, then
$\gen{1,d\wedge a}=0$ for $d=1$ and $a\in A^{n-2}$, and the
equality~\eqref{eq:ggg} holds. 

Suppose now that instead $i\in\alpha$ and, as before, let $\beta_1$
and~$\beta_2$ be the elements of~$\alpha\setminus\{i\}$ indexed so that
$\beta_1<\beta_2$. Since there are repeated elements in~\eqref{eq:seq} and
$i\notin\{r_1,\dots,r_{j-1},r_{j+1},\dots,r_{n-2}\}$, the sequence
  \(
  (\beta_1,\beta_2,r_1,\dots,r_{n-2})
  \)
also has repeated elements and therefore
  \[
  \gen{1,d\wedge a}
        = \gen{1, (1\pm\lambda\cdot x_{\beta_1}\wedge x_{\beta_2})
              \wedge(x_{r_1}\wedge\cdots\wedge x_{r_{n-2}})}
        = 0,
  \]
so that \eqref{eq:ggg} also holds in this situation.

\item It may be the case that there are no repeated elements in the
sequence~\eqref{eq:ggg}. Let $\epsilon$ be the signature of the permutation
that sorts it in increasing order. In view of~\eqref{eq:seq:0} we now have
that
  \[
  \gen{1,\gamma^{-1}(a)} = -\epsilon\lambda.
  \]
Since the sequence~\eqref{eq:seq} has length~$n$ and $r_j=i$, we see that
$i\in\alpha$. Once more, let $\beta_1$ and~$\beta_2$ be the elements
of~$\alpha\setminus\{i\}$ indexed so that $\beta_1<\beta_2$, and let $\eta$
be the signature of the permutation that
takes~$(\alpha_1,\alpha_2,\alpha_3)$ to~$(i,\beta_1,\beta_2)$. We have
$d=1-\eta\lambda\cdot x_{\beta_1}\wedge x_{\beta_2}$ and
  \[
  \gen{1,d\wedge a}
        = \gen{1, (1-\eta\lambda\cdot x_{\beta_1}\wedge x_{\beta_2})
              \wedge(x_{r_1}\wedge\cdots\wedge x_{r_{n-2}})}
        = -\eta\eta'\lambda,
  \]
with $\eta'$ the signature of the permutation that
sorts~$(\beta_1,\beta_2,r_1,\dots,r_{n-2})$. As $\epsilon\eta\eta'=1$, the
equality~\eqref{eq:ggg} also holds in this case.

\end{itemize}
\end{itemize}
This completes the proof of our claim of the equality~\eqref{eq:jac:gamma}.

Next, we compute Bavula's Jacobian~$\jac(\gamma)$. If $k$,~$l\in\inter{n}$
are such that $k\neq 0$, then $\gamma(x_k)=x_k$, and this implies that
  \[
  \frac{\partial \gamma(x_k)}{\partial x_l} = \delta_{k,l}.
  \]
We see with this that the matrix
  \(
  \left(\frac{\partial \gamma(x_k)}{\partial x_l}\right)
  \)
coincides with the identity matrix except possibly along its $i$th row, so
that its determinant is
  \[
  \jac(\gamma) 
        = \det \left(\frac{\partial \gamma(x_k)}{\partial x_l}\right)
        = \frac{\partial \gamma(x_i)}{\partial x_i}
        = \begin{cases*}
          1 & if $i\notin\alpha$; \\
          1 + \lambda\cdot x_{\alpha_2}\wedge x_{\alpha_3} & if $i = \alpha_1$; \\
          1 - \lambda\cdot x_{\alpha_1}\wedge x_{\alpha_3} & if $i = \alpha_2$; \\
          1 + \lambda\cdot x_{\alpha_1}\wedge x_{\alpha_3} & if $i = \alpha_3$.
          \end{cases*}
  \]
It is the clear that $\jac(\gamma)^{-1}=d=\jac^\sim_\sigma(\gamma)$, as we want.

\medskip

Let now $f:V\to V$ be an element of~$\GL(V)$, let $(f_{i,j})\in\GL_n(\kk)$
be the matrix of~$f$ with respect to the ordered basis~$\B$ of~$V$, so that
$f(x_j)=\sum_{i=1}^nf_{i,j}x_i$ for all~$j\in\inter{n}$, and let
$\phi_f:A\to A$ be the automorphism of~$A$ corresponding to~$f$ by the
isomorphism~\eqref{eq:gliso}, which respects the $\ZZ$-grading on~$A$. We
want to show that $\jac^\sim_\sigma(\phi_d)=\det(f)^{-1}$ or, what is the
same, that for all $a\in A$ we have
  \(
  \gen{1,\phi_f^{-1}(a)} = \gen{\det(f)^{-1},a}
  \).
This is clear if there is an element~$l$ in~$\inter{0,n-1}$ such that $a\in
A^l$, for in that case then $\phi_f^{-1}(a)$ is also in~$A^l$ and the
scalars~$\gen{1,\phi_f^{-1}(a)}$ and~$\gen{\det(f),a}$ are both zero. If
instead $a$ is a multiple of~$x_1\wedge\cdots\wedge x_n$, then of course
$\phi_f^{-1}(a)=\det(f)^{-1}\cdot a$ and the equality also holds.

On the other hand, we clearly have that 
  \(
  \frac{\partial \phi_f(x_i)}{\partial x_j} = f_{i,j}
  \)
for all $i$ and all~$j$ in~$\inter{n}$, so that
  \[
  \jac(\phi_f) 
        = \det\left(\frac{\partial \phi_f(x_i)}{\partial x_j}\right) 
        = \det(f_{i,j})
        = \det(f).
  \]
We see that $\jac^\sim_\sigma(\phi_f)=\jac(\phi_f)^{-1}$. This completes
the proof of the proposition. \end{proof}

An immediate corollary of this proposition and its proof is that our
cohomological invariant is not trivial:

\begin{Corollary}
The class~$\JAC(A)$ in~$\H^1(\Aut(A),A^\times)$ is not trivial.
\end{Corollary}

\begin{proof}
To prove this it is enough that we show that the restriction of the
$1$-cocycle $\jac^\sim_\sigma:\Aut(A)\to A^\times$ to the
subgroup~$\Aut_\ZZ(A)$ of automorphisms of~$A$ that preserve the
$\ZZ$-grading is not a coboundary. In the proof of
Proposition~\ref{prop:jacjac} we saw that, up to the
isomorphism~\eqref{eq:gliso}, that restriction is the map
$f\in\GL(V)\mapsto\det(f)^{-1}\in A^\times$. If this restriction was a
coboundary, we would have a unit~$\xi\in A^\times$ such that
$\det(f)^{-1}=\xi^{-1}\cdot\phi_f(\xi)$ for all~$f\in\GL(V)$. If we let
$\xi_0$ be the «constant term» of~$\xi$, which is a non-zero scalar, then
this tells us that $\det(f)^{-1}\xi_0=\xi_0$ for all~$f\in\GL(V)$, and this
is absurd: since the characteristic of~$\kk$ is not~$2$, there are elements
in~$\GL(V)$ whose determinant is not~$1$.
\end{proof}

Proposition~\ref{prop:jacjac} tells us what the restriction
of~$\jac^\sim_\sigma$ to~$\Aut_\odd(A)$ is: as this is a proper subgroup
of~$\Aut(A)$, there is more to be said. According to
\cite{Bavula}*{Corollary~2.15}, the
subgroup~$\Aut_\odd(A)=\Gamma\rtimes\Aut_\ZZ(A)$ intersects
trivially~$\Inn(A)$ and we have that, in fact,
  \[
  \Aut(A) = \Inn(A) \rtimes \Aut_\odd(A).
  \]
Moreover, \cite{Bavula}*{Lemmas 2.8 and~2.9} tells us that the function
  \[
  a\in A^\odd\mapsto\iota_{1+a}\in\Inn(A)
  \]
is a surjective morphism of groups whose kernel is $A^\odd\cap A^n$ and
that therefore induces an isomorphism $A^\odd/A^\odd\cap A^n\to\Inn(A)$.
Using this information we can easily compute $\jac^\sim_\sigma$ on the
subgroup~$\Inn(A)$:

\begin{Lemma}
If $a$ is an element of~$A^\odd$, then
  \[
  \jac^\sim_\sigma(\iota_{1+a})
        = \begin{cases*}
          1-2a & if $n$ is even; \\
          1 & if $n$ is odd.
          \end{cases*}
  \]
\end{Lemma}

When $n$ is odd, the algebra~$A$ is symmetric, and we already know
from~Lemma~\ref{lemma:jac:inner} that $\jac^\sim_\sigma$ is necessarily
trivial on~$\Inn(A)$.

\begin{proof}
Let $a$ be an element of~$A^\odd$. We have that
$\sigma^{-1}(a)=(-1)^{n-1}a$, and that and Lemma~\ref{lemma:jac:inner} let
us compute that
  \begin{align}
  \jac^\sim_\sigma(\iota_{1+a})
       &= \jac_\sigma(\iota_{1+a}^{-1})
        = \jac_\sigma(\iota_{1-a})
        = \sigma^{-1}((1-a)^{-1})\cdot (1-a) 
        = \sigma^{-1}(1+a)\cdot (1-a) \\
       &= (1+(-1)^{n-1}a)\cdot (1-a) 
        = 1 + ((-1)^{n-1}-1)a.
  \end{align}
This is what the lemma claims.
\end{proof}

Let us remark that Bavula defines in~\cite{Bavula} the Jacobian of an
automorphism $u:A\to A$ only when it maps~$V$ to~$A^\odd$: as we said
above, this ensures that the matrix $\left(\frac{\partial
\phi_f(x_i)}{\partial x_j}\right)$ has all its entries in the commutative
ring~$A^\even$ and that therefore we can compute its determinant as usual.
It is a non-trivial fact established in \cite{Bavula}*{Section 2} that the
subset
  \[
  \Aut_\even(A) \coloneqq \{u\in\Aut(A):u(V)\subseteq V+A^\even\}
  \]
of~$\Aut(A)$ is actually a subgroup, and that
  \begin{gather}
  \Aut(A) = \Aut_\odd(A)\cdot\Aut_\even(A)=\Aut_\even(A)\cdot\Aut_\odd(A), \\
  \Aut_\ZZ(A) = \Aut_\odd(A)\cdot\Aut_\even(A),
\shortintertext{and}
  \Aut_\even(A) = \Inn(A)\rtimes\Aut_\ZZ(A).
  \end{gather}
This shows that the only part of~$\Aut(A)$ on which Bavula does not define
his Jacobian is~$\Inn(A)$. Up to the inversion that showed up in the
statement of Proposition~\ref{prop:jacjac}, what we have seen here is that
the twisted Nakayama Jacobian~$\jac^\sim_\sigma$ gives a good global
definition on the whole automorphism group.

\subsubsection{An example: trivial extensions}
\label{subsubsect:jac:trivial}

Let $B$ be a finite dimensional algebra, let $\D B\coloneqq\Hom(B,\kk)$ be
its dual space viewed as a $B$-bimodule in the usual way, and let
$A\coloneqq B\oplus\D{B}$ be the \newterm{trivial extension} of~$B$ by~$\D
B$, whose product is such that
  \[
  (x,x')\cdot(y,y') = (xy,xy'+x'y)
  \]
for all $(x,x')$,~$(y,y')\in A$. The bilinear
form~$\gen{\place,\place}:A\times A\to\kk$ that has
  \[
  \gen{(x,x'),(y,y')} = x'(y) + y'(x)
  \]
whenever $(x,x')$ and~$(y,y')$ are elements of~$A$ is non-degenerate,
associative and symmetric, so the algebra~$A$ is symmetric. 

Let $u:A\to A$ be an automorphism of~$A$. There are linear maps $a:B\to B$,
$b:\D B\to B$, $c:B\to\D B$ and~$d:\D B\to\D B$ such that
  \[
  u = \begin{pmatrix}
      a & b \\
      c & d
      \end{pmatrix}
      : B\oplus\D B \to B\oplus \D B.
  \]
Since the form on~$A$ is symmetric, the Jacobian of~$u$ is a central unit
of~$A$, so it can be written as $t+\tau$ with $t\in B^\times\cap\Z(B)$ and
$\tau\in\D{B}^B$. The condition that
$\gen{u(1),u(x)}=\gen{\jac_\sigma(u),x}$ for all $x\in A$ implies at once
that $t$ is the unique element of~$B$ such that
  \[ \label{eq:jtv:1}
  x'(t) = d(x')(1) \quad\text{for all $x'\in\D B$,}
  \]
and that the map~$\tau:B\to\kk$ is such that
  \[ \label{eq:jtv:2}
  \tau(x) = c(x)(1) \quad\text{for all $x\in B$.}
  \]
That $\tau$ belongs to~$\D{B}^B\coloneqq\{\lambda\in\D
B:\text{$x\lambda=\lambda x$ for all $x\in B$}\}$ implies immediately that
the map~$\tau$ vanishes on commutators of~$B$, and we can therefore view
$\tau$ as the linear map 
  \[
  x+[B,B]\in \HH_0(B)\coloneqq B/[B,B] \mapsto \tau(x)\in \kk.
  \]

It is interesting to notice that not every central unit of~$A$ is the
Jacobian of an automorphism of~$A$:

\begin{Lemma}
A central unit of~$A$ is the Jacobian of an automorphism of~$A$ if and only
if it is of the form $t+\tau$ with $t$ a central unit of~$B$ and
$\tau:B\to\kk$ an element of the image of the transposed map
  \[
  \CB^\T : \HH_1(B)^* \to \HH_0(B)^*
  \]
of the Connes boundary $\CB:\HH_0(B)\to\HH_1(B)$.
\end{Lemma}

Bavula spends a significant part of his paper~\cite{Bavula} studying the
problem of describing the image of his Grassmannian Jacobian function: this
lemma corresponds to that in the case of trivial extensions. It would be
very interesting to understand why exactly the Connes boundary map appears
in this context.

\begin{proof}
Let $u:A\to A$ be an automorphism, and let  $a:B\to B$, $b:\D B\to B$,
$c:B\to\D B$ and~$d:\D B\to\D B$ be the linear maps such that
  \(
  u = \begin{psmallmatrix}
      a & b \\
      c & d
      \end{psmallmatrix}
      : B\oplus\D B \to B\oplus \D B
  \).
The map~$a$ is a morphism of algebras and $c(xy) = a(x)c(y) + c(x)a(y)$ for
all $x$,~$y\in B$. 

As we noted above, the Jacobian of~$u$ is of the
form~$\jac_\sigma(u)=t+\tau$, with $t$ a central unit of~$B$ and
$\tau:B\to\kk$ the element of~$\D{B}^B$ such that $\tau(x)=c(x)(1)$ for all
$x\in B$. A calculation shows that the map 
  \[
  \bar c:x\otimes y\in B\otimes B\mapsto c(y)(a(x))\in\kk
  \]
has the property that
  \[
  \bar c(xy\otimes z - x\otimes yz + zx\otimes y) = 0
  \]
for all $x$,~$y$,~$z\in B$, and this means that $\bar c$ vanishes on
Hochschild $1$-boundaries and thus that its restriction to Hochschild
$1$-cycles induces a map $\hat c:\HH_1(B)\to\kk$, that is, an element
of~$\HH_1(B)^*$. On the other hand, according to~\cite{Loday}*{Eq.~2.1.7.3}
the Connes coboundary $\CB:\HH_0(B)\to\HH_1(B)$ is induced on homology by
the map
  \[
  x\in B \mapsto 1\otimes x + x\otimes 1\in B\otimes B,
  \]
so the map $\CB^\T(\hat c):\HH_0(B)\to\kk$ is induced on~$\HH_0(B)$ by the
map
  \[
  x\in B\mapsto\bar c(\CB(x)) = c(x)(a(1))+c(1)(a(x))=\tau(x)\in \kk,
  \]
which is, of course, just the linear map~$\tau$ with which we started ---
here we have used the fact that 
  \(
  c(1) = c(1\cdot1) = a(1)c(1)+c(1)c(1) = c(1)+c(1)
  \)
because~$a(1)=1$, and thus that $c(1)=0$.  This shows that the condition
given by the lemma on a central unit of~$A$ is necessary for it to be the
Jacobian of an automorphism.

\bigskip

Let us next show that that condition is also sufficient. Let $t$ be a
central unit of~$B$, let $\tau:B\to\kk$ be a linear map that vanishes on
commutators, and let us suppose that there is a map $\hat c:\HH_1(B)\to\kk$
such that $\CB^\T(\hat c)=\tau$. If we write $P_*$ for the standard
bimodule resolution of~$B$ as a $B$-bimodule, so that, for example, the
Hochschild homology $\HH_*(B)$ is canonically the homology of the
complex~$B\otimes_{B^e}P_*$, then we have adjunction isomorphisms
  \[
  \Hom_\kk(B\otimes_{B^e}P_*,\kk)
        \cong \Hom_{B^e}(P_*,\Hom_\kk(B,\kk))
        = \Hom_{B^e}(P_*,\D B)
  \]
that induce, since the functor~$\Hom_\kk(\place,\kk)$ is exact, an
isomorphism 
  \[
  \Hom(\HH_*(B),\kk)\cong\H^*(B,\D B).
  \]
There is therefore a Hochschild $1$-cocycle $\delta:B\to \D B$ or,
equivalently, a derivation, whose class in~$\H^1(B,\D B)$ corresponds to
the map~$\hat c:\HH_1(B)\to\kk$ via this isomorphism. Explicitly, this
means the following: from the derivation~$\delta$ we can construct the map
$\delta':x\otimes y\in B\otimes B\mapsto\delta(y)(x)\in\kk$, which vanishes
on Hochschild $1$-boundaries, so its restriction to $1$-cycles induces a
map $\HH_1(B)\to\kk$, and this map is precisely~$\hat c$. It follows from
this and the hypothesis that $\tau=\CB^\T(\hat c)$ that for each $x\in B$
we have 
  \[ \label{eq:jtv:3}
  \tau(x) 
    = \hat c(\CB(x))
    = \delta'(1\otimes x+x\otimes 1)
    = \delta(x)(1).
  \]
Let now $m_t:x\in B\mapsto tx\in B$ be the map given by multiplication by
the central unit~$t$, let $m_t^\T:\D B\to\D B$ be its transpose, and let us
consider the linear map
  \[
  u \coloneqq \begin{pmatrix}
              \id_B & 0 \\
              \delta & m_t^\T
              \end{pmatrix}
        : B\oplus\D B\to B\oplus\D B.
  \]
A direct calculation shows that this is an automorphism of the trivial
extension~$A$, and according to the equations~\eqref{eq:jtv:1},
\eqref{eq:jtv:2} and~\eqref{eq:jtv:3} the Jacobian of~$u$ is
$\jac_\sigma(u)=t+\tau$.
\end{proof}

Almost by definition --- see Max Karoubi's \cite{Karoubi}*{Théorème 2.15}
--- the degree zero non-commutative de Rham cohomology~$\H_\dR^0(B)$ of the
algebra~$B$ fits in an exact sequence of the form
  \[
  \begin{tikzcd}
  0 \arrow[r]
    & \H_\dR^0(B) \arrow[r]
    & \HH_0(B) \arrow[r, "\CB"]
    & \HH_1(B)
  \end{tikzcd}
  \]
The lemma we have just proved therefore tells us that the only obstruction
for a linear map $\tau:B\to\kk$ which vanishes on commutators to appear as
the Jacobian of an automorphism of the trivial extension~$A$ is its class
in the dual space~$\H^0_\dR(B)^*$. This is reminiscent of the beautiful
theorem of Jürgen Moser~\cite{Moser} that says that two volume forms on
a compact orientable manifold are related by a self-diffeomorphism of the
manifold exactly when they have the same integral.

\bigskip

We want to show that the class~$\JAC(A)$ in~$\H^1(\Aut(A),A^\times)$ is
most often non-trivial. It turns out that the automorphism group of a
trivial extension can be rather complicated --- this can be seen from the
results of Takayoshi Wakamatsu in~\cite{Wakamatsu} ---  so to do that we
will show that the restriction of~$\JAC(A)$ to a simpler subgroup is
non-trivial, as that will be enough. For each central unit~$z$ of~$B$, the
map
  \[
  u_z \coloneqq \begin{pmatrix}
                \id B & 0 \\
                0 & m_z^\T
                \end{pmatrix}
                : B\oplus\D B\to B\oplus\D B
  \]
is an automorphism of~$A$, with~$m_z^\T$ the transpose of the map $m_z:x\in
B\mapsto zx\in B$, and the function
  \[
  z\in B^\times\cap\Z(B)\mapsto u_z\in\Aut(A)
  \]
is an injective morphism of groups, which we will regard, for convenience,
as an inclusion. 

\begin{Lemma}\label{lemma:ta:class:1}
The restriction of~$\JAC(A)\in\H^1(\Aut(A),A^\times)$ to a class
in~$\H^1(B^\times\cap\Z(B),A^\times)$ is trivial if and only if the
group~$B^\times\cap\Z(B)$ is itself trivial.
\end{Lemma}

Notice that if the group~$B^\times\cap\Z(B)$ is trivial the ground
field~$\kk$ is necessarily~$\FF_2$, the field with two elements, because
that intersection contains a copy of the multiplicative group~$\kk^\times$.

\begin{proof}
A direct calculation shows that for all $z\in B^\times\cap\Z(B)$ we have
  \[
  \jac^\sim_\sigma(u_z) = z^{-1}.
  \]
If the restriction of~$\jac^\sim_\sigma:\Aut(A)\to A^\times$
to~$B^\times\cap\Z(B)$ is a $1$-coboundary, so that there is a
unit~$a+\alpha$ in~$A=B\oplus \D B$ such that $z^{-1} =
(a+\alpha)^{-1}\cdot u_z(a+\alpha)$ for all $z\in B^\times\cap\Z(B)$, then
computing explicitly the right hand side of this equality we see that
$B^\times\cap\Z(B)$ has exactly one element, the unit element~$1$ of~$B$.
Conversely, if that intersection is trivial then of course any class
in~$\H^1(B^\times\cap\Z(B),A^\times)$ is trivial.
\end{proof}

The conclusion of this lemma is somewhat strange: for \emph{any}
finite-dimensional algebra~$B$ we have constructed an element
of~$\H^1(B^\times\cap\Z(B),(B\oplus\D B)^\times)$ that is in general not
zero. We can get rid of the dual space here, if we want, by projecting that
class into $\H^1(B^\times\cap\Z(B),B^\times)$ There is a similar
construction which produces another invariant attached to any
finite-dimensional algebra. We explain this next.

\bigskip

For each derivation $\delta:B\to\D B$ the map
  \[
  u_\delta \coloneqq \begin{pmatrix}
                     \id_B & 0 \\
                     \delta & \id_{\D B}
                     \end{pmatrix}
                     : B\oplus \D B\to B\oplus \D B
  \]
is an automorphism of the trivial extension~$A$, and the function
  \[
  \delta\in\Der(B,\D B) \mapsto u_\delta\in\Aut(A)
  \]
is an injective morphism of groups. Viewing it as an inclusion, we can
restrict the cocycle $\jac^\sim_\sigma:\Aut(A)\to A^\times$ to its image
and consider its cohomology class in
  \[ \label{eq:weird}
  \H^1(\Der(B,\D B),A^\times).
  \]
Since the algebra~$A$ is symmetric, we know the Jacobian of any inner
automorphism is~$1$, and one can check that the intersection of~$\Der(B,\D
B)$ with the subgroup of~$\Aut(A)$ of inner automorphisms is precisely
$\InnDer(B,\D B)$, so the class we obtained is in fact the inflation of a
class in~$\H^1(\H^1(B,\D B),A^\times\cap\Z(A))$ --- the outer~$\H^1$ here
is non-abelian group cohomology, and the inner~$\H^1$ is the Hochschild
cohomology of~$B$ with values in~$\D B$. There is a canonical isomorphism
of groups $\H^1(B,\D B)\cong\HH_1(B)^*$, so we can also say that we have a
class in
  \[ \label{eq:weird:2}
  \H^1(\HH_1(B)^*,A^\times\cap\Z(A)).
  \]

In contrast with what we saw in Lemma~\ref{lemma:ta:class:1}, the classes
in~\eqref{eq:weird} and in~\eqref{eq:weird:2} can be trivial in quite
general situations. For example, if $B$ is an admissible quotient of the
path algebra on a quiver without oriented cycles, then we know from the
theorem in Claude Cibils' \cite{Cibils} that~$\HH_1(B)=0$, so that the
classes in~\eqref{eq:weird:2} and in~\eqref{eq:weird} are of course
trivial.

We do not know how to characterize the triviality of these classes. The
following is just what we get from the definition:

\begin{Lemma}
The class in $\H^1(\Der(B,\D B),A^\times)$ is trivial if and only if there
exists a central unit~$a$ in~$B$ such that
  \(
  \delta(b)(1) = \delta(a)(ba^{-1})
  \)
for all derivations~$\delta:B\to\D B$ and all elements~$b$ of~$B$. \qed
\end{Lemma}

This condition does not seem to be easy to satisfy.

\bigskip

Let us finish this discussion with a remark. If $\tau:B\to B$ is an
automorphism of~$B$, then we can consider the twisted dual bimodule
${}_\tau\D B$, and the corresponding trivial extension $A\coloneqq
B\oplus{}_\tau\D B$. This is again a Frobenius algebra, with respect to the
bilinear form $\gen{\place,\place}$ on it such that
  \[
  \gen{(x,x'),(y,y')} = x'(y) + y'(\tau(x))
  \]
for all $x$,~$y\in B$ and all~$x'$,~$y'\in\D B$, and the corresponding
Nakayama automorphism is 
  \[
  \sigma:
        (x,x')\in B\oplus{}_\tau\D B 
        \mapsto 
        (\tau(x),\tau^{-\T}(x'))\in B\oplus{}_\tau\D B,
  \]
which is usually not inner --- if it is inner, for example, then so is
$\tau$ itself --- so that the algebra $A$ is usually not symmetric. We can
repeat the arguments we presented here for trivial extensions to these
twisted trivial extensions, and thus obtain cohomological invariants now
not for the algebra~$B$ but for the pair~$(B,\tau)$.

\subsubsection{An example: the quantum complete intersection of
dimension \texorpdfstring{$4$}{four}}
\label{subsubsect:jac:quantum}

Let $q$ be a non-zero element of~$\kk$, let~$A$ be the algebra of
Example~\ref{ex:four}, and let $\{1,x,y,xy\}$, $\gen{\place,\place}$
and~$\sigma$ be as there. If $a$,~$b$,~$c$,~$d$ are elements of~$\kk$ such
that $a\neq0$ and $b\neq0$, then there is an algebra automorphism
  \(
  \alpha(a,b,c,d) : A\to A
  \)
such that 
  \[
  \alpha(a,b,c,d)(x) = ax+cxy, 
  \qquad
  \alpha(a,b,c,d)(y) = by+dxy.
  \]
A little calculation shows that
  \begin{gather}
  \alpha(a,b,c,d) \circ \alpha(a',b',c',d') 
        = \alpha(aa',bb',ca'+abc',db'+abd') \label{eq:faut0}
\shortintertext{and}
  \alpha(a,b,c,d)^{-1} 
        = \alpha(a^{-1},b^{-1};-a^{-2}b^{-1}c,-a^{-1}b^{-2}d)
  \end{gather}
whenever $a$,~$b\in\kk^\times$ and $c$,~$d\in\kk$, and it follows from this
that the image~$\Aut_0(A)$ of the injective function
  \[
  \alpha:(a,b,c,d)\in\kk^\times\times\kk^\times\times\kk\times\kk
  \mapsto \alpha(a,b,c,d)\in\Aut(A)
  \]
is a subgroup of~$\Aut(A)$. We are interested in this subgroup because of
the following observation.

\begin{Lemma}
If $q^2\neq1$, then $\Aut(A)=\Aut(A)_0$.
\end{Lemma}

Let us remark that when $q^2=1\neq-1$ the subgroup~$\Aut_0(A)$ has
index~$2$ in~$\Aut(A)$, and that when $q^2=1=-1$ then it has, in an
appropriate sense, codimension~$2$ there.

\begin{proof}
Let $f:A\to A$ be an automorphism of~$A$. Since $f(\rad A)=\rad A$ and
$\rad A$ is freely spanned by~$\{x,y,xy\}$, there are scalars
$a_1$,~$a_2$,~$a_3$ and $b_1$,~$b_2$,~$b_3$ in~$\kk$ such that
$f(x)=a_1x+a_2y+a_3xy$ and $f(y)=b_1x+b_2y+b_3xy$. As $x^2=0$, $y^2=0$ and
$yx-qxy=0$, we have that $(1+q)a_1a_2=0$, $(1+q)b_1b_2=0$,
and~$(1-q^2)a_2b_1=0$. On the other hand, since $f$ induces an automorphism
of the vector space~$\rad A/\rad^2A$, we have that
$\det\begin{psmallmatrix}a_1&b_1\\a_2&b_2\end{psmallmatrix}\neq0$. When
$q^2\neq1$ all this implies that $a_2=0$, $b_1=0$ and therefore that
$f=\alpha(a_1,b_2,a_3,b_3)$.
\end{proof}

We can describe this subgroup $\Aut_0(A)$ in two more familiar ways. There
are a left action~$\smalltriangleright$ and a right action
$\smalltriangleleft$ of the group~$\kk^\times\times\kk^\times$
on~$\kk\times\kk$ by group morphisms such that
  \[
  (a,b) \smalltriangleright (c,d) = (abc,abd),
  \qquad
  (c,d) \smalltriangleleft (a,b) = (ca,db)
  \]
whenever $(a,b)\in\kk^\times\times\kk^\times$ and~$(c,d)\in\kk\times\kk$,
and there is a group
$(\kk^\times\times\kk^\times)\diamondvert(\kk\times\kk)$ with underlying
set the cartesian product
$(\kk^\times\times\kk^\times)\times(\kk\times\kk)$ and multiplication given
by
  \[
  (g,u)\cdot(h,v) = (gh,g\smalltriangleright v+u\smalltriangleleft h)
  \]
whenever~$(g,u)$ and~$(h,v)$ are two of its elements; the key fact in
proving this is that the equality $(g\smalltriangleright
u)\smalltriangleleft h=g\smalltriangleright(u\smalltriangleleft h)$ holds
for all $g$,~$h\in \kk^\times\times\kk^\times$ and all~$u\in\kk\times\kk$.
Comparing with~\eqref{eq:faut0} it is clear that the function
  \[
  ((a,b),(c,d))\in(\kk^\times\times\kk^\times)\diamondvert(\kk\times\kk)
  \mapsto
  \alpha(a,b,c,d)\in\Aut(A)_0
  \]
is an isomorphism of groups. The map
  \[
  \mathord{\blacktriangleleft}:
        (u,g)\in(\kk\times\kk)\times(\kk^\times\times\kk^\times)
        \mapsto 
        g^{-1}\smalltriangleright u\smalltriangleleft g\in\kk\times\kk
  \]
is a right action of~$\kk^\times\times\kk^\times$ on~$\kk\times\kk$ by
group automorphisms, so we can construct the semidirect product
$(\kk^\times\times\kk^\times)\ltimes(\kk\times\kk)$ in which
  \[
  (g,u)\cdot(h,v) = (gh,u\blacktriangleleft h+v)
  \]
whenever $g$,~$h\in\kk^\times\times\kk^\times$ and~$u$,~$v\in\kk\times\kk$,
and we have an isomorphism
  \[
  (g,u) \in (\kk^\times\times\kk^\times)\ltimes(\kk\times\kk) 
  \mapsto 
  (g, g\smalltriangleright u) \in (\kk^\times\times\kk^\times)\diamondvert(\kk\times\kk).
  \]

Next we compute Jacobians.

\begin{Lemma}\label{lemma:jac:quantum}
Let $a$,~$b$,~$c$ and~$d$ be scalars such that $a\neq0$ and~$b\neq0$. The
Jacobian of the automorphism $\alpha(a,b,c,d):A\to A$ is
  \[
  \jac_\sigma(\alpha(a,b,c,d)) = ab+dx+q^{-1}cy.
  \]
\end{Lemma}

\begin{proof}
A direct calculation shows that
$\gen{1,\alpha(a,b,c,d)(z)}=\gen{ab+dx+q^{-1}cy,z}$ when $z$ is an element
of $\{1,x,y,xy\}$, and that proves the lemma.
\end{proof}

This shows that the Jacobian of automorphisms of non-symmetric algebras can
well be non-central. Moreover, using this we see at once that the
class~$\JAC(A)$ is not trivial.

\begin{Lemma}
The image of the class~$\JAC(A)$ under the composition
  \[
  \begin{tikzcd}[column sep=1em]
  \H^1(\Aut(A),A^\times) \arrow[r]
        & \H^1(\Aut_0(A),A^\times) \arrow[r]
        & \H^1(\Aut_0(A),\kk^\times) \arrow[r, equal]
        & \Hom(\kk^\times\times\kk^\times,\kk^\times)
  \end{tikzcd}
  \]
is the morphism $(a,b)\in\kk^\times\times\kk^\times\mapsto
a^{-1}b^{-1}\in\kk^\times$ and, in particular, the class~$\JAC(A)$ itself
is not trivial. \qed
\end{Lemma}

Here the map $\H^1(\Aut(A),A^\times)\to\H^1(\Aut_0(A),A^\times)$ is the
restriction, and the map
$\H^1(\Aut_0(A),A^\times)\to\H^1(\Aut_0(A),\kk^\times)$ is induced by the
projection $A^\times\to\kk^\times$. The action of~$\Aut_0(A)$
on~$\kk^\times$ is trivial, so we can identify~$\H^1(\Aut_0(A),\kk^\times)$
with $\Hom(\Aut_0(A),\kk^\times)$, and this $\Hom$ set with
$\Hom(\kk^\times\times\kk^\times,\kk^\times)$ because the
group~$\kk^\times$ is abelian and $\kk^\times\times\kk^\times$ is the
abelianization of~$\Aut_0(A)$. 

Of course, when $q^2\neq1$ the composition in the lemma is just an
isomorphism.

\subsubsection{An example: separable algebras and group algebras}
\label{subsubsect:jac:separable}

Let $A$ be a finite-dimensional algebra, as usual, and for each element~$a$
of~$A$ let us write $L_a:b\in A\mapsto ab\in A$ for the linear map given by
left multiplication by~$A$. The \newterm{trace form} on~$A$ is the bilinear
map
  \[
  \gen{\place,\place}_{\tr}:(a,b)\in A \times A
        \mapsto \tr L_{ab}\in \kk,
  \]
which is clearly associative and symmetric\footnote{We could have used
\emph{right} multiplications here: it is important to note that the
corresponding «right trace» that we obtain in that way is in general a
different linear map --- Nicolas Bourbaki gives an example of this
in~\cite{Bourbaki:alg8}*{\S12, no.\,3}. The two traces coincide if the
algebra~$A$ is separable, though: this follows from Theorem~9.31 in Irving
Reiner's \cite{Reiner:orders}.}. We say that the algebra~$A$ is
\newterm{strongly separable} if the form is also non-degenerate --- this is
not the original definition given by Teruo Kanzaki
in~\cite{Kanzaki:strongly} but it is equivalent to it: see Marcelo Aguiar's
\cite{Aguiar:strongly}*{Theorem 3.1} --- and in that case the algebra is, in
particular, separable. Of course, such an algebra is Frobenius with respect
to the trace form, and even symmetric.

There are many examples of this. If the characteristic of~$\kk$ is zero, an
algebra is strongly separable if and only if it is semisimple. On the other
hand, the group algebra~$\kk G$ of a finite group is strongly separable if
and only if it is semisimple, in any characteristic, and in that case the
trace form is the familiar one that has
$\gen{g,h}_{\tr}=\abs{G}\cdot\delta_{gh,1}$ for all elements~$g$ and~$h$ in
the group. Similarly, a finite-dimensional Hopf algebra over any field is
strongly separable if it is semisimple and involutory\footnote{If the
characteristic of~$\kk$ is $0$ a semisimple Hopf algebra is automatically
involutory by a theorem of Richard Larson and David Radford \cite{LR}, and
a well-known and very open conjecture of  Irving Kaplansky is that this is
actually true for all fields.}, that is, the order of its antipode is~$2$,
and, according to what we have just said, if the field~$\kk$ has
characteristic zero then this condition is also necessary.

\begin{Lemma}\label{lemma:jac:strongly-separable}
The Nakayama Jacobian of an automorphism of a strongly separable algebra
is~$1$.
\end{Lemma}

\begin{proof}
Let $A$ be a strongly separable algebra, and let $u:A\to A$ be an
automorphism of~$A$. We need to show that $\tr L_{u(a)}=\tr L_a$ for all
$a\in A$. Let us fix an element~$a$ in~$A$, and let $\B=(a_1,\dots,a_n)$ be
an ordered basis for~$A$. Since $u$ is an automorphism of~$A$, the sequence
$\B'\coloneqq(u(a_i),\dots,u(a_n))$ is also an ordered basis for~$A$, and
the matrix of the map~$L_a:A\to A$ with respect to~$\B$ coincides with the
matrix of the map $L_{u(a)}$ with respect to~$\B'$: that $L_a$
and~$L_{u(a)}$ have the same traces is therefore clear.
\end{proof}

In fact, a similar result applies to all separable algebras, but needs a
little more work. Let $A$ be a separable algebra. There is then a finite
separable extension~$E\fe\kk$ such that there are positive integers~$r$
and~$m_1$,~\dots,~$m_r$ and an isomorphism 
  \[ \label{eq:phi}
  \phi :E\otimes_\kk A \to \prod_{i=1}^r\M_{m_i}(E)
  \]
of $E$-algebras, and it can be proved that for each $a$ in~$A$ there is a
unique element~$\trd(a)$ in~$\kk$ such that $\tr(\phi(1_E\otimes
a))=\phi(1\otimes\trd(a))$ --- the $\tr$ on the left hand side of this
equality denotes the componentwise trace in the product appearing
in~\eqref{eq:phi} --- and that $\trd(a)$ is in fact independent of the
splitting field~$E$ and of the isomorphism~$\phi$. In this way we obtain a
linear map $\trd:A\to\kk$, the \newterm{reduced trace} of~$A$, with the
property that the bilinear form
  \[
  \gen{\place,\place}_{\trd}:(a,b)\in A \times A
        \mapsto \trd(ab)\in \kk,
  \]
is non-degenerate, associative and symmetric. We refer the reader to
\cite{CR:methods:1}*{\S7D} and, in particular, Proposition~7.41 there, and
to~\cite{Reiner:orders}*{\S9} for a discussion of this. In any case, we see
that $A$ is a symmetric Frobenius algebra when endowed with this bilinear
form. 

With respect to that Frobenius structure we have the result we want:

\begin{Lemma}\label{lemma:jac:separable}
The Nakayama Jacobian of an automorphism of a separable algebra is~$1$.
\end{Lemma}

\begin{proof}
Let $u:A\to A$ be an automorphism of a separable algebra~$A$. The fact that
when we compute~$\trd(a)$ for an element~$a$ of~$A$ using the procedure
sketched above the choice of the automorphism~\eqref{eq:phi} is irrelevant
implies that $\trd(u(a))=\trd(a)$. That $\jac_\sigma(u)=1$ follows
immediately from that.
\end{proof}

One should keep in mind that when $A$ is a strongly separable algebra the
trace form and the reduced trace form are in general different --- for
example, when $A$ is a semisimple group algebra they coincide if and only
if the group is abelian. This means that the Jacobians which these two
lemmas tell us are both equal to~$1$ are \emph{different} (!). On the other
hand, while Lemma~\ref{lemma:jac:separable} is very nice, actually
computing the reduced trace form of a separable algebra in concrete
examples is usually very difficult as it seems to require knowing the
Artin--Wedderburn decomposition of the algebra. This is done for the
rational group algebras of dihedral and generalized quaternion groups
in~\cite{CR:methods:1}*{\S7D}. 

\bigskip

It follows from Lemma~\ref{lemma:jac:strongly-separable} that automorphisms
of semisimple group algebras have trivial Jacobian with respect to their
trace forms. It is natural to ask what happens with non-semisimple group
algebras, that is, of the form~$\kk G$ with $\kk$ of positive
characteristic~$p$ dividing the order of~$G$. Of course, we cannot use the
trace form in this case for it is identically zero and such algebras do not
have reduced trace forms, but we do have the bilinear form
$\gen{\place,\place}:\kk G\times\kk G\to\kk$ such that
$\gen{g,h}=\delta_{gh,1}$ for all elements~$g$ and~$h$ of the group that
makes $\kk G$ a symmetric Frobenius algebra and which is morally a
rescaling of the trace form.

We pick a prime~$p$, consider a cyclic group~$G$ of order~$p$, choose a
generator~$\gamma$ for~$G$, suppose that the field~$\kk$ has
characteristic~$p$, and study the algebra $\kk G$. In fact, to describe the
automorphisms of~$\kk G$ it is more convenient to consider a different
incarnation of the algebra. We let $A$ be the quotient~$\kk[X]/(X^p)$ and
write~$x$ for the class of~$X$ in~$A$: there is then an isomorphism
$\phi:\kk G\to A$ such that $\phi(\gamma)=x+1$, and this isomorphism is
isometric if we endow~$A$ with the bilinear
form~$\gen{\place,\place}:A\times A\to\kk$  that for all $i$ and all~$j$
in~$\inter{0,p-1}$ has
  \[
  \gen{x^i,x^j} 
        = \begin{cases*}
          (-1)^{i+j} & if $i+j<p$; \\
          0 &if $i+j\geq p$.
          \end{cases*}
  \]
We will compute the Jacobians of automorphisms of~$A$ with respect to this
form.

\begin{Lemma}
Let $p$ be a prime number, suppose that the field~$\kk$ has
characteristic~$p$, let $A$ be the algebra~$\kk[X]/(X^p)$ endowed with the
bilinear form described above, and let $x$ be the class of~$X$ in~$A$. 
\begin{thmlist}

\item If $f$ is an element of the set $R\coloneqq\rad A\setminus\rad^2A$,
then there is exactly one automorphism $u_f:A\to A$ of~$A$ such that
$u_f(x)=f$. The map $f\in R\mapsto u_f\in\Aut(A)$ is bijective.

\item Let $\mu:A\to\kk$ be the linear map such that $\mu(x^i)=(-1)^i$ for
each $i\in\inter{0,p-1}$.  If $f$ is an element of~$R$, then the Jacobian
of~$u_f:A\to A$ is
  \[ \label{eq:juf}
  \jac_\sigma(u_f) = \mu(f^{p-1}) 
                     + \sum_{i=1}^{p-1}\bigl(\mu(f^{p-1-i})+\mu(f^{p-i})\bigr)x^i.
  \]

\end{thmlist}
\end{Lemma}

The bijection in~\thmitem{1} is even an isomorphism of groups if we
endow~$R$ with the operation induced by the functional composition of
polynomials. One should notice that in~\thmitem{2} the right hand side of
the equality~\eqref{eq:juf} is indeed a unit of~$A$ when~$f$ is in~$R$, for
in that case~$f^{p-1}$ is a non-zero scalar multiple of~$x^{p-1}$ and that
$\mu(f^{p-1})$ is a non-zero scalar.

The map~$\mu$ that is defined in~\thmitem{2} is such that
$\mu(a)=\gen{a,1}$ for all $a\in A$. In fact, if $a$ is an element of~$A$
and $a_0$,~\dots,~$a_{p-1}$ are the scalars such that
$a=a_0+a_1x+\cdots+a_{p-1}x^{p-1}$, then
$\mu(a)=a_0+a_1(-1)+\cdots+a_{p-1}(-1)^{p-1}$. We see with this that $\mu$
looks like the «evaluation of elements of~$A$ at~$-1$, except that, of
course, there is no such thing.

\begin{proof}
Let $f$ be an element of the set~$R$. Since $f^p\in\rad^pA=0$, there is a
unique algebra endomorphism $u_f:A\to A$ such that $u(x)=f$. Since $f\in R$
there are $t\in\kk^\times$ and $g\in\rad^2A$ such that $f=tx+g$, and
therefore $x^i\equiv t^{-i}f^i\mod\rad^{i+1}A$ for all $i\in\inter{p-1}$:
we can deduce from this that all powers of~$x$ are in the image of~$u_f$
and thus that $u_f$ is bijective, so an isomorphism. On the other hand, if
$u:A\to A$ is an automorphism of~$A$, then $u(x)$ belongs to~$\rad A$
because~$x$ does, and not to~$\rad^2A$ because $f$ induces a bijective
endomorphism of the vector space $\rad A/\rad^2A$ and the class of~$x$ in
that quotient is not zero, and it follows form this that $u(x)\in R$: ss
the automorphisms~$u$ and~$u_{u(x)}$ of~$A$ coincide on~$x$, they are
therefore equal. This shows that the function in~\thmitem{1} is surjective,
and it is manifestly injective. This proves~\thmitem{1}.

Let $f$ be an element of the set~$R$, and let us write~$J$ for the element
of~$A$ that appears in the right hand side of the equality~\eqref{eq:juf}.
For each $j\in\inter{0,p-1}$ we have that 
  \begin{align}
  \gen{J,x^j}
        &= \mu(f^{p-1})\cdot\gen{1,x^j}
           + \sum_{i=1}^{p-1}\bigl(\mu(f^{p-1-i})+\mu(f^{p-i})\bigr)\cdot\gen{x^i,x^j} \\
        &= (-1)^j\cdot\mu(f^{p-1})
           + \sum_{i=1}^{p-1-j}(-1)^{i+j}\cdot\bigl(\mu(f^{p-1-i})+\mu(f^{p-i})\bigr) \\
        &= (-1)^j\cdot\mu(f^{p-1})
           - \sum_{i=2}^{p-j}(-1)^{i+j}\cdot\mu(f^{p-i}) 
           + \sum_{i=1}^{p-1-j}(-1)^{i+j}\cdot\mu(f^{p-i}) \\
        &= (-1)^j\cdot\mu(f^{p-1})
           - (-1)^{p}\cdot\mu(f^{j}) 
           + (-1)^{1+j}\cdot\mu(f^{p-1})  \\
        &= (-1)^{p-1}\mu(f^j) 
         = \mu(f^j)
         = \gen{u_f(x^i),1}.
  \end{align}
It follows from this that, more generally, $\gen{J,a}=\gen{u_f(a),1}$ for
all $a\in A$, so that $\jac_\sigma(u_f)=J$, and this is the equality we
want to prove.
\end{proof}

The formula~\eqref{eq:juf} in this lemma shows that the group algebra~$\kk
G$ of a finite cyclic group of prime order~$p$ has, when the characteristic
of~$\kk$ is~$p$, many automorphisms with non-trivial Jacobian. It is
natural to wonder what units of~$A$ are the Jacobian of an automorphism. It
follows immediately form the formula in~\eqref{eq:juf} that an element
$a=a_0+a_1x+\cdots+a_{p-1}x^{p-1}$ of~$A$ that is the Jacobian of an
automorphism~$u_f$ of~$A$ for some~$f=f_1x+\cdots+f_{p-1}x^{p-1}$ in the
set~$R$ has 
  \begin{gather}
  \sum_{i=0}^{p-1}(-1)^ia_i 
        = \mu(f^{p-1})
          + \sum_{i=1}^{p-1}(-1)^i\bigl(\mu(f^{p-1-i})+\mu(f^{p-i})\bigr)
        = (-1)^{p-1}\mu(f^{0})
        = 1
\intertext{and}
  a_0 = \mu(f^{p-1}) = f_1^{p-1} \neq 0.
  \end{gather}
Using an elimination calculation using Groebner basis for the map $f\in
R\mapsto\jac_\sigma(u_f)\in A$ for small values of~$p$ seems to indicate
that these necessary conditions~$\sum_{i=0}^{p-1}(-1)^ia_i=0$ and
$a_0\neq0$ are in fact also sufficient for~$a$ to be in the image. We do
not know how to prove this, though.

\bigskip

In general, if the field~$\kk$ is of characteristic~$p$ and $P$ is a Sylow
$p$-subgroup of a finite group~$G$, then the algebra~$\kk G$ is a separable
extension of its subalgebra~$\kk P$. The calculations in this section
indicate that it may be interesting to study the behavior of Nakayama
Jacobians under general separable extensions.

\subsubsection{Crossed products}
\label{subsubsect:jac:cross}

Let $A$ be a Frobenius algebra, let $\gen{\place,\place}$ be a
non-degenerate and associative linear form on~$A$, and let $G$ be a finite
group that acts on~$A$ by algebra automorphisms. We can construct from this
data the \newterm{crossed product algebra} $A\crs G$: it is the free left
$A$-module on the set~$G$ whose monomials we will write in the form $a\crs
g$ and on which the multiplication is such that 
  \[
  a\crs g\cdot b\crs h = ag(b)\crs gh
  \]
whenever $a$,~$b\in A$ and~$g$,~$h\in G$. If the action of~$G$ on~$A$ is
``orthogonal'' with respect to the bilinear form, so that
  \[ \label{eq:orth}
  \gen{g(a),g(b)} = \gen{a,b}
  \]
for all $a$,~$b\in A$ and all~$g\in G$, then the algebra~$A\crs G$ is also
a Frobenius algebra: it is easy to check that the bilinear form
  \(
  \ggen{\place,\place}:A\crs G\times A\crs G\to\kk
  \)
such that
  \[
  \ggen{a\crs g,b\crs h} = \gen{a,g(b)}\cdot\delta_{gh,e}
  \]
for all $a$,~$b\in A$ and all~$g$,~$h\in G$ is non-degenerate and
associative, and the corresponding Nakayama automorphism $\Sigma:A\crs G\to
A\crs G$ is such that
  \[
  \Sigma(a\crs g) = \sigma(a)\crs  g
  \]
for all $a\in A$ and all~$g\in G$. If the action of~$G$ does not preserve
the bilinear form on~$A$ the crossed product is also a Frobenius algebra,
and we can use our Jacobians to describe its Nakayama automorphism. In
fact, we can do this more generally  for twisted crossed products with no
extra complication, so we do that instead.

\bigskip

Let $A$ be a Frobenius algebra, let $\gen{\place,\place}$ be a
non-degenerate and associative linear form on~$A$, let $G$ be a finite
group that acts on~$A$ by algebra automorphisms, and let $\alpha:G\times
G\to\kk^\times$ be a $2$-cocycle on the group~$G$ with values in the
multiplicative group of our ground field, so that for all $g$,~$h$,~$k\in
G$ we have that
  \[
  \alpha(h,k)\cdot\alpha(g,hk) = \alpha(gh,k)\cdot\alpha(g,h).
  \]
From this data we can construct the \newterm{twisted crossed product}
algebra $A\crs _\alpha G$: it is the free left $A$-module spanned by~$G$
whose monomials we write, as before, $a\crs g$, endowed this time with the
multiplication such that
  \[
  a\crs g\cdot b\crs h = \alpha(g,h)\cdot ag(b)\crs gh
  \]
for all $a$,~$b\in A$ and all~$g$,~$h\in G$. Indeed, the associativity of
this multiplication follows from the cocycle condition on~$\alpha$, and
$\alpha(1_G,1_G)^{-1}\cdot1_A\crs 1_G$ is a unit element for it.

\begin{Lemma}\label{lemma:jac:cross}
Let $A$ be a Frobenius algebra with non-degenerate and associative bilinear
form~$\gen{\place,\place}$ and corresponding Nakayama automorphism
$\sigma:A\to A$, let $G$ be a finite group acting on~$A$ by algebra
automorphisms, and let $\alpha:G\times G\to\kk^\times$ be a $2$-cocycle
on~$G$. There is a non-degenerate and associative bilinear form
$\ggen{\place,\place}:A\crs _\alpha G\times A\crs _\alpha G\to\kk$ on the
twisted crossed product algebra~$A\crs _\alpha G$ that has
  \[
  \ggen{a\crs g,b\crs h} = \alpha(g,h)\cdot\gen{a,g(b)}\cdot\delta_{gh,e}
  \]
for all $a$,~$b\in A$ and all~$g$,~$h\in G$, and the Nakayama automorphism
$\Sigma:A\crs_\alpha G\to A\crs_\alpha G$ that corresponds to it has
  \[ \label{eq:crs:s}
  \Sigma(a\crs g) = \frac{\alpha(g,g^{-1})}{\alpha(g^{-1},g)}
                    \cdot\bigl( \sigma(a)\cdot g\sigma(\jac_\sigma(g))\bigl)\crs g
  \]
for all $a\in A$ and all $g\in G$.
\end{Lemma}

For each $g\in G$ we are writing here ~$\jac_\sigma(g)$ for the Jacobian of
the algebra automorphism $a\in A\mapsto g(a)\in A$, of course. If the
cocycle~$\alpha$ is normalized, as it is often the case, then the fraction
that appears in~\eqref{eq:crs:s} is equal to~$1$.

\begin{proof}
That the bilinear form~$\ggen{\place,\place}$ is associative follows from a
direct calculation. On the other hand, if $x=\sum_{g\in G}a_g\crs g$ is a
non-zero element of~$A\crs_\alpha G$, then there is an~$h$ in~$G$ such that
$a_h\neq0$, and a~$b$ in~$A$ such that $\gen{a_h,b}\neq0$, and then
$\ggen{x,h^{-1}(b)\crs h^{-1}}=\gen{a_h,b}\neq0$: this tells us that the
form~$\ggen{\place,\place}$ is non-degenerate.

Let now $a$ and~$b$ be two elements of~$A$, and $g$ and~$h$ two of~$G$. We
have that
  \begin{align}
  \MoveEqLeft
  \ggen*{b\crs h, \frac{\alpha(g,g^{-1})}{\alpha(g^{-1},g)}
                    \cdot\bigl( \sigma(a)\cdot g\sigma(\jac_\sigma(g))\bigl)\crs g} \\
     &= \frac{\alpha(g,g^{-1})}{\alpha(g^{-1},g)}\cdot\alpha(h,g)
        \cdot \gen{b,h\sigma(a)\cdot hg\sigma(\jac_\sigma(g))}\cdot\delta_{hg,e}. 
  \end{align}
This is zero if~$h$ and~$g^{-1}$ are not equal, and thus equal to
$\ggen{a\crs g,b\crs h}$, which is also zero in that case. If instead
$h=g^{-1}$, then
  \begin{align}
  \MoveEqLeft
  \frac{\alpha(g,g^{-1})}{\alpha(g^{-1},g)}\cdot\alpha(h,g)
        \cdot \gen{b,h\sigma(a)\cdot hg\sigma(\jac_\sigma(g))}\cdot\delta_{hg,e} \\
    &= \alpha(g,g^{-1}) \cdot \gen{b,g^{-1}\sigma(a)\cdot \sigma(\jac_\sigma(g))} \\
    &= \alpha(g,g^{-1}) \cdot \gen{\jac_\sigma(g)\cdot b,g^{-1}\sigma(a)} \\
    &= \alpha(g,g^{-1}) \cdot \gen{g(b),\sigma(a)} \\
    &= \alpha(g,g^{-1}) \cdot \gen{a,g(b)} \\
    &= \ggen{a\crs g,b\crs h}
  \end{align}
We therefore see that in all cases we have
  \[
  \ggen*{b\crs h, \frac{\alpha(g,g^{-1})}{\alpha(g^{-1},g)}
                    \cdot\bigl( \sigma(a)\cdot g\sigma(\jac_\sigma(g))\bigl)\crs g} 
    = \ggen{a\crs g,b\crs h},
  \]
so that the Nakayama automorphism corresponding to the
form~$\ggen{\place,\place}$ is as in the statement of the lemma.
\end{proof}

There is a variant of the notion of twisted crossed products, originally
studied by Oswald Teichmüller in~\cite{Teichmuller} in the context
non-commutative Galois theory and that is useful in the study of the Brauer
group of fields --- see Jean-Pierre Tignol's \cite{Tignol} for a modern
exposition --- that results in algebras over smaller fields. We can also
use Jacobians to describe their Nakayama automorphisms, up to a small,
interesting change. Let us finish this section explaining this.

Let $G$ be a finite group that acts faithfully by isomorphisms on our
ground field~$\kk$, let $A$ be an algebra, and let $\Aut_\Ring(A)$ be the
group of automorphisms of~$A$ as a \emph{ring}, and for every unit~$u$
of~$A$ let $\iota_u:x\in A\mapsto xax^{-1}\in A$ be the corresponding inner
automorphism, as usual. A \newterm{factor set} in this situation is an
ordered pair~$(\omega,\alpha)$ with $\omega:G\to\Aut_{\Ring}(A)$ and
$\alpha:G\times G\to A^\times$ two functions such that
\begin{itemize}

\item $\omega(g)(x) = g(x)$ for all $g\in G$ and all $x\in\kk$;

\item $\omega(g)\circ\omega(h) = \iota_{\alpha(g,h)}\circ\omega(gh)$ for
all $g$,~$h\in G$; and

\item $\omega(g)(\alpha(h,k))\cdot \alpha(g,hk) = \alpha(g,h)\cdot
\alpha(gh,k)$ for all $g$,~$h$,~$k\in G$.

\end{itemize}
We write $\F(G,A)$ for the set of all such factor sets. If the action
of~$G$ on~$\kk$ is trivial, then there is a «trivial» element in~$\F(G,A)$,
so that the set is then not empty, but in general it can well be. There is
a relation of cohomology between these factor sets, and a corresponding
cohomology set --- we do not need this, so we limit ourselves to refer the
interested read to~\cite{Tignol} for more information about this.

Let $(\omega,\alpha)$ be an element of~$\F(G,A)$. Let us note that
  \[
  \omega(1_G)\circ\omega(1_G) = \iota_{\alpha(1_G,1_G)}\circ\omega(1_G),
  \]
so that in fact
  \[ \label{eq:oi1}
  \omega(1_G) = \iota_{\alpha(1_G,1_G)}.
  \]
Using this and taking $(g,h,k)$ to be first~$(1_G,1_G,g)$ and
then~$(g,1_G,1_G)$ in the third condition above we find that for all
elements~$g$ of~$G$ we have
  \[ \label{eq:oi2}
  \alpha(1_G,g) = \alpha(1_G,1_G),
  \qquad
  \omega(g)(\alpha(1_G,1_G)) = \alpha(g,1_G)
  \]
We write $A\crs_{(\omega,\alpha)} G$ for the left $A$-module spanned by the
set~$G$ endowed with the unique $\ZZ$-bilinear map $A\crs_{(\omega,\alpha)}
G\times A\crs_{(\omega,\alpha)} G\to A\crs_{(\omega,\alpha)} G$ such that
  \[
  a\crs g\cdot b\crs h = (a\cdot\omega(g)(b)\cdot\alpha(g,h))\crs gh
  \]
whenever $a$,~$b\in A$ and~$g$,~$h\in G$. This is an associative
multiplication: indeed, whenever $a$,~$b$,~$c\in A$ and $g$,~$h$,~$k\in G$
we have that
  \begin{align}
  \MoveEqLeft
  a\crs g\cdot(b\crs h\cdot c\crs k) \\
        &= a\crs g\cdot\bigl((b\cdot\omega(h)(c)\cdot \alpha(h,k))\crs hk\bigr) \\
        &= \bigl(
                a
                \cdot\omega(g)(b\cdot\omega(h)(c)\cdot \alpha(h,k))
                \cdot\alpha(g,hk)
           \bigr)\crs ghk \\
        &= \bigl(
                a
                \cdot\omega(g)(b)
                \cdot\omega(g)(\omega(h)(c))
                \cdot\omega(g)(\alpha(h,k))
                \cdot\alpha(g,hk)
           \bigr)\crs ghk \\
        &= \bigl(
                a
                \cdot\omega(g)(b)
                \cdot\alpha(g,h)
                \cdot\omega(gh)(c)
                \cdot\alpha(g,h)^{-1}
                \cdot\omega(g)(\alpha(h,k))
                \cdot\alpha(g,hk)
           \bigr)\crs ghk \\
        &= \bigl(
                a
                \cdot\omega(g)(b)
                \cdot\alpha(g,h)
                \cdot\omega(gh)(c)
                \cdot\alpha(gh,k)
           \bigr)\crs ghk \\
        &= \bigl((a\cdot\omega(g)(b)\cdot\alpha(g,h))\crs gh)\bigr)\cdot c\crs k \\
        &= (a\crs g\cdot b\crs h)\cdot c\crs k.
  \end{align}
On the other hand, the element $\alpha(1_G,1_G)^{-1}\crs 1_G$ is a unit
element for this multiplication. Indeed, if $a\in A$ and $g\in G$ we have
that
  \begin{align}
  \MoveEqLeft
  \alpha(1_G,1_G)^{-1}\crs1_G\cdot a\crs g \\
        &= \bigl(\alpha(1_G,1_G)^{-1}\cdot\omega(1_G)(a)\cdot\alpha(1_G,g)\bigr)\crs g \\
        &= \bigl(a\cdot\alpha^{-1}(1_G,1_G) \cdot\alpha(1_G,g) \bigr)\crs g \\
        &= a\crs g
  \end{align}
and
  \begin{align}
  \MoveEqLeft
  a\crs g\cdot\alpha(1_G,1_G)^{-1}\crs1_G \\
        &= \bigl(a\cdot\omega(g)(\alpha(1_G,1_G)^{-1})\cdot\alpha(g,1_G)\bigr)\crs g \\
        &= a\crs g
  \end{align}
thanks to~\eqref{eq:oi1} and~\eqref{eq:oi2}. In view of all this what we
have is a ring. Now, there is an obvious left $\kk$-vector space structure
on~$A\crs_{(\omega,\alpha)}G$, but this ring is \emph{not} in general a
$\kk$-algebra with respect to that structure. For example, if $t\in\kk$,
then $t\cdot\alpha(1_G,1_G)^{-1}1_A\crs1_G$ is
$t\alpha(1_G,1_G)^{-1}\crs1_G$, of course, and for each $g\in G$ we have
that
  \begin{align}
  \MoveEqLeft
  t\alpha(1_G,1_G)^{-1}\crs1_G \cdot 1_A\crs g \\
       &= \bigl(
            t\alpha(1_G,1_G)^{-1}\cdot\omega(1_G)(1_A)\cdot\alpha(1_G,g)
          \bigr)\crs g \\
       &= t\crs g
\shortintertext{while}
  \MoveEqLeft
   1_A\crs g\cdot t\alpha(1_G,1_G)^{-1}\crs1_G \\
       &= \bigl(
            1_A\cdot\omega(g)(t\alpha(1_G,1_G)^{-1})\cdot\alpha(g,1_G)
          \bigr)\crs g \\
       &= \omega(g)(t)\crs g \\
       &= g(t)\crs g.
  \end{align}
It follows from this that if we choose~$t$ in~$\kk$ but not in the fixed
field~$F\coloneqq\kk^G$, then the element $t\alpha(1_G,1_G)^{-1}\crs1_G$ is
not central, so that $A\crs_{(\omega,f)}G$ is not a $\kk$-algebra. The same
calculation, though, shows that this ring \emph{is} an~$F$-algebra.

The classical reason to study this construction is that if the algebra~$A$
with which we started is a central simple algebra over~$\kk$, then the
algebra $A\crs_{(\omega,\alpha)}G$ with which we ended is a central simple
algebra over~$F$ of which~$\kk$ is a subfield whose centralizer is
isomorphic to~$A$ --- we refer again to~\cite{Tignol} for details about
this --- so this is useful in relating the Brauer groups of the two fields.
The reason for which we are interested in this construction here is the
following result, which generalizes Lemma~\ref{lemma:jac:cross}. Since in
it we deal with semi-linear maps, we need for it a mild generalization of
our Nakayama Jacobians, and that is the first part of the lemma.

\begin{Lemma}\label{lemma:jac:generalized-cross}
Let $G$ be a finite group of automorphism of the field~$\kk$, let
$F\coloneqq\kk^G$ be its fixed field, and let $\tr:\kk\to F$ be the trace
map for the extension~$\kk\fe F$. Let~$A$ be a Frobenius algebra with
respect to a non-degenerate and associative bilinear form
$\gen{\place,\place}:A\times A\to\kk$, let $\sigma:A\to A$ be the
corresponding Nakayama automorphism, and let $(\omega,\alpha)$ be an
element of the set~$\F(G,A)$.
\begin{thmlist}

\item There is a unique function $\jac_\sigma^G:G\to A$ such that
  \[
  \tr{}\gen{\omega(g)(a),\omega(g)(b)}
        = \tr{}\gen{\jac_\sigma^G(g)\cdot a,b}
  \]
for all $g\in G$ and all $a$,~$b\in A$.

\item There is a unique $F$-bilinear form
  \[
  \ggen{\place,\place}
        : A\crs_{(\omega,\alpha)}G \times A\crs_{(\omega,\alpha)}G
        \to F
  \]
such that 
  \begin{multline} \label{eq:fb:1}
  \ggen{a\crs g,b\crs h}
        = \tr{}
          \gen{a,\omega(g)(b)\cdot\alpha(g,h)}\cdot\delta_{gh,e} \\
  \text{for all $a$,~$b\in A$ and all $g$,~$h\in H$,}\qquad
  \end{multline}
it is non-degenerate and associative, so that the $F$-algebra
$A\crs_{(\omega,\alpha)}G$ is Frobenius, and the corresponding Nakayama
automorphism $\Sigma:A\crs_{(\omega,\alpha)}G\to A\crs_{(\omega,\alpha)}G$
has
  \begin{multline}
  \Sigma(a\crs g) 
        = 
                \bigl(
                \alpha(1_G,1_G)^{-1} 
                \cdot\sigma(a)
                \cdot\omega(g)(\sigma(\jac_\sigma^G(g))) 
                \cdot \alpha(g,g^{-1})  \\
                \cdot \alpha(1_G,1_G)^{-1} 
                \cdot \omega(g^{-1})^{-1}(\alpha(g^{-1},g)^{-1})
                \bigr)
                \crs g
  \end{multline}
for all $a\in A$ and all $g\in G$.

\end{thmlist}
\end{Lemma}

The point of~\thmitem{1} is that the maps $\omega(g):A\to A$ are --- if the
action of~$G$ on~$\kk$ is not trivial~--- \emph{not} morphisms
of~$\kk$-algebras, only of $F$-algebras, and thus do not have Jacobians
according to our general definition. The function~$\jac_\sigma^G$ depends
on the group~$G$ and the factor set~$(\omega,\alpha)$, even if the notation
does not reflect that. Notice that since the extension $\kk\fe F$ is
Galois, the trace map $\tr:\kk\to F$ is non-zero.

\begin{proof}
\thmitem{1} Let $g$ be an element of~$A$. The function $a\in A\mapsto
g^{-1}(\gen{\omega(g)(a),1})\in\kk$ is $\kk$-linear, so the non-degeneracy
of the bilinear form~$\gen{\place,\place}$ implies that there is an
element~$\jac_\sigma^G(g)$ in~$A$ such that
$g^{-1}(\gen{\omega(g)(a),1})=\gen{\jac_\sigma^G(g),a}$ for all $a\in A$.
Since the trace map $\tr:\kk\to F$ is constant on $G$-orbits, we then have
that $\tr{}\gen{\omega(g)(a),1}=\tr{}\gen{\jac_\sigma^G(g),a}$ for all
$a\in A$. 

Let $j$ be another element of~$A$ such that
$\tr{}\gen{\omega(g)(a),1}=\tr{}\gen{j,a}$ for all $a\in A$. We then have
that $\tr{}\gen{\jac_\sigma^G(g)-j,a}=0$ for all $a\in A$, so that the
$\kk$-linear function $a\in A\mapsto \gen{\jac_\sigma^G(g)-j,a}\in\kk$ has
image contained in the kernel of~$\tr$: the latter is a proper $F$-subspace
of~$\kk$, and it follows from this that, in fact
$\gen{\jac_\sigma^G(g)-j,a}=0$ for all $a\in A$, so that
$j=\jac_\sigma^G(g)$. Putting everything together, we see that there is a
unique function $\jac_\sigma^G:G\to A$ with the property described in the
lemma.

\medskip

\thmitem{2} That there is an $F$-bilinear~$\ggen{\place,\place}$ form
on~$A\crs_{(\omega,\alpha)}G$ that satisfies the condition~\eqref{eq:fb:1}
is clear, since for each $g\in G$ the map $\omega(g):A\to A$ is $F$-linear.
Let $x=\sum_{g\in G}a_g\crs g$ be a non-zero element
of~$A\crs_{(\omega,\alpha)}$. There is then an element~$h$ in~$G$ such that
$a_h\neq0$, and there is in turn an element~$b$ in~$A$ such that
$\gen{a,b}\neq0$. The function $t\in\kk\mapsto\gen{a,tb}\in\kk$ is
therefore $\kk$-linear and non-zero: as the trace map~$\tr$ is non-zero,
this implies that we can choose $s\in\kk$ so that $\tr\gen{a,sb}\neq0$, and
this implies that
  \[
  \ggen{x, \omega(h)^{-1}(sb\cdot\alpha(h,h^{-1})^{-1}) \crs h^{-1}} \neq 0.
  \]
This shows that the form~$\ggen{\place,\place}$ is non-degenerate, and a
direct, boring calculation proves that it is also associative.

Let now $a$ and~$b$ be two elements of~$A$ and~$g$ one of~$G$. We have that
  \begin{align}
  \MoveEqLeft
  \ggen{a\crs g,b\crs g^{-1}} \\
        &= \tr{}\gen{a,\omega(g)(b)\cdot\alpha(g,g^{-1})} \\
        &= \tr{}\gen{\omega(g)(b)\cdot\alpha(g,g^{-1}),\sigma(a)} \\
        &= \tr{}\gen{\omega(g)(b),\alpha(g,g^{-1})\cdot\sigma(a)} \\
        &= \tr{}\gen{\jac_\sigma^G(g)\cdot b,
                   \omega(g)^{-1}(\alpha(g,g^{-1})\cdot\sigma(a))} \\
        &= \tr{}\gen{b,
                   \omega(g)^{-1}(\alpha(g,g^{-1})\cdot\sigma(a))
                   \cdot\sigma(\jac_\sigma^G(g))
                   } \\
        &= \tr{}\gen{b,
                   \omega(g)^{-1}(\alpha(g,g^{-1})\cdot\sigma(a)
                   \cdot\omega(g)(\sigma(\jac_\sigma^G(g))))
                   }. \label{eq:sww}
  \end{align}
As
  \(
  \omega(g)\circ\omega(g^{-1})
        = \iota_{\alpha(g,g^{-1})}\circ\omega(1_G)
        = \iota_{\alpha(g,g^{-1})}\circ\iota_{\alpha(1_G,1_G)}
  \),
we have that
  \[ \label{eq:swr}
  \omega(g)^{-1}
        = \omega(g^{-1})
          \circ \iota_{\alpha(1_G,1_G)^{-1}}
          \circ \iota_{\alpha(g,g^{-1})^{-1}}.
  \]
and therefore
  \begin{align*}
  \MoveEqLeft
  \omega(g)^{-1}(
        \alpha(g,g^{-1})
        \cdot\sigma(a)
        \cdot\omega(g)(\sigma(\jac_\sigma^G(g)))
        ) \\
        &= \begin{multlined}[t][.85\displaywidth]
           \omega(g^{-1})
                (
                \alpha(1_G,1_G)^{-1} 
                \cdot \alpha(g,g^{-1})^{-1} 
                \cdot \alpha(g,g^{-1}) \\
                \cdot\sigma(a)
                \cdot\omega(g)(\sigma(\jac_\sigma^G(g))) 
                \cdot \alpha(g,g^{-1})
                \cdot \alpha(1_G,1_G)^{-1} 
                )
           \end{multlined}
           \\
        &= \begin{multlined}[t][.85\displaywidth]
           \omega(g^{-1})
                (
                \alpha(1_G,1_G)^{-1} 
                \cdot\sigma(a)
                \cdot\omega(g)(\sigma(\jac_\sigma^G(g))) 
                \cdot \alpha(g,g^{-1})
                \cdot \alpha(1_G,1_G)^{-1} 
                ).
          \end{multlined}
  \end{align*}
Going back to~\eqref{eq:sww} with this, we see that
  \begin{align*}
  \MoveEqLeft
  \ggen{a\crs g,b\crs g^{-1}} \\
        &= \tr{}\gen{b,
                \omega(g^{-1})
                (
                \alpha(1_G,1_G)^{-1} 
                \cdot\sigma(a)
                \cdot\omega(g)(\sigma(\jac_\sigma^G(g))) 
                \cdot \alpha(g,g^{-1})
                \cdot \alpha(1_G,1_G)^{-1} 
                )
                } \\
        &= \begin{multlined}[t][.85\displaywidth]
           \bigl\llangle b\crs g^{-1},
                \bigl(
                \alpha(1_G,1_G)^{-1} 
                \cdot\sigma(a)
                \cdot\omega(g)(\sigma(\jac_\sigma^G(g))) 
                \cdot \alpha(g,g^{-1}) \\
                \cdot \alpha(1_G,1_G)^{-1} 
                \cdot \omega(g^{-1})^{-1}(\alpha(g^{-1},g)^{-1})
                \bigr)
                \crs g
                \bigr\rrangle
          \end{multlined} 
          \\
        &= \ggen{b\crs g^{-1},\Sigma(a\crs g)}.
  \end{align*}
On the other hand, we have that $\ggen{a\crs g,b\crs h}=\ggen{b\crs
h,\Sigma(a\crs g)}$ when $h$ is an element of~$G$ different from~$g^{-1}$,
simply because both sides of the equality are equal to zero. This shows
that the map~$\Sigma$ is the Nakayama automorphism
of~$A\crs_{(\omega,\alpha)}G$ with respect to the
form~$\ggen{\place,\place}$.
\end{proof}

The Jacobian map $\jac_\sigma^G:G\to A$ that appears in this lemma has many
properties similar to those of the Nakayama Jacobian that we have been
studying. We will not look into this here. Let us just remark that, much as
in Lemma~\ref{lemma:jac:cocycle}, we can obtain a formula
for~$\jac_\sigma^G(gh)$ in terms of~$\jac_\sigma^G(g)$
and~$\jac_\sigma^G(h)$ that is a complicated version of a cocycle condition
and that from it we can see that $\jac_\sigma^G$ takes values that are
units in~$A$.

\subsubsection{Some reductions}
\label{subsubsect:jac:reductions}

Let us finish our discussion of Jacobians by presenting three useful
observations that simplify their calculation, and that arise from
well-known properties of the class of Frobenius algebras already studied by
Samuel Eilenberg and Tadasi Nakayama in~\cite{EN}. The first one is that we
can compute the Jacobian of an automorphism of a decomposable algebra that
preserves one of its decompositions one component at a time.

\begin{Lemma}\label{lemma:jac:times}
Let $A$ be an associative algebra, let $\gen{\place,\place}:A\times
A\to\kk$ be an non-degenerate and associative bilinear form on~$A$, and let
$\sigma:A\to A$ be the corresponding Nakayama automorphism. Let moreover
$A_1$,~\dots,~$A_n$ be ideals of~$A$ such that $A=A_1\times\cdots\times
A_n$.
\begin{thmlist}

\item Let $i$ be an element of~$\inter{n}$. The
restriction~$\gen{\place,\place}_i$ of the bilinear
form~$\gen{\place,\place}$ to~$A_i$ is non-degenerate and associative, so
that $A_i$ is a Frobenius algebra with respect to it. We have
$\sigma(A_i)\subseteq A_i$ and the
restriction~$\sigma_i\coloneqq\sigma|_{A_i}:A_i\to A_i$ is the Nakayama
automorphism of~$A_i$ corresponding to~$\gen{\place,\place}_i$.

\item The Jacobian of an automorphism $u:A\to A$ such that $u(A_i)\subseteq
A_i$ for each $i\in\inter{n}$, so that we can consider the restriction
$u|_{A_i}:A_i\to A_i$, is
  \[ \label{eq:jac:times}
  \jac_\sigma(u) = \jac_{\sigma_1}(u|_{A_1})
                   + \cdots
                   + \jac_{\sigma_n}(u|_{A_n}).
  \]

\end{thmlist}
\end{Lemma}

\begin{proof}
\thmitem{1} Let us remark that the ideals~$A_1$,~\dots,~$A_n$ are
orthogonal with respect to the bilinear form~$\gen{\place,\place}$. Indeed,
if $i$ and~$j$ are two different elements of~$\inter{n}$, $a$ and~$b$ are
elements of~$A_i$ and of~$A_j$, respectively, and~$1_j$ is the unit element
of~$A_j$, then we have that
$\gen{a,b}=\gen{a,1_jb}=\gen{a1_j,b}=\gen{0,b}=0$. This  implies at once
that the restrictions
$\gen{\place,\place}_1$,~\dots,~$\gen{\place,\place}_n$ are non-degenerate,
and it is clear that they are all associative, since the
form~$\gen{\place,\place}$ is. 

Let $i$ be an element of~$\inter{n}$, let $\sigma_i:A_i\to A_i$ be the
Nakayama automorphism of~$A_i$ with respect to the
form~$\gen{\place,\place}_i$, and let $a$ be an element of~$A_i$.
\begin{itemize}

\item For all $b\in A_i$ we have that
  \begin{align}
  \gen{b,\sigma(a)-\sigma_i(a)}
       &= \gen{b,\sigma(a)}-\gen{b,\sigma_i(a)}
        = \gen{a, b}-\gen{b,\sigma_i(a)}_i
        = \gen{a, b}-\gen{a, b}_i \\
       &= \gen{a, b}-\gen{a, b}
        = 0.
  \end{align}

\item If instead $j$ is an element of~$\inter{n}$ different from~$i$ and
$b$ is in~$A_j$, then we have
  \[
  \gen{b,\sigma(a)-\sigma_i(a)}
        = \gen{b,\sigma(a)} - \gen{b,\sigma_i(a)}
        = \gen{a,b} 
        = 0
  \]
because $A_i$ and~$A_j$ are orthogonal with respect
to~$\gen{\place,\place}$. 

\end{itemize}
It follows from this that, in fact, $\sigma_i(a)=\sigma(a)$.

\medskip

\thmitem{2} Let $u:A\to A$ be an automorphism of~$A$ such that
$u(A_i)\subseteq A_i$ for each $i\in\inter{n}$. If $i\in\inter{n}$, $a\in
A_i$, and $1_i$ is the unit element of~$A_i$, then we have that
  \begin{align}
  \gen{u(a),1} 
        &= \gen{u_i(a),1_i} 
         = \gen{u|_{A_i}(a),1}_i
         = \gen{\jac_{\sigma_i}(u|_{A_i}),a}_i \\
        &= \gen[\big]{
                        \jac_{\sigma_1}(u|_{A_1})
                        + \cdots
                        + \jac_{\sigma_n}(u|_{A_n}),
                        a}.
  \end{align}
This shows that the equality~\eqref{eq:jac:times} is true.
\end{proof}

A second, very easy observation that is useful in our context is that we
can extend scalars in our algebras and that doing that does not change, in
the appropriate sense, the Jacobians we are trying to compute.

\begin{Lemma}\label{lemma:jac:ext}
Let $A$ be an associative algebra, let $\gen{\place,\place}:A\times
A\to\kk$ be an non-degenerate and associative bilinear form on~$A$, and let
$\sigma:A\to A$ be the corresponding Nakayama automorphism. Let moreover
$K$ be an extension of our ground field~$\kk$ and let $A_K\coloneqq
K\otimes A$ be the algebra obtained from~$A$ by extension of scalars.
\begin{thmlist}

\item There is a unique $K$-bilinear form $\gen{\place,\place}_K:A_K\times
A_K\to K$ on the algebra~$A_K$ such that 
  \(
  \gen{s\otimes a,t\otimes b}=st\gen{a,b}
  \)
for all $s$,~$t\in K$ and all~$a$,~$b\in A$, it is non-degenerate and
associative, so that $A_K$ is a Frobenius algebra, and the corresponding
Nakayama automorphism is $\sigma_K\coloneqq\id_K\otimes\sigma:A_K\to A_K$.

\item If $u:A\to A$ is an automorphism of~$A$, then of course
$u_K\coloneqq\id_K\otimes u:A_K\to A_K$ is an automorphism of~$A_K$, and we
have that
  \[
  \jac_{\sigma_K}(u_K) = 1 \otimes \jac_\sigma(u).
  \]

\end{thmlist}
\end{Lemma}

\begin{proof}
The first part of the lemma is clear, and the second one is a consequence
of the fact that 
  \begin{align}
  \gen{u_K(s\otimes a),u_K(t\otimes b)}
       &= st\cdot\gen{u(a),u(b)}
        = st\cdot\gen{\jac_\sigma(u)\cdot a,b} \\
       &= \gen{1\otimes\jac_\sigma(u)\cdot s\otimes a,t\otimes b}
  \end{align}
for all $s$,~$t\in K$ and all $a$~$b\in A$.
\end{proof}

Lemma~\ref{lemma:jac:ext} implies that if $A$ is an algebra that is
actually defined over a subfield of~$\kk$, then the Jacobian of an
automorphism $u:A\to A$ of~$A$ that is also defined over~$\kk$ is itself
defined over~$\kk$. This suggests studying the problem of Galois descent
for Jacobians and various integrality questions --- we will do that
elsewhere.

\bigskip

The third observation we want to make is more subtle. We can view an
algebra as an algebra over any of its central subfields, and it turns out
that changing the field in this way affects neither the Frobenius property
--- as long as the algebra remains finite-dimensional over the new field,
of course --- nor the Jacobians of automorphisms.

\begin{Lemma}\label{lemma:jac:change}
Let $A$ be a finite-dimensional algebra, and let $K$ be a central subfield
of~$A$ that contains~$\kk$.
\begin{thmlist}

\item $A$ is Frobenius as a $\kk$-algebra if and only if it is Frobenius as
a $K$-algebra. 

\end{thmlist}
Let now $\epsilon:K\to \kk$ be a non-zero linear map, and let
$\gen{\place,\place}_\kk:A\times A\to\kk$ and $\sigma:A\to A$ be a
non-degenerate and associative $\kk$-bilinear form on~$A$ and the
corresponding Nakayama automorphism, respectively.
\begin{thmlist}[resume*]

\item There is a unique non-degenerate and associative $K$-bilinear form
$\gen{\place,\place}_K:A\times A\to K$ on~$A$ such that
$\gen{a,b}_K=\epsilon(\gen{a,b}_\kk)$ for all $a$ and all~$b$ in~$A$, and
the Nakayama automorphisms of~$A$ with respect to it is precisely~$\sigma$.

\item The Jacobian of an automorphism $u:A\to A$ of~$A$ as a $K$-algebra
coincides with its Jacobian as an automorphism of~$A$ as a~$\kk$-algebra.

\end{thmlist}
\end{Lemma}

The claim in~\thmitem{1} is usually proved by noting that $\Hom_K(A,K)$ is
an injective envelope of~$A/\rad A$ in the category of left $A$-modules, so
that its isomorphism type there does not depend on the choice of the
central subfield~$K$ and, in particular, that it is free exactly when
$\Hom_\kk(A,\kk)$ is free. This is how Tsit Yuen Lam does it in his
\cite{Lam:lectures}*{Example 3.42}, for example --- we need explicit
information about the bilinear forms involved so we do it in a different
way.

\begin{proof}
Let us fix an arbitrary non-zero $\kk$-linear map $\epsilon:K\to\kk$.

\thmitem{1} A
$K$-linear function~$f:A\to K$ is automatically $\kk$-linear, so the
composition $\epsilon\circ f$ belongs to~$\Hom_\kk(A,\kk)$, and we
therefore have a function 
  \[
  \Psi:f\in\Hom_K(A,K)\mapsto \epsilon\circ f\in\Hom_\kk(A,\kk).
  \]
The domain and the codomain of~$\Psi$ are central $A$-bimodules and the map
is a  morphism of $A$-bimodules: in particular, the domain and codomain
of~$\Psi$ are $\kk$-vector spaces in a canonical way, both are
finite-dimensional over~$\kk$, and the map $\Psi$ is $\kk$-linear. The
map~$\Psi$ is also injective: indeed, if $f:A\to K$ is a non-zero
$K$-linear function, then $f$ is surjective and ---~since $\epsilon$ is not
the zero map --- we can choose an element~$a$ in~$A$ such that $f(a)$ is in
$K\setminus\ker\epsilon$, so that $\psi(f)(a)\neq0$. As
  \begin{align}
  \dim_\kk \Hom_\kk(A,\kk)
       &= \dim_\kk A
        = \dim_KA\cdot\dim_\kk K
        = \dim_K\Hom_K(A,K)\cdot\dim_\kk K
        \\
       &= \dim_\kk\Hom_K(A,K)
  \end{align}
we see that the $K$-linear map~$\Psi$ is in fact bijective, so it is an
isomorphism of $A$-bimodules. It follows from this, of course, that
$\Hom_\kk(A,\kk)$ is a free left or right $A$-module exactly when
$\Hom_K(A,K)$ is a free left or right $A$-module, respectively, and this
means that $A$ is a Frobenius $\kk$-algebra exactly when it is a Frobenius
$K$-algebra.

\medskip

\thmitem{2} The map $\beta_\kk:A\to\Hom_\kk(A,\kk)$ such that
$\beta_\kk(a)(b)=\gen{a,b}$ for all $a$ and all~$b$ in~$A$ is an
isomorphism of right $A$-modules and, since the map $\Psi$ is bijective,
there is an element $\lambda$ in~$\Hom_K(A,K)$ such that
$\Psi(\lambda)=\beta_\kk(1)$. We can therefore consider the $K$-bilinear
form $\gen{\place,\place}_K:(a,b)\in A\times A\mapsto\lambda(ab)\in K$,
which is clearly associative, and the corresponding right $A$-linear map
$\beta_K:A\to\Hom_K(A,K)$ such that $\beta_K(a)(b)=\gen{a,b}_K$ for all $a$
and all~$b\in A$. For all $a\in A$ we have that $\beta_K(a)=\beta_\kk(a)$,
because
  \[
  \beta_K(a)(b)
        = \lambda(ab)
        = \beta_\kk(1)(ab)
        = (\beta_\kk(1)a)(b)
        = \beta_\kk(a)(b)
  \]
for all $b\in A$, and this tells us that $\beta_K$ is bijective, so that
the $K$-bilinear form~$\gen{\place,\place}_K$ is non-degenerate and $A$ is
a Frobenius $K$-algebra with respect to it. As
$\Psi(\lambda)=\beta_\kk(1)$, when $a$ and~$b$ are in~$A$ we have that
  \[
  \gen{a,b}_\kk
        = \beta_\kk(a)(b)
        = (\beta_\kk(1)a)(b)
        = \beta_\kk(1)(ab)
        = \Psi(\lambda)(ab)
        = \epsilon(\lambda(ab))
        = \epsilon(\gen{a,b}_K).
  \]
We know from Proposition~\ref{prop:sigma:central} that the Nakayama
automorphism $\sigma:A\to A$ corresponding to~$\gen{\place,\place}_\kk$
acts trivially on the center of~$A$, so it is in particular an automorphism
of~$A$ as a $K$-algebra. If $a$ is an element of~$A$, then we have that
  \[
  \epsilon(\gen{b,\sigma(a)}_K)
        = \gen{b,\sigma(a)}_\kk
        = \gen{a,b}_\kk
        = \epsilon(\gen{a,b}_K)
  \]
for all $b\in B$, so that the $K$-linear function $b\in
A\mapsto\gen{b,\sigma(a)}_K-\gen{a,b}_K\in K$ takes values in the kernel
of~$\epsilon$: as its image is a $K$-subspace of~$K$ and $\epsilon\neq0$,
this implies, of course, that it is identically zero. We see with this that
$\gen{a,b}_K=\gen{b,\sigma(a)}_K$ for all $a$ and all $b$ in~$K$, so that
$\sigma$ is also the Nakayama automorphism of~$A$ with respect to the
$K$-bilinear form $\gen{\place,\place}_K$. This completes the proof
of~\thmitem{2}.

\medskip

\thmitem{3} Let $u:A\to A$ be an automorphism of~$A$ as a $K$-algebra and
let $\jac_\sigma(u)$ be the corresponding Jacobian, so that
$\gen{u(a),u(b)}_K=\gen{\jac_\sigma(u)\cdot a,b}_K$ for all $a$ and all $b$
in~$A$. This implies, of course, that 
  \[
  \gen{u(a),u(b)}_\kk
        = \epsilon(\gen{u(a),u(b)}_K)
        = \epsilon(\gen{\jac_\sigma(u)\cdot a,b}_K)
        = \gen{\jac_\sigma(u)\cdot a,b}_\kk
  \]
for all $a$ and all $b$ in~$A$, so that $\jac_\sigma(u)$ is also the
Jacobian of~$u$ when we regard it as an automorphism of~$A$ as
a~$\kk$-algebra.
\end{proof}

\subsection{Divergences}
\label{subsect:div}

As usual, we let $A$ be a Frobenius algebra with respect to a
non-degenerate and symmetric bilinear form $\gen{\place,\place}:A\times
A\to\kk$, and write $\sigma:A\to A$ for the corresponding Nakayama
automorphism, and $\beta:A\to{}_\sigma\D A$ for the associated isomorphism
of bimodules, which is such that $\beta(a)(b)=\gen{a,b}$ for all $a$ and
all~$b$ in~$A$.

\begin{Lemma}
Let $\delta:A\to A$ be a derivation. 
\begin{thmlist}

\item There is a unique linear map $\delta^*:A\to A$ that for all $a$ and
all~$b$ in~$A$ is such that
  \[
  \gen{\delta(a),b} = \gen{a,\delta^*(b)}.
  \]

\item For all $a$ and~$b$ in~$A$ we have that
  \[
  \delta^*(ab) = a\delta^*(b) - \delta(a)b,
  \qquad
  \delta^*(ab) = \delta^*(a)b - a\delta^\sigma(b),
  \]
so that $\delta^*$ is a left $(-\delta)$-operator and a right
$(-\delta^\sigma)$-operator.

\end{thmlist}
\end{Lemma}

\begin{proof}
The first claim is a easy consequence of the non-degeneracy of the bilinear
form~$\gen{\place,\place}$ on~$A$. If $a$ and~$b$ are two elements of~$A$,
then we have that
  \begin{align}
  \gen{c,\delta^*(ab)}
        &= \gen{\delta(c),ab}
         = \gen{\delta(c)a,b}
         = \gen{\delta(ca),b} - \gen{c\delta(a),b}
         = \gen{ca,\delta^*(b)} - \gen{c,\delta(a)b} \\
        &= \gen{c,a\delta^*(b) - \delta(a)b}
\shortintertext{and}
  \gen{c,\delta^*(ab)}
        &= \gen{\delta(c),ab}
         = \gen{\sigma^{-1}(b)\delta(c),a}
         = \gen{\delta(\sigma^{-1}(b)c),a} - \gen{\delta(\sigma^{-1}(b))c,a}  \\
        &= \gen{\sigma^{-1}(b)c,\delta^*(a)} - \gen{c,a\sigma(\delta(\sigma^{-1}(b)))} 
         = \gen{c,\delta^*(a)b - a\sigma(\delta(\sigma^{-1}(b)))} 
  \end{align}
for all $c\in A$, so that $\delta^*(ab)=a\delta^*(b)-\delta(a)b$ and
$\delta^*(ab)=\delta^*(a)b-a\delta^\sigma(b)$.
\end{proof}

In the situation of the lemma, we define the \newterm{Nakayama divergence}
of the derivation~$\delta:A\to A$ to be the element
  \[
  \div_\sigma(\delta) \coloneqq \delta^*(1).
  \]
The second part of the lemma implies at once that $\div_\sigma(\delta)$
determines the map~$\delta^*$: indeed, for all $a\in A$ we have that
  \[ \label{eq:deltastar}
  \delta^*(a) = a\cdot\div_\sigma(\delta) - \delta(a),
  \qquad
  \delta^*(a) = \div_\sigma(\delta)\cdot a - \delta^\sigma(a).
  \]
In particular, we have that
  \(
  \delta^\sigma(a)-\delta(a) = \div_\sigma(\delta)\cdot a - a\cdot\div_\sigma(\delta)
  \)
for all $a\in A$, so that in fact we have
  \[ \label{eq:div:comm}
  \delta^\sigma - \delta = \ad(\div_\sigma(\delta)).
  \]
We know from Subsection~\ref{subsect:hh01} that the classes of the
derivations~$\delta:A\to A$ and~$\delta^\sigma:A\to A$ in~$\HH^1(A)$
coincide, so that the difference~$\delta^\sigma-\delta:A\to A$ is an inner
derivation: the equality~\eqref{eq:div:comm} shows that the
divergence~$\div_\sigma(\delta)$ is a particular choice on an element that
proves this.

\begin{Lemma}\label{lemma:div:ids}
Let $\delta:A\to A$ be a derivation. 
\begin{thmlist}

\item For all $a$ and~$b$ in~$A$ we have that
  \begin{align}
  &  \gen{\delta(a),b} + \gen{a,\delta(b)} = \gen{a,b\cdot\div_\sigma(\delta)},
  && \gen{\delta(a),b} + \gen{a,\delta^\sigma(b)} = \gen{a,\div_\sigma(\delta)\cdot b}
\intertext{and, in particular,}
  &  \gen{\delta(a),1} = \gen{a,\div_\sigma(\delta)}.
  \end{align}

\item The element $\div_\sigma(\delta)$ is the unique element of~$A$ that
makes any one of the three equalities in~\thmitem{1} true for all $a$ and
all~$b$ in~$A$.

\end{thmlist}
\end{Lemma}

\begin{proof}
The truth of the first two equalities in~\thmitem{1} follows immediately
from the equalities in~\eqref{eq:deltastar}, and that of the third one from
them by taking $b=1$ and recalling that $\delta$ vanishes on~$1$. On the
other hand, the claim in \thmitem{2} follows at once from the
non-degeneracy of the bilinear form on~$A$.
\end{proof}

We can easily compute the divergence of inner derivations:

\begin{Lemma}\label{lemma:div:inner}
Let $x$ be an element of~$A$ and let 
  \(
  \ad(x):a\in A\mapsto xa-ax\in A
  \)
be the inner derivation associated with~$x$. We have
  \[
  \div_\sigma(\ad(x)) = \sigma(x)-x.
  \]
\end{Lemma}

Notice that this tells us that $\div_\sigma$ vanishes on inner derivation
when the algebra~$A$ is symmetric and $\sigma=\id_A$.

\begin{proof}
Indeed, for all $a\in A$ we have that
  \[
  \gen{\ad(x)(a),1} 
        = \gen{xa,1} - \gen{ax,1}
        = \gen{a,\sigma(x)} - \gen{a,x}
        = \gen{a,\sigma(x)-x},
  \]
so the lemma follows from the characterization of~$\div_\sigma(\ad(x))$
given in Lemma~\ref{lemma:div:ids}.
\end{proof}

As with divergences of vector fields in calculus, $\div$ is not a linear
map but a differential operator:

\begin{Lemma}\label{lemma:div:connection}
Let $\delta:A\to A$ be a derivation of~$A$. For each central element~$z$
of~$A$ we have that
  \[
  \div_\sigma(z\delta) = z\div_\sigma(\delta) - \delta(z).
  \]
\end{Lemma}

Notice that this makes sense because $\Der(A)$ is, in a canonical way, a
left module over the center~$\Z(A)$ of~$A$.

\begin{proof}
Let $z$ be a central element of~$A$. Whenever $a$ and~$b$ are two arbitrary
elements of~$A$ we have that
  \begin{align}
  \gen{(z\delta)(a),b} + \gen{a,(z\delta)(b)}
        &= \gen{z\delta(a),b} 
                + \gen{a,z\delta(b)} \\
        &= \gen{\delta(a),bz} 
                + \gen{az,\delta(b)} \\
\intertext{because $\sigma$ fixes~$z$, and this is}
        &= -\gen{a,\delta(bz)} 
                + \gen{a,bz\cdot\div_\sigma(\delta)} 
                + \gen{az,\delta(b)} \\
        &= -\gen{a,\delta(b)z} 
                - \gen{a,b\delta(z)} 
                + \gen{a,bz\cdot\div_\sigma(\delta)} 
                + \gen{az,\delta(b)} \\
        &= \gen{a,b\cdot(-\delta(z) 
                + z\cdot\div_\sigma(\delta))}.
  \end{align}
The equality in the lemma follows from this.
\end{proof}

Another remarkable fact is that the operator~$\div_\sigma$ interacts nicely
with the Lie bracket on~$\Der(A)$, for an appropriately adjusted meaning of
the word «nice».

\begin{Lemma}\label{lemma:div:cocycle}
If $\delta$,~$\eta:A\to A$ are derivations of~$A$, then
  \[
  \div_\sigma([\delta,\eta])
        = \delta(\div_\sigma(\eta)) - \eta(\div_\sigma(\delta))
                + [\div_\sigma(\delta),\div_\sigma(\eta)].
  \]
\end{Lemma}

This is not quite the formula for the divergence of the commutator of two
vector fields on a Riemannian manifold, as in that context we do not have
the third summand of the right hand side. That term goes away, of course,
if $A$ is commutative, for example.

\begin{proof}
Let $\delta$,~$\eta:A\to A$ be two derivations of~$A$. For all $a\in A$ we
have that
  \begin{align}
  \MoveEqLeft[1]
  \gen{a,\div_\sigma([\delta,\eta])}
         = \gen{[\delta,\eta](a),1} \\
        &= \gen{\delta(\eta(a)),1} - \gen{\eta(\delta(a)),1} \\
        &= \gen{\eta(a),\div_\sigma(\delta)}-\gen{\delta(a),\div_\sigma(\eta)} \\
        &= -\gen{a,\eta(\div_\sigma(\delta)}
           +\gen{a,\div_\sigma(\delta)\cdot\div_\sigma(\eta)}
           +\gen{a,\delta(\div_\sigma(\eta))}
           -\gen{a,\div_\sigma(\eta)\cdot\div_\sigma(\delta)} \\
        &= \gen{a,\delta(\div_\sigma(\eta))
           -\eta(\div_\sigma(\delta))
           +[\div_\sigma(\delta),\div_\sigma(\eta)]}
  \end{align}
thanks for the identities in Lemma~\ref{lemma:div:ids}, and therefore the
equality in the lemma holds.
\end{proof}

We can state the result in Lemma~\ref{lemma:div:cocycle} in an interesting
way. The Lie algebra $\Der(A)$ acts on~$A$ in a canonical way, and with
that action $A$ is a Lie module, so we can construct the familiar
Chevalley-Eilenberg complex~$\CE^*(\Der(A),A)=\Hom(\bigwedge^*\Der(A),A)$
that computes the Lie algebra cohomology~$\H^*(\Der(A),A)$ of~$\Der(A)$
with values in~$A$. We can turn~$A$ into a Lie algebra using the commutator
bracket, and using the formulas of Schouten--Nijenhuis we can turn the
complex $\CE^*(\Der(A),A)$ into a differential graded Lie algebra in a
canonical way. Using this formalism we can restate
Lemma~\ref{lemma:div:cocycle} as follows.

\begin{Corollary}\pushQED{\qed}\label{coro:div:MC}
Set us suppose that $2\neq0$ in~$\kk$. The $1$-cochain
$\div_\sigma:\Der(A)\to A$ in the differential graded Lie
algebra~$\CE^*(\Der(A),A)$ is a solution to the Maurer--Cartan equation,
that is, we have
  \begin{equation*}
  d\div_\sigma + \frac{1}{2}[\div_\sigma,\div_\sigma] = 0. \qedhere
  \end{equation*}
\end{Corollary}

In general, the map $\div_\sigma:\Der(A)\to A$ is \emph{not} a $1$-cocycle
in~$\CE^*(\Der(A),A)$. It is, though, when the algebra $A$ is symmetric and
we have $\sigma=\id_A$. The following lemma describes what we have in that
case.

\begin{Lemma}\label{lemma:DIV}
Let us suppose that the bilinear form~$\gen{\place,\place}$ is symmetric,
so that $\sigma=\id_A$.
\begin{thmlist}

\item For all derivations~$\delta:A\to A$ the element~$\div_\sigma(\delta)$
is central in~$A$.

\item The map $\div_\sigma:\Der(A)\to\Z(A)$ that we therefore have in this
situation vanishes on inner derivations, so it induces a linear map
$\div_\sigma:\HH^1(A)\to\Z(A)$. The latter is a $1$-cocycle in the
complex~$\CE^*(\HH^1(A),\Z(A))$ that gives rise to a cohomology class
$\DIV(A)$ in the Lie algebra cohomology~$\H^1(\HH^1(A),\Z(A))$
of~$\HH^1(A)$ with values in~$\Z(A)$.

\end{thmlist}
\end{Lemma}

Here we are viewing the Lie algebras $\Der(A)$ and~$\HH^1(A)$ acting
on~$\Z(A)$ in the tautological way. In particular, the action of~$\HH^1(A)$
on~$\HH^0(A)=\Z(A)$ is given by the Gerstenhaber bracket.

\begin{proof}
If $\delta:A\to A$ is a derivation of~$A$, then we know from
equation~\eqref{eq:div:comm} that
  \[
  \ad(\div_\sigma(\delta))=\delta^\sigma-\delta=0,
  \]
so that $\div_\sigma(\delta)$ is a central element of~$A$. This
proves~\thmitem{1} and shows that we can restrict the codomain
of~$\div_\sigma:\Der(A)\to A$ to obtain a linear map $\Der(A)\to\Z(A)$. We
know from Lemma~\ref{lemma:div:inner} that this map vanishes on inner
derivations, so it induces another map $\HH^1(A)\to\Z(A)$. Since central
elements are, well, central, we know from the identity in
Lemma~\ref{lemma:div:cocycle} that this map is a $1$-cocycle
in~$\CE^*(\HH^1(A),\Z(A))$.
\end{proof}

We will present below examples that show that the class $\DIV(A)$
in~$\H^1(\HH^1(A),\Z(A))$ is in general not trivial.

\subsubsection{An example: the Grassmann algebra}
\label{subsubsect:div:grass}

We go back to the setting of Subsubsection~\ref{subsubsect:jac:grass}: we
suppose that the characteristic of our ground field~$\kk$ is not~$2$, let
$V$ be a finite-dimensional vector space, write~$n$ for its the dimension,
choose an arbitrary ordered basis $\B=(x_1,\dots,x_n)$ for~$V$, and
consider the exterior algebra $A\coloneqq\Lambda(V)$ on~$V$, endowed with
its usual $\ZZ$- and $\ZZ/2\ZZ$-gradings.

The Lie algebra of derivations of~$A$ was described by Dragomir {\v{Z}}.
{\DJ}okovi{\'c} in~\cite{Djokovic} and later, in a more general context, by
Vladimir V. Bavula in~\cite{Bavula:der}. We prefer to refer to this second
paper, for its results are stated in a way more convenient to us. There the
following statements are obtained as part of \cite{Bavula:der}*{Theorem 2.1}.
\begin{itemize}

\item Let us write $\Der(A)^\even$ and~$\Der(A)^\odd$ for the subspaces
of~$\Der(A)$ of those derivations that preserve the $\ZZ/2\ZZ$-grading and
that flip it. We have 
  \[
  \Der(A)=\Der(A)^\even\oplus\Der(A)^\odd
  \]
and that this decomposition turns~$\Der(A)$ into a $\ZZ/2\ZZ$-graded Lie
algebra (but not a Lie super-algebra!)

\item The ideal $\Der(A)^\odd$ coincides with the ideal $\InnDer(A)$ of all
inner derivations of~$A$, and therefore the Lie algebra of outer
derivations of~$A$ or, equivalently, its first Hochschild cohomology Lie
algebra, is
  \[
  \HH^1(A) \coloneqq \OutDer(A) = \frac{\Der(A)}{\InnDer(A)} \cong \Der(A)^\even.
  \]

\item If $\partial_1$,~\dots,~$\partial_n:A\to A$ are the left skew
derivations that we defined in Subsubsection~\ref{subsubsect:jac:grass} on
page~\pageref{def:lsder} and we write, as we did there, $A^\odd$ for the
odd part of~$A$, then
  \[
  \Der(A)^\even = A^\odd\partial_1\oplus\cdots\oplus A^\odd\partial_n.
  \]
One should keep in mind the facts that $A^\odd$ is not a subalgebra of~$A$
and that the left skew derivations $\partial_1$,~\dots,~$\partial_n$ are
not derivations of~$A$, except in trivial cases.

\item The linear map
  \[
  a\in A^\odd \mapsto \ad(a)\in\InnDer(A)=\Der(A)^\odd
  \]
is surjective and its kernel is $A^\odd\cap\Z(A)$, which is $0$ if $n$ is
even and $A^n$ if $n$ is odd.

\end{itemize}

\bigskip

With this information we can compute the divergence of all derivations of
the algebra~$A$. Interestingly, that divergence is, for the most part,
given by familiar formula from calculus.

\begin{Lemma}
\begin{thmlist}

\item If $a$ is an element of~$A^\odd$, then
  \[
  \div_\sigma(\ad(a)) = \begin{cases*}
                        -2a     & if $n$ is even; \\
                        0       & if $n$ is odd.
                        \end{cases*}
  \]

\item If $a_1$,~\dots,~$a_n$ are elements of~$A^\odd$, then
  \[
  \div_\sigma(a_1\partial_1+\cdots+a_n\partial_n)
        = \partial_1(a_1)+\cdots+\partial_n(a_n).
  \]

\end{thmlist}
\end{Lemma}

\begin{proof}
We know that the Nakayama automorphism of~$A$ is such that
$\sigma(x)=(-1)^{n-1}x$ for all $x\in V$, so that $\sigma$ acts by
multiplication by~$(-1)^{n-1}$ on all of~$A^\odd$: the claim in~\thmitem{1}
follows from this and Lemma~\ref{lemma:div:inner}.

In order to prove the claim in~\thmitem{2} we will start by showing that
for all choices of~$i$ and~$j$ in~$\inter{n}$ we have that
  \[ \label{eq:div:dd}
  \div_\sigma(x_i\partial_j) = \delta_{i,j}.
  \]
According to Lemma~\ref{lemma:div:ids} and the linearity of everything in
sight to do that we need to fix two elements~$u$ and~$v$ in~$\inter{n}$ and
two strictly increasing sequences $r_1$,~\dots,~$r_u$
and~$s_1$,~\dots,~$s_v$ of elements of~$\inter{n}$, and show that
  \begin{multline} \label{eq:div:xs}
  \gen{x_i\wedge\partial_j(x_{r_1}\wedge\cdots\wedge x_{r_u}), 
       x_{s_1}\wedge\cdots\wedge x_{s_v}}
  + \gen{x_{r_1}\wedge\cdots\wedge x_{r_u}, 
       x_i\wedge\partial_j(x_{s_1}\wedge\cdots\wedge x_{s_v})}
  \\
  = \delta_{i,j}
    \cdot\gen{x_{r_1}\wedge\cdots\wedge x_{r_u},x_{s_1}\wedge\cdots\wedge x_{s_v}}.
  \end{multline}
If $u+v$ is not~$n$, then the three terms here are zero, so we suppose that
it is. Let us consider the sequence~$(r_1,\dots,r_u,s_1,\dots,s_v)$.
\begin{itemize}

\item Let us suppose first that that sequence has repeated elements, so
that there are then $k\in\inter{n}$, $p\in\inter{u}$ and~$q\in\inter{v}$
such that $k=r_p=s_q$. In that case the right hand side
of~\eqref{eq:div:xs} is
  \[
  \gen{x_{r_1}\wedge\cdots\wedge x_{r_u},x_{s_1}\wedge\cdots\wedge x_{s_v}}
        = \int(x_{r_1}\wedge\cdots\wedge x_{r_u}\wedge x_{s_1}\wedge\cdots\wedge x_{s_v})
        = 0.
  \]
If $k\neq j$, then $\partial_j(x_{r_1}\wedge\cdots\wedge x_{r_u})$ and
$\partial_j(x_{s_1}\wedge\cdots\wedge x_{s_v})$ are both scalar multiples
of monomials that involve~$x_k$, and this implies that the two terms that
appear on the left of the equality~\eqref{eq:div:xs} vanish. If instead
$k=j$, then the left hand side of that equality is
  \begin{multline}
  (-1)^{p-1}\int(x_i\wedge x_{r_1}\wedge\cdots\wedge x_{r_{p-1}}\wedge
  x_{r_{p+1}}\wedge\cdots\wedge x_{r_u}\wedge x_{s_1}\wedge\cdots\wedge
  x_{s_v}) \\
  + (-1)^{q-1}\int(x_{r_1}\wedge\cdots\wedge x_{r_u}\wedge x_i
        \wedge x_{s_1}\wedge\cdots\wedge x_{s_{q-1}}\wedge x_{s_{q+1}}
        \wedge \cdots \wedge x_{v})
  \end{multline}
and it is easy to see --- by permuting the factors in the second summand
and keeping track of the signs introduced by this --- that these two terms
cancel each other.

\item Let us now suppose that that sequence has no repeated elements, so
that it is in fact a permutation of~$\inter{n}$, and let $\epsilon$ be the
sign of that permutation. The right hand side of the
equality~\eqref{eq:div:xs} is now
  \[
  \gen{x_{r_1}\wedge\cdots\wedge x_{r_u},x_{s_1}\wedge\cdots\wedge x_{s_v}}
        = \int(x_{r_1}\wedge\cdots\wedge x_{r_u}\wedge x_{s_1}\wedge\cdots\wedge x_{s_v})
        = \delta_{i,j}\cdot\epsilon.
  \]
As the sequence is a permutation, the integer~$j$ belongs to exactly one of
the sets $\{r_1,\dots,r_u\}$ and~$\{s_1,\dots,s_v\}$. Let us suppose, for
example, that it belongs to the first one, so that there is a $p$
in~$\inter{u}$ such that $j=r_p$. In that case
$\partial_j(x_{s_1}\wedge\cdots\wedge x_{s_v})=0$ and the left hand side of
the equality~\eqref{eq:div:xs} is
  \begin{multline}
  \gen{x_i\wedge\partial_j(x_{r_1}\wedge\cdots\wedge x_{r_u}), 
       x_{s_1}\wedge\cdots\wedge x_{s_v}} \\
  = (-1)^{p-1}\int(
        x_i
        \wedge x_{r_1}
        \wedge \cdots
        \wedge x_{r_{p-1}}
        \wedge x_{r_{p+1}}
        \wedge \cdots
        \wedge x_{r_u}
        \wedge x_{s_1}
        \wedge \cdots
        \wedge x_{s_v}
        ).
  \end{multline}
If $i\neq j$, then this is zero because the «integrand» is a monomial that
does not involve~$x_j$, and if $i=j$ then rearranging the factors in that
integrand we easily see that this is $\epsilon$.

\end{itemize}
The conclusion of all this is that the equality~\eqref{eq:div:xs} holds in
all cases, so that the identity~\eqref{eq:div:dd} is true.

\bigskip

Now we want to show that for all $a\in A^\odd$ and all $i\in\inter{n}$ we
have that
  \[
  \div_\sigma(a\partial_i)=\partial_i(a),
  \]
and it is enough to do this when $a=x_{j_1}\wedge\cdots\wedge x_{j_k}$ for
some odd integer~$k$ and some choice of~$j_1$,~\dots,~$j_k$ in~$\inter{n}$
such that $j_1<\cdots<j_k$. Let us suppose then that that is the case. As
$x_{j_1}\wedge\cdots \wedge x_{j_{k-1}}$ is a central element of~$A$,
Lemma~\ref{lemma:div:connection} tells us that
  \[ \label{eq:div:ddd}
  \div_\sigma(a\partial_i)
        = x_{j_1}\wedge\cdots\wedge x_{j_{k-1}}\wedge\div_\sigma(x_{j_k}\partial_i)
          - x_{j_k}\wedge\partial_i(x_{j_1}\wedge\cdots\wedge x_{j_{k-1}}).
  \]        
We consider three different cases.
\begin{itemize}

\item Suppose first that $i\not\in\{j_1,\dots,j_k\}$, so that
$\div_\sigma(x_{j_k}\partial_i)=0$ and
$\partial_i(x_{j_1}\wedge\cdots\wedge x_{j_{k-1}})=0$, and thus
$\div_\sigma(a\partial_i)=0$. As also $\partial_i(a)=0$, the equality we
want holds in this case.

\item Suppose next that $j_k=i$. Now we have that
$\div_\sigma(x_{j_k}\partial_i)=1$ because of~\eqref{eq:div:dd} and that
$\partial_i(x_{j_1}\wedge\cdots\wedge x_{j_{k-1}})=0$, so
from~\eqref{eq:div:ddd} we know that, as we want,
  \[
  \div_\sigma(a\partial_i) = x_{j_1}\wedge\cdots\wedge x_{j_{k-1}}
        = \partial_i(x_{j_1}\wedge\cdots\wedge x_{j_{k}}) = \partial_i(a).
  \]

\item Finally, suppose that $i\in\{j_1,\dots,j_{k-1}\}$, so that there is a
unique $\ell\in\inter{k-1}$ such that $i=j_l$. In this situation we have
that $\div_\sigma(x_{j_k}\partial_i)=0$, again because
of~\eqref{eq:div:dd}, and that
  \[
  \partial_i(x_{j_1}\wedge\cdots\wedge x_{j_{k-1}})
        = (-1)^{l-1}
          \cdot 
          x_{j_1}
          \wedge\cdots\wedge x_{j_{l-1}}
          \wedge x_{j_{l+1}}\wedge\cdots\wedge x_{j_{k-1}}.
  \]
This and~\eqref{eq:div:ddd} imply that
  \begin{align}
  \div_\sigma(a\partial_i)
       &= (-1)^{l} \cdot x_{j_k}\wedge
          x_{j_1}\wedge\cdots\wedge x_{j_{l-1}}
          \wedge x_{j_{l+1}}\wedge\cdots\wedge x_{j_{k-1}} \\
       &= (-1)^{l-1} \cdot 
          x_{j_1}\wedge\cdots\wedge x_{j_{l-1}}
          \wedge x_{j_{l+1}}\wedge\cdots\wedge x_{j_{k-1}}
          \wedge x_{j_k} \\
       &= \partial_i(a).
  \end{align}
\end{itemize}
This completes the proof of the lemma.
\end{proof}

This calculation allows us to show that the class~$\DIV(A)$
in~$\H^1(\HH^1(A),\Z(A))$ that we have from Lemma~\ref{lemma:DIV} is not
trivial --- remember that this class is defined when $A$ is symmetric, that
is, when $n$ is odd. 

\begin{Lemma}
Let us suppose that the integer~$n$ is odd. The cohomology class $\DIV(A)$
in~$\H^1(\HH^1(A),\Z(A))$ is not trivial.
\end{Lemma}

\begin{proof}
As we recalled above from Bavula's \cite{Bavula:der}*{Theorem 2.1}, we have
that
  \[
  \HH^1(A)
        = \Der(A)/\InnDer(A)
        \cong \Der(A)^\even
        = \bigoplus_{i=1}^nA^\odd\cdot\partial_i.
  \]
On the other hand our calculation tells us that the cocycle that gives rise
to the class~$\DIV(A)$ is
  \[
  \div_\sigma:\sum_{i=1}^na_i\partial_i\in\Der(A)^\even
        \mapsto \sum_{i=1}^n\partial_i(a_i)\in\Z(A).
  \]
Were this a coboundary, we would have an element~$z$ in~$\Z(A)$ such that
$\div_\sigma(\delta)=\delta(z)$ for all $\delta\in\Der(A)^\even$ and, in
particular, such that $1=\div_\sigma(x_1\partial_1)=x_1\partial_1(z)$. As
this is impossible, we see that the claim of the lemma is true.
\end{proof}

\subsubsection{An example: trivial extensions}
\label{subsubsect:div:trivial}

We go back to the situation of Subsubsection~\ref{subsubsect:jac:trivial}.
We let $B$ be a finite-dimensional algebra and write $A\coloneqq B\oplus \D
B$ for its trivial extension, which is a symmetric algebra with respect to
the non-degenerate, associative and symmetric bilinear form
$\gen{\place,\place}:A\times A\to\kk$ described there. We want to compute
the divergence of the derivations of~$A$, and to do that we have to start
by describing the derivations themselves.
\begin{itemize}

\item If $\delta:B\to B$ is a derivation of~$B$ and $\delta^\T:\D B\to\D B$
is the transpose map of~$\delta$, then the map
  \[
  \tilde\delta \coloneqq \begin{pmatrix}
                         \delta & 0 \\
                         0 & -\delta^\T
                         \end{pmatrix}
                         : B\oplus \D B\to B\oplus \D B
  \]
is a derivation of~$A$, and it is an inner derivation of~$A$ exactly when
$\delta$ is an inner derivation of~$B$. The map
  \[ \label{eq:tilde:delta}
  \delta\in\Der(B) \mapsto \tilde\delta\in\Der(A)
  \]
is an injective morphism of Lie algebras.

\item Similarly, if $z$ is a central element in~$B$ and we write $m_z:b\in
B\mapsto zb\in B$ for the map given by multiplication by~$z$ and~$m_z^\T:\D
B\to\D B$ for the transpose of~$m_z$, then the function
  \[
  \tilde z \coloneqq \begin{pmatrix}
                         0 & 0 \\
                         0 & m_z^\T
                         \end{pmatrix}
                         : B\oplus \D B\to B\oplus \D B
  \]
is a derivation of~$A$. The function
  \[ \label{eq:tilde:z}
  z\in\Z(B) \mapsto \tilde z\in\Der(A)
  \]
is linear and injective, and its image is an abelian Lie subalgebra of its
codomain.

\item If $\gamma:B\to\D B$ is a derivation, then the map
  \[
  \tilde \gamma \coloneqq \begin{pmatrix}
                         0 & 0 \\
                         \gamma & 0
                         \end{pmatrix}
                         : B\oplus \D B\to B\oplus \D B
  \]
is a derivation of~$A$. The function
  \[ \label{eq:tilde:gamma}
  \gamma\in\Der(B,\D B) \mapsto \tilde\gamma\in\Der(A)
  \]
is linear and injective, and its image is an abelian Lie subalgebra
of~$\Der(A)$.

\item Finally, let $\Alt(B)$ be the vector space of all morphisms $\beta:\D
B\to B$ of $B$-bimodules such that
  \[
  \beta(\lambda)\cdot\mu + \lambda\cdot\beta(\mu) = 0
  \]
for all choices of~$\lambda$ and~$\mu$ in~$\D B$. If $\beta$ is an element
of~$\Alt(B)$, then the map
  \[
  \tilde \beta \coloneqq \begin{pmatrix}
                         0 & \beta \\
                         0 & 0
                         \end{pmatrix}
                         : B\oplus \D B\to B\oplus \D B
  \]
is a derivation of~$A$. The function
  \[ \label{eq:tilde:beta}
  \beta\in\Alt(B) \mapsto \tilde\beta\in\Der(A)
  \]
is linear and injective, and its image is an abelian Lie subalgebra
of~$\Der(A)$ which we identify with~$\Alt(B)$.

\end{itemize}
In the following lemma and the rest of what follows we will regard the
maps~\eqref{eq:tilde:delta}, \eqref{eq:tilde:z}, \eqref{eq:tilde:gamma}
and~\eqref{eq:tilde:beta} as inclusions.

\begin{Lemma}
There is a direct sum decomposition
  \[ \label{eq:derta}
  \Der(A) = \Der(B) \oplus \Z(B) \oplus \Der(B,\D B) \oplus \Alt(B).
  \]
The ideal of inner derivations of~$A$ decomposes with respect to that
decomposition in the form
  \[ \label{eq:derta:inn}
  \InnDer(A) = \InnDer(B) \oplus 0 \oplus \InnDer(B,\D B) \oplus 0,
  \]
so that the first Hochschild cohomology of~$A$ is
  \[ \label{eq:derta:hh}
  \HH^1(A) \cong \HH^1(B) \oplus \Z(B) \oplus \H^1(B,\D B)\oplus\Alt(B).
  \]
There is, moreover, a canonical isomorphism $\H^1(B,\D B)\cong\HH_1(B)^*$.
\end{Lemma}

The description of~$\HH^1(A)$ in this lemma comes from the work of Claude
Cibils, Eduardo Marcos, María Julia Redondo and Andrea Solotar
in~\cite{CMRS}, where it arises as the low degree information coming out of
a spectral sequence converging to the whole Hochschild cohomology of~$A$.
We will sketch here a direct proof.

\begin{proof}
Let $d:A\to A$ be a derivation of~$A$. There are linear maps $\alpha:B\to
B$, $\beta:\D B\to B$, $\gamma:B\to\D B$ and~$\delta:\D B\to\D B$ such that
$d=\begin{psmallmatrix}\alpha&\beta\\\gamma&\delta\end{psmallmatrix}$. If
$x$ and~$y$ are elements of~$B$, we have that
  \[
  \alpha(xy)+\gamma(xy)
        = u(xy)
        = u(x)y + xu(y)
        = \alpha(x)y + x\alpha(y) + \gamma(x)y + x\gamma(y),
  \]
so the maps~$\alpha$ and~$\gamma$ are a derivation of~$B$ and an element
of~$\Der(B,\D B)$, respectively, and the difference 
  \(
  d'\coloneqq d-\tilde\alpha-\tilde\gamma
    = \begin{psmallmatrix}0&\beta\\0&\delta+\alpha^\T\end{psmallmatrix}
    : A\to A
  \)
is a derivation. On the other hand, if $x$ and~$\lambda$ are an element
of~$B$ and of~$\D B$, respectively, then we have that
  \[
  \beta(x\lambda) + (\delta+\alpha^\T)(x\lambda)
        = d'(x\lambda)
        = d'(x)\lambda + xd'(\lambda)
        = x\beta(\lambda) + x(\delta+\alpha^\T)(\lambda).
  \]
This and a similar calculation show that maps~$\gamma$
and~$\delta+\alpha^\T$ are maps of $B$-bimodules. 

That $\delta+\alpha^T:\D B\to\D B$ be a map of $B$-bimodules implies that
so is its transpose $\delta^\T+\alpha:B\to B$, and therefore the element
$z\coloneqq\delta^\T(1)+\alpha(1)$ is central in~$B$ and
$\delta^\T+\alpha=m_z$: we thus have that $\delta+\alpha^\T=m_z^\T$ and
that
  \(
  d''\coloneqq d'-\tilde z
    = \begin{psmallmatrix}0&\beta\\0&0\end{psmallmatrix}
    : A\to A
  \)
is a derivation. In particular, if $\lambda$ and~$\mu$ are two elements
of~$\D B$, then
  \[
  0 = d''(\lambda\mu)
    = d''(\lambda)\mu+\lambda d''(\mu)
    = \beta(\lambda)\mu+\lambda\beta(\mu),
  \]
and with this, together with the $B^e$-linearity of~$\beta$, we see that
$\beta$ belongs to~$\Alt(B)$. Putting everything together, we have that
$d=\tilde\alpha+\tilde\beta+\tilde\gamma+\tilde z$ and can therefore
conclude that $\Der(A)=\Der(B)+\Z(B)+\Der(B,\D B)+\Alt(B)$.

This argument shows moreover that if $\alpha\in\Der(B)$, $\beta\in\Alt(B)$,
$\gamma\in\Der(B,\D B)$ and~$z\in\Z(B)$ are such that
$\tilde\alpha+\tilde\beta+\tilde\gamma+\tilde z=0$, then necessarily
$\alpha=0$, $\beta=0$ and~$\gamma=0$, and thus also $z=0$. This proves that
the decomposition~\eqref{eq:derta} holds.

If $x\in B$ and~$\lambda\in\D B$, then the inner derivation of~$A$
corresponding to~$x+\lambda$ is 
  \[
  \ad_A(x+\lambda) = \widetilde{\ad_B(x)} + \widetilde{\ad_{\D B}(\lambda)},
  \]
and the equality in~\eqref{eq:derta:inn} follows form this, and thus also
the isomorphism in~\eqref{eq:derta:hh}.

Let $P$ be a projective resolution of~$B$ as a $B$-bimodule. The adjunction
isomorphism $\Hom_{B^e}(P,\D B)\cong\Hom(B\otimes_{B^e}P,\kk)$ induces upon
taking cohomology an isomorphism $\HH^*(B,\D B)\cong\Hom(\HH_*(B),\kk)$,
and the isomorphism mentioned in the lemma is the one we get by restricting
our attention here to degree~$1$. We can make this explicit: if $d:B\to \D
B$ is a derivation, the map $a\otimes b\in B\otimes B\mapsto d(b)(a)\in
\kk$ is zero on elements of~$B\otimes B$ that are of the form $ab\otimes
c-a\otimes bc+ca\otimes b$ with $a$,~$b$,~$c\in B$, that is, on Hochschild
$1$-boundaries, so restricting it to Hochschild $1$-cycles and passing to
the quotient it gives a linear map $\hat d:\HH_1(B)\to\kk$, and this map
$\hat d$ is zero if the derivation~$d$ is inner.
\end{proof}

With this description of the derivations of~$A$ at hand it is an easy
matter to compute their divergences.

\begin{Lemma}\label{lemma:div:trivial}
Let $\delta:A\to A$ be a derivation, and let $\alpha\in\Der(B)$,
$\beta\in\Alt(B)$, $\gamma\in\Der(B,\D B)$ and~$z\in\Z(A)$ be such that
  \[
  \delta = \begin{pmatrix}
           \alpha & \beta \\
           \gamma & -\alpha^\T+m_z^\T
           \end{pmatrix}
           : B\oplus\D B\to B\oplus\D B.
  \]
The Nakayama divergence of~$\delta$ is
  \[
  \div_\sigma(\delta) = z + \tau,
  \]
with $\tau:B\to\kk$ the linear map such that $\tau(x)=\gamma(x)(1)$ for all
$x\in B$.
\end{Lemma}

This is, naturally, very similar to what we found in
Subsubsection~\ref{subsubsect:jac:trivial} about the Nakayama Jacobians of
the automorphism of~$A$. We leave it as an exercise for the reader to
describe the set of elements of~$A$ that are divergences of derivations
of~$A$, in the same spirit of what we did there.

\begin{proof}
According to Lemma~\ref{lemma:div:ids}, to prove the lemma it is enough
that we check that $\gen{\delta(x),1}=\gen{x,z+\tau}$ for all $x\in A$, and
that is a simple calculation.
\end{proof}

Since the trivial extension~$A$ is a symmetric algebra, from the Nakayama
divergence map we obtain a cohomology class~$\DIV(A)$
in~$\H^1(\HH^1(A),\Z(A))$ as in Lemma~\ref{lemma:DIV}.

\begin{Lemma}
The class~$\DIV(A)$ in~$\H^1(\HH^1(A),\Z(A))$ is not trivial.
\end{Lemma}

\begin{proof}
Because of symmetry what we have is a linear map
$\div_\sigma:\Der(A)\to\Z(A)$ that is a Chevalley--Eilenberg $1$-cocycle on
the Lie algebra~$\Der(A)$ with values in~$\Z(A)$. In view of the result of
Lemma~\ref{lemma:div:trivial}, the restriction of that $1$-cocycle to the
Abelian subalgebra~$\Z(B)$ of~$\Der(A)$ is simply the inclusion
$z\in\Z(B)\mapsto z\in\Z(A)$, so the restriction of $\DIV(A)$ to a class in
the Lie algebra cohomology~$\H^1(\Z(B),\Z(A))$ of~$\Z(B)$ with values
in~$\Z(A)$ is a non-zero class. Indeed, as the (Lie!) action of~$\Z(B)$
on~$\Z(A)$ is trivial, that first cohomology space identifies canonically
with~$\Hom(\Z(B),\Z(A))$, and the inclusion is a non-zero element of this.
\end{proof}

\subsubsection{An example: the quantum complete intersection of
dimension \texorpdfstring{$4$}{four}}
\label{subsubsect:div:quantum}

Finally, let us compute the divergences of the derivations of the
four-dimensional algebra of Subsubsection~\ref{subsubsect:jac:quantum}. As
before, then, we let $q$ be a non-zero scalar and $A$ the algebra freely
generated by two letters~$x$ and~$y$ subject to the relations $x^2=0$,
$y^2=0$, and $yx=qxy$.

\begin{Lemma}
Let us suppose that the characteristic of the field~$\kk$ is not~$2$ and
that $q^2\neq1$.
\begin{thmlist}

\item For each choice of four scalars $a$,~$b$,~$c$ and~$d$ in~$\kk$ there
is exactly one derivation $\delta_{a,b,c,d}:A\to A$ such that
$\delta_{a,b,c,d}(x)=ax+cxy$ and $\delta_{a,b,c,d}(y)=by+dxy$, and it is an
inner derivation exactly when $a=0$ and~$b=0$. 

\item The Lie algebra~$\HH^1(A)$ is abelian of dimension~$2$, and it is
spanned by the cohomology classes of the derivations~$\delta_{1,0,0,0}$
and~$\delta_{0,1,0,0}$.

\end{thmlist}
\end{Lemma}

\begin{proof}
A direct calculation shows that for each choice of~$a$,~$b$,~$c$ and~$d$
in~$\kk$ there is a unique derivation~$\delta_{a,b,c,d}$ that satisfies the
two conditions that appear in the statement of the lemma. On the other
hand, if $a$,~$b$,~$c$ and~$d$ are scalars, the inner
derivation~$\ad(a+bx+cy+dxy)$ maps~$x$ and~$y$ to~$(q-1)cxy$ and
to~$(1-q)bxy$, respectively: since $q\neq1$, this implies the
characterization of inner derivations in the lemma.

Let now $\delta:A\to A$ be a derivation of~$A$. There are scalars
$a_0$,~$a$,~$a'$, $b_0$,~$b$,~$b'$, $c$ and~$d$ such that
$\delta(x)=a_0+ax+a'y+cxy$ and $\delta(y)=b_0+b'x+by+dxy$. Since $x^2=0$
and~$y^2=0$, we have that $\delta(x)x+x\delta(x)=0$
and~$\delta(y)y+y\delta(y)=0$, and these equalities imply, since $2\neq0$
and $q\neq-1$, that the scalars $a_0$,~$a'$, $b_0$ and~$b'$ are all zero so
that $\delta=\delta_{a,b,c,d}$.

With this we have proved all the claims in part~\thmitem{1} of the lemma,
and the ones in part~\thmitem{2} are now all clear.
\end{proof}

Knowing the derivations of~$A$ we can compute their divergences.

\begin{Lemma}\label{lemma:div:quantum}
Let us suppose that the characteristic of the field~$\kk$ is not~$2$ and
that $q^2\neq1$, and let $a$,~$b$,~$c$ and~$d$ be four scalars in~$\kk$.
The divergence of the derivation~$\delta_{a,b,c,d}:A\to A$ described in the
previous lemma is
  \[
  \div_\sigma(\delta_{a,b,c,d}) = (a+b)+q^{-1}dx+cy.
  \]
\end{Lemma}

This is similar to the formula for the Jacobians of the automorphisms
of~$A$ that we found in Lemma~\ref{lemma:jac:quantum}, and shows, just as
that lemma, that the divergence of a derivation of a non-symmetry algebra
need not be central.

\begin{proof}
To check this we need to show that
$\gen{\delta(z),1}=\gen{z,(a+b)+q^{-1}dx+cy}$ for each element~$z$
of~$\{1,x,y,xy\}$, and that can be done by an easy direct calculation.
\end{proof}

\subsection{A Liouville formula for Frobenius algebras}
\label{subsect:div:liouville}

Our formalism of Jacobians and divergences for Frobenius algebras has a lot
of points in common with the classical setup of vector calculus. We want to
give here as an example a proof of an analogue of the classical formula of
Liouville for the Jacobian of the flow of a vector field on a compact
Riemannian manifold.

\bigskip

We go back to the general situation in which we have a Frobenius
algebra~$A$ with respect to a non-degenerate and associative bilinear form
$\gen{\place,\place}:A\times A\to\kk$, with corresponding Nakayama
automorphism $\sigma:A\to A$. We start with the following observation.

\begin{Lemma}\label{lemma:delta-powers}
Let $\delta:A\to A$ be a derivation, and let $(\phi_k)_{k\geq0}$ be the
sequence of elements of~$A$ that has $\phi_0=1$ and 
  \(
  \phi_{n+1} = \phi_n\cdot\div_\sigma(\delta) - \delta(\phi_n)
  \)
for each $n\in\NN_0$. For each $n\in\NN_0$ and each choice of~$a$ and~$b$
in~$A$ we have that
  \[
  \sum_{k=0}^n\binom{n}{k}\cdot\gen{\delta^i(a),\delta^{n-i}(b)}
        = \gen{a, b\cdot\phi_n}.
  \]
If $k$ is a non-negative integer such that $\delta^k=0$, then $\phi_k=0$.
\end{Lemma}

\begin{proof}
The first claim can be proved by an obvious induction, and from it we can
see at once that $\gen{\delta^k(a),1}=\gen{a,\phi_k}$ for all $k\in\NN_0$
and all $a\in A$, so that $\phi_k=0$ whenever $k$ is a non-negative integer
such that $\delta^k=0$.
\end{proof}

Let us suppose that the characteristic of the ground field~$\kk$ is zero.
If $\delta:A\to A$ is a nilpotent derivation of~$A$, then we can consider
for each $t\in\kk$ the map
  \[
  \exp(t\delta) \coloneqq \sum_{k=0}^\infty \frac{t^k}{k!}\delta^k : A\to A,
  \]
since the hypothesis on~$\delta$ implies that this sum is finite. This is
an automorphism of~$A$ and our Liouville formula computes its Jacobian in
terms if the divergence of~$\delta$ --- at least «infinitesimally».

\begin{Lemma}\label{lemma:Liouville}
Let $\delta:A\to A$ be a nilpotent derivation, let $(\phi_k)_{k\geq0}$ be
the sequence defined in Lemma~\ref{lemma:delta-powers}, and for each
$t\in\kk$ let $\exp(t\delta):A\to A$ be the automorphism of~$A$ described
above. The formal series
  \[
  \Phi \coloneqq \sum_{k=0}^\infty \phi_k\frac{t^k}{k!} \in A[[t]]
  \]
is the unique solution to the formal initial value problem
  \[ \label{eq:li:0}
  \left\{
  \begin{lgathered}
  \Phi'(t) + \delta(\Phi(t)) = \Phi(t)\cdot\div_\sigma(\delta), \\
  \Phi(0) = 1.
  \end{lgathered}
  \right.
  \]
We have that
  \[ \label{eq:li:1}
  \jac_\sigma(\exp(t\delta)) = \sigma^{-1}(\Phi(t))
  \]
and, in particular, that
  \[ \label{eq:li:2}
  \frac{\d}{\d t}\biggm\mid_{t=0}\jac_\sigma(\exp(t\delta))
        = \sigma^{-1}(\div_\sigma(\delta)).
  \]
\end{Lemma}

Since the derivation~$\delta$ is supposed to be nilpotent, we know that
$\phi_k=0$ when $k\geq\dim A$, and thus the formal series~$\Phi$ is really
a polynomial. What we mean by the equality~\eqref{eq:li:1} is that for each
$t\in\kk$ the Jacobian $\jac_\sigma(\exp(t\delta))$ is the value of that
polynomial at~$t$, and by the equality~\eqref{eq:li:2} that the coefficient
of~$t$ in that polynomial is precisely~$\sigma^{-1}(\div_\sigma(\delta))$.

The left hand side of the differential equation in~\eqref{eq:li:0} can be
viewed as the so-called \emph{material derivative} of~$\Phi$ in the
direction of the flow generated by~$\delta$, in the sense of continuum
mechanics.

\begin{proof}
We have that $\gen{\delta^k(a),1} = \gen{a,\phi_k}$ for all $k\in\NN_0$ and
all $a\in A$, so for all $t\in\kk$ we have that
  \[
  \gen{\exp(t\delta)(a),1}
        = \sum_{k=0}^\infty\frac{t^k}{k!}\gen{\delta^k(a),1}
        = \sum_{k=0}^\infty\frac{t^k}{k!}\gen{a,\phi_k}
        = \gen{a, \Phi(t)}
        = \gen{\sigma^{-1}(\Phi(t)),a}
  \]
This means that the equality~\eqref{eq:li:1} is true, and now the
equality~\eqref{eq:li:2} follows from that, since the definition of the
sequence~$(\phi_k)_{k\geq0}$ implies that $\phi_1=\div_\sigma(\delta)$.
\end{proof}

The claims of the lemma are true also in the situation in which the ground
field~$\kk$ has characteristic zero and is complete with respect to some
absolute value. In that case the algebra~$A$, being finite dimensional, can
be endowed with a norm with respect to which it is complete and we can make
sense of the exponential~$\exp(t\delta)$ for all derivations~$\delta$
of~$A$, not only the nilpotent ones, provided we restrict $t$ to lie in a
sufficiently small open neighborhood of~$0$ in~$\kk$. The initial value
problem~\eqref{eq:li:0} then has a unique solution and the
equalities~\eqref{eq:li:1} and~\eqref{eq:li:2} hold for sufficiently small
values of~$t$.

It follows from this, for example, that if we are working over~$\CC$, $A$
is symmetric and $\HH^1(A)=0$, then the Jacobian of any automorphism of~$A$
belonging to the connected component of the identity in~$\Aut(A)$ is~$1$.
Indeed, in that case Lemma~\ref{lemma:div:inner} tells us that all
derivations have zero divergence, so that the series~$\Phi$ in the lemma is
constant, and the claim follows at once from~\eqref{eq:li:1} and the fact
that $\jac_\sigma$ is a $1$-cocycle and that exponentials of derivations
generate the connected component of the identity in this situation.

\bibliography{references}

\end{document}